\documentclass{ip-journal}
\usepackage{psfrag}
\usepackage{graphicx}
\usepackage{pinlabel}
\usepackage{graphicx}
\usepackage{amsmath}
\usepackage{amssymb}
\usepackage{amscd}
\usepackage{picinpar}
\usepackage{times}
\usepackage{pb-diagram}
\usepackage{graphicx}
\usepackage{wrapfig}
\usepackage{xspace}
\usepackage{hyperref}
\usepackage{url}
\usepackage{xcolor}

\def\C{\mathbb{C}}
\def\R{\mathbb{R}}
\def\H{\mathbf{H}}
\def\T{\mathcal{T}}

\def\D{\mathbb{D}}
\def\SS{\mathbf{S}}
\def\B{\mathbb{B}}

\newtheorem{theorem}{Theorem}[section]
\newtheorem{lemma}[theorem]{Lemma}
\newtheorem{proposition}[theorem]{Proposition}
\newtheorem{corollary}[theorem]{Corollary}
\newtheorem{definition}[theorem]{Definition}

\newtheorem{remark}[theorem]{Remark}

\newtheorem*{theorem*}{Theorem}

\begin{document}

\title[Circle domain]{The Koebe conjecture and the Weyl problem for convex surfaces in hyperbolic 3-space}
\author{Feng Luo}

\address{Department of Mathematics, Rutgers University, Piscataway, NJ, 08854}

\email{fluo@math.rutgers.edu}

\author{Tianqi Wu}

\address{Center of Mathematical Sciences and Applications, Harvard University, Cambridge, MA 02138
}

\email{mike890505@gmail.com}         


\begin{abstract}
We prove that the Koebe circle domain conjecture is equivalent to the Weyl type problem  that every complete hyperbolic surface of genus zero is isometric to the boundary of the hyperbolic convex hull of the complement of a circle domain in the hyperbolic 3-space.
Applications of the result to discrete conformal geometry will be discussed. 
The main tool we use is Schramm's transboundary extremal lengths.

\end{abstract}
\maketitle
\tableofcontents
\newtheorem*{equicont}{Theorem \ref{equicont}}


\section{Introduction}
\subsection{The main result}


A \it circle domain \rm is an open connected set in the Riemann sphere $\hat{\C}$ whose boundary components are either circles or points.
In 1908,  P. Koebe  \cite{k1} made the circle domain conjecture that any domain (i.e., open connected set) in the plane is conformally diffeomorphic to a circle domain. 
The classical Weyl problem concerns isometric embeddings of positively curved 2-spheres into  3-space. The Weyl problem in the hyperbolic space asks for isometric embeddings of genus-zero surfaces with complete metrics of curvature at least $-1$ into hyperbolic 3-space $\H^3$. This paper shows that the Koebe conjecture is equivalent to a special form of the Weyl problem in $\H^3$.


 Let  $\SS^2$ be the unit 2-sphere, $\hat\C$ be identified with $\SS^2$ by the stereographic projection, ($\H^3, d^P)$ or $\H_P^3$  be the Poincar\'e ball model of the hyperbolic 3-space whose boundary is $\SS^2$, and $C_P(Y)$ be the convex hull of a closed set $Y$ in $\H^3 \cup \SS^2$ in  hyperbolic 3-space.  A \it circle type \em closed set $Y \subset \SS^2$ or $Y\subset\hat \C$ is a compact set whose complement is a circle domain, i.e., each connected component of $Y$ is either a round disk or a point.
It is well known by the work of W. Thurston (see page 185 in \cite{th} or Theorem 1.12.1 in \cite{em}) that if $Y$ is a closed set in $\SS^2$ containing more than two points, then the boundary of the convex hull $\partial C_P(Y) \subset \H^3$, in the intrinsic path metric, is a genus-zero complete hyperbolic surface. The special form of the Weyl problem, which is an inverse version of Thurston's theorem,   is:


\medskip
\noindent
{\bf Conjecture 1.} (\cite{lsw}, see also \cite{luo8}) \it
Every genus zero complete hyperbolic surface is isometric to $\partial C_P(Y)$ for a circle type closed set $Y$ in $\SS^2$.   \rm
\medskip

We remark that in the case that the convex hull $C_P(Y)$ is 2-dimensional, as a convention in this paper,   we use $\partial C_P(Y)$ to denote the metric double of $C_P(Y)$ across its boundary, i.e., $\partial C_P(Y) := C_P(Y) \cup_{ id_{\partial}} C_P(Y)$.


Since the circle domain conjecture can be proved easily for domains whose complements in $\hat{\C}$ contain at most two points, we will consider in the rest of the paper only those domains $U$ whose boundaries $\partial U$  contain more than two points. For each such domain $U$, by the uniformization theorem,  there exists a unique complete conformal hyperbolic metric $d^U$ of the form $\lambda(z) |dz|$ on $U$.  The metric $d^U$ will be called the Poincar\'e metric of the domain $U$. Using Koebe's theorem that any genus-zero Riemann surface is conformally diffeomorphic to a domain in the Riemann sphere, we see that every genus-zero complete hyperbolic surface is isometric to a Poincar\'e metric  $(U, d^U)$ where $U \subset \mathbf S^2$ with $|\partial U| \geq 3$. In particular, the circle domain conjecture is equivalent to the statement that every genus-zero complete hyperbolic surface is isometric to $(U, d^U)$ for some circle domain $U$.

The main result of the paper shows that the circle domain conjecture of Koebe is equivalent to Conjecture 1.
More, precisely,  we prove:


\begin{theorem}\label{1.222} (a) For any circle domain $U \subset \hat \C$ whose boundary contains at least three points, there exists a circle type closed set $Y \subset \mathbf S^2$ such that the Poincar\'e metric  $(U, d^U)$ is isometric to $\partial C_P(Y)$.

(b) For any circle type closed set $Y \subset \mathbf S^2$ with $|Y| \geq 3$, there exists a circle domain $U \subset \hat \C$ such that $\partial C_P(Y)$ is isometric to the Poincar\'e metric $(U, d^U)$.
\end{theorem}

Using  He-Schramm's theorem  \cite{hes} that the Koebe conjecture holds for domains with countably many boundary components, we obtain:

\begin{corollary}\label{1.22} Every genus zero complete hyperbolic surface with countably many topological ends is isometric to  $\partial C_P(Y)$ for a circle type closed set $Y$ in $\SS^2$.
\end{corollary}

The relationship between the Koebe conjecture and Conjecture 1 was discovered during our investigation of the discrete uniformization conjecture for polyhedral surfaces. Indeed, Conjecture 1 can be considered a generalized version of the existence part of the discrete uniformization conjecture.
The above corollary implies that every non-compact simply connected polyhedral surface is discrete conformal to the complex plane $\C$ or the unit disk $\mathbb D$. For more details, see \S10 and also \cite{glsw}, \cite{gglsw},  \cite{lsw} and \cite{luo8}.


The main tool we use to show Theorem \ref{1.222} is Schramm's transboundary extremal lengths.

\subsection{History and a generalized Weyl problem in $\H^3$}
The Koebe conjecture is known to be true for connected open sets
$U \subset \hat{\C}$
which have finitely many boundary components (\cite{k1}).
The best work done to date is by He-Schramm \cite{hes}, where they proved the conjecture for $U$ having countably many boundary components.
Conjecture 1 is known to be true for finite area hyperbolic surfaces and hyperbolic surfaces of finite topological types whose ends are funnels by the works of Rivin \cite{rivin2} and Schlenker \cite{sch1}, respectively. F. Fillastre \cite{fillastre} proved that Conjecture 1 holds for many symmetric domains with countably many boundary components.

Conjecture 1 is a Weyl-type problem for convex surfaces.  It is well known that a smooth convex surface in the Euclidean space (respectively the hyperbolic space) has Gaussian curvature at least zero (respectively $-1$).  Weyl's problem asks the converse. Namely, whether any Riemannian metric of positive curvature on the 2-sphere is isometric to the boundary of a convex body in  3-space.  The problem was solved affirmatively by Levy, Nirenberg and Alexandrov. 
 The natural generalization of Weyl's problem to hyperbolic 3-space states that every complete Riemannian metric of curvature at least $-1$ on a genus zero surface is isometric to the boundary of a closed convex set in hyperbolic 3-space.  This was established by Alexandrov in \cite{al}.  

Take a simply connected domain
$U \subset \C$ with $U \neq \C$ and let $Y =\hat{\C}-U$ be its complement. Then, in the upper-half space model of the hyperbolic 3-space, the boundary surface $\partial C_P(Y)$ is
homeomorphic to $U$ and hence is simply connected (see Corollary \ref{3212}). 
Thurston's theorem says that there exists an isometry $\Phi$ from $\partial C_P(Y)$ to $\partial C_P(\mathbb D^c) \cong \H^2$, where $\mathbb D$ is the unit disk. On the other hand, the Riemann mapping theorem says that there exists a conformal diffeomorphism $\phi$ from $U$ to $\mathbb D$. Thus Thurston's isometry $\Phi$ can be considered as a geometric realization of the Riemann mapping $\phi$.
The Koebe conjecture and Conjecture 1 are the corresponding Riemann mapping-Thurston's isometry picture for non-simply connected domains.


The uniqueness aspect of Conjecture 1 is the following statement.

\medskip
\noindent {\bf Conjecture 2. (\cite{lsw})} \it Suppose $X$ and $Y$ are two circle type closed sets in $\mathbf S^2$ such that $\partial C_P(X)$ is isometric to $\partial C_P(Y)$. Then $X$ and $Y$ differ by a M\"obius transformation.\rm
\medskip

Though it is known that the uniqueness part of the Koebe circle domain conjecture is false, it is possible that Conjecture 2 may still be true in view of Pogorelov's rigidity theorem \cite{pogo} (see also Theorem 1 in \cite{bori}). Since the uniqueness part of the Koebe conjecture holds for domains with countably many ends by \cite{hes}, we believe Conjecture 2 holds for $X$ with countably many connected components.

The strongest version of the Weyl problem in hyperbolic 3-space is the following.

\medskip
\noindent {\bf Conjecture 3.} \it Suppose $(S, d)$ is a planar surface with a complete path metric whose curvature is at least $-1$.  Then there exists a complete convex surface $Y$ in $\H^3$ isometric to $(S, d)$ such that each end of $Y$ is either a circle or a point in the sphere at infinity of $\H^3$. Furthermore, the convex surface $Y$ is unique up to isometry of $\H^3$.\rm
\medskip

\subsection{Organization of the paper and acknowledgement}

The paper is organized as follows. Preliminaries and an outline of the proof of the main theorem are in \S2. In \S3, we give a proof of the main theorem for a special case which has to be dealt with separately.
In \S4, we prove an area estimate theorem for convex surfaces in  hyperbolic 3-space and a few results on the shortest distance projection maps. In \S5, we establish some results relating to Hausdorff convergence and the convergence of the Poincar\'e metrics. In \S6, we recall Schramm's transboundary extremal length and establish a duality theorem. Part (a) of Theorem \ref{1.222} is proved in \S7.  Part (b) of Theorem \ref{1.222} is proved in \S8, assuming the key equicontinuity result, which is proved in \S9.  In \S10, we briefly discuss the relationship between the Weyl problem, the discrete conformal geometry of polyhedral surfaces, and the discrete uniformization problem.  In the Appendix, we recall the work of  Reshetnyak \cite{res} on the complex structure of non-smooth convex surfaces, which is used in the paper.

We thank Michael Freedman and Francis Bonahon for stimulating discussions.  Part of the work was carried out while the first author was visiting CMSA at Harvard. We thank S.T. Yau for the invitation.  We greatly appreciate the referee's meticulous reading of the paper
and his/her detailed comments and suggestions which helped us to improve considerably the manuscript.
The work is supported in part by NSF 1760527, NSF 1737876, NSF 1811878, NSF 1405106, NSF 1760471 and NSF 2220271.

\section{Notations, Preliminaries,  and Outline of the proof of Theorem \ref{1.222}}
In this section, we recall some of the basic facts on convex surfaces, surfaces of bounded curvature, and their conformal structures. We will outline the main steps in the proof of 
 Theorem \ref{1.222}. 

 The strategy of proving Theorem \ref{1.222} is to approximate an arbitrary circle domain by circle domains of finite topology.
  The key result that enables us to show that the limiting domain is still a circle domain is the equicontinuity of the family of approximation conformal maps, i.e., Theorems \ref{equicont} and \ref{equicont1}.  To prove the equicontinuity, we use Schramm's transboundary extremal length \cite{schramm} and the duality Theorem \ref{4373} of transboundary extremal lengths on annuli.  The main result for estimating extremal lengths is Theorem \ref{area} which states that the Euclidean area of any convex hyperbolic surface in the Poincar\'e model of the hyperbolic space $\H^3_P$ is at most $16\pi$.

\medskip
\subsection{Notations and Preliminaries}
We use $\mathbf S^2$, $\C$,  $\D$, $\H^3_P=(\H^3, d^P)$ and $\H^3_K=(\H^3, d^K)$ to denote the standard 2-sphere, the complex plane, the open unit disk, the Poincar\'e model of the hyperbolic 3-space, and the Klein model of the hyperbolic 3-space respectively.
We use $\H^2$ to denote the Poincar\'e disk model and consider $\H^2 \subset \H^3$ as a totally geodesic plane. The closure of a set $X \subset \R^3$ is denoted by $\overline{X}$.
The boundary of a 3-dimensional convex set $X$ in $\H^3$ or $\R^3$ is denoted by $\partial X$. 
The spherical metric on $\mathbf S^2$ or $\hat \C$ is denoted by $d^{\mathbf S}$,  and the Euclidean metric on $\mathbb R^3$ is denoted by $d^E$.  
The Poincar\'e metric on a domain $U \subset \hat \C$ with $|\partial U| \geq 3$ is denoted by $d^U$. The interior of a manifold or convex set $M$ is denoted by $int(M)$. The open ball of radius $r$ centered at $p$ in a metric space $(Z,d)$ is $B_r(p, d)$, and the diameter of $X \subset Z$ is $diam_d(X)$.  

 A metric space $(X,d)$ is called a \it path metric space \rm if the distance between two points $p$ and $q$ is the infimum of the lengths of paths between $p$ and $q$. We call $d$ a path metric. For instance, each Riemannian manifold is a path metric space.  Suppose $(X, d)$ is a path connected metric space such that every pair of points in $X$ can be joined by a rectifiable path, we can define the induced path metric $d^*$ on $X$ to be the infimum of the lengths of paths between two points, i.e., $d^*(p,q) =\inf\{l_d(\gamma):  \text{  $\gamma$ is a rectifiable path from $p$ to $q$}\}$. If $X$ is a rectifiably path connected subset of a metric space $(Z,d^Z)$, we use $d_X^Z$ to denote the induced path metric on $X$. We remark that, in general, the induced path metrics $d_X^Z$ and $d^*$ are only pseudo metrics. However, in the cases we encounter in this paper, the induced path metrics are true metrics. See \cite{burago} chapter 2 for more details on path metrics.

We will  use  $d=\sum_{i=1}^n g_{ij} dx_idx_j$ and sometimes $a(z)|dz|$ for a Riemannian metric on a manifold $M$. The Riemannian distance $d_M$ on $M$ is defined by the infimum of the lengths of smooth paths between two points. It is known (Proposition 2.4.1, \cite{burago}) that if $X$ is a connected smooth submanifold of a Riemannian manifold $(M, d)$ with associated Riemannian distance $d_M$, then the induced path metric on $X$ from $d_M$ is the same as the Riemannian distance metric on $X$ where the Riemannian metric on $X$ is obtained by restricting the Riemannian metric $d$ to the tangent spaces of $X$.   To avoid excessive use of notations and if no confusion arises, we will use $d$ to denote both the Riemannian metric on $M$ and the associated Riemannian distance $d_M$. 

\medskip
\subsection{Basic facts about convex surfaces and shortest distance projections}
In this subsection, we collect some of the known results on convex sets and surfaces.
For a convex set $X$ in $\R^3$ or $\H^3$, the dimension of $X$ is the dimension of the smallest totally geodesic submanifold $P$ which contains $X$. In particular, a convex set $X$ has a non-empty interior in the submanifold $P$.  We will mainly deal with 3-dimensional convex sets. A convex surface $S$ is a topological surface such that $S =\partial X$ for some 3-dimensional convex set in $\R^3$ or $\H^3$. Note that we do not require $X$ to be closed. 

The topology of convex sets and convex surfaces in hyperbolic 3-space can be understood using the Klein model. Recall that geodesics and totally geodesic planes in the Klein model $\H^3_K$ are exactly the intersections of Euclidean lines and Euclidean planes with the open unit ball $\B^3$. Therefore, convex surfaces, and convex sets in $\H_K^3$ are the same (as point sets) as convex surfaces, and convex sets in the open unit ball $\B^3$ in Euclidean geometry. This shows that all topological properties of hyperbolic convex sets and convex surfaces are the same as those of Euclidean convex sets and convex surfaces in $\B^3$. 

The fundamental topological properties of a compact 3-dimensional convex set $X$ in $\R^3$ are that  $X$ is homeomorphic to the closed unit 3-ball and its boundary $\partial X$ is homeomorphic to the unit sphere $\SS^2$.  Indeed, take a point $p$ in the interior of $X$. 
Then the restriction of the radial projection map $(x-p)/|x-p|$ (from the point $p$) to the boundary of $X$ is a homeomorphism $\psi$ from $\partial X$ to $\SS^2$. It follows that the map $F(x) =|x|\psi^{-1}(x/|x|)$ from the closed unit ball $\{x \in \R^3| |x|\leq 1\}$ to $X$  is a homeomorphism. Therefore, each convex surface in $\R^3$ or $\H^3$ is homeomorphic to an open subset of $\SS^2$ and hence has genus zero. Furthermore, if $W$ is a compact subset of $\SS^2$ which does not lie in a Euclidean plane, then the hyperbolic convex hull (union with $W$) $C_K(W) \cup W$ is a topological closed ball since,  as point set,  $C_K(W)\cup W$ is the Euclidean convex hull of $W$.

The radial projection map constructed above implies the following stronger result. 
\begin{lemma} \label{radialprj} Suppose $X$ and $Y$ are n-dimensional compact convex sets such that $X \subset Y$. Then the radial projection from an interior point of $X$ induces a homeomorphism $h$ from $\partial X$ to $\partial Y$. Furthermore,  $h$ is the identity map when restricted to $\partial X \cap \partial Y$ and sends $\partial X -\partial Y$ homeomorphically to $\partial Y -\partial X$. 
\end{lemma}

\begin{corollary}(Theorem II.1.4.3 (5) in \cite{em}\label{3212}) If $A$ is a compact subset of the unit sphere $\SS^2$ such that its convex hull $X$ in $\R^3$ is 3-dimensional,  then $\partial X -A$ is homeomorphic to $\SS^2-A$ by  the radial projection from an interior point of $X$. 
\end{corollary}


One of the important tools in convex geometry is the shortest distance projection (see page 9 of \cite{schneider}).  The shortest distance projection map $\pi$ from a Euclidean space to a closed convex subset  $X$ sends each point $p \in \mathbf R^n$ to the unique point $\pi(p) \in X$ which is the point in $X$ closest to $p$, i.e., $|p -\pi(p)| =\min\{ |p-x|| x \in X\}$.  Geometrically, the projection point $\pi(p)$  is the intersection of $X$ with the largest closed ball centered at $p$ whose interior is disjoint from $X$. 

\begin{lemma} (pages 9-10 in \cite{schneider}, or page 201 in \cite{bertsekas}) Suppose $X$ is a closed convex subset of $\R^n$ and 
 $\pi: \mathbf R^n \to X$ is the shortest distance projection. Then the following statements hold.

(a) $\pi$ is distance decreasing, i.e., $ |\pi(p)-\pi(q)| \leq |p-q|$.

(b) If $p \notin X$ and $q \in X$, then the angle $\angle p \pi(p) q$ at $\pi(p)$ is at least $\pi/2$.

(c) If $X$ is an n-dimensional compact convex set and $Y$ is a compact convex set containing $X$, then the restriction map $\pi|_{\partial Y}: \partial Y \to \partial X$ is onto.  
\end{lemma} 

Similar properties for the shortest distance projection in hyperbolic spaces hold. These will be discussed in detail in \S4.2.
\medskip

\subsection{Basic facts about area and conformal structures on surfaces of bounded curvature}
This paper deals with the geometry of the boundary of the convex hulls, i.e., convex surfaces,  in hyperbolic spaces. The study is complicated by the fact that most of the convex surfaces we use are not smooth.
However, these surfaces have been extensively investigated in Alexandrov geometry.  A convex surface $S$ in $\H^3_P$ or $\H^3_K$ carries two natural structures: an induced path metric and a conformal structure.  The metric structure on $S$ is the induced path metric $d_S^P$ (or $d^K_S)$ on $S$ derived from the hyperbolic metric $d^P$ (or $d^K$). Unlike the restriction metric $d^P|_S$ which is extrinsic, the path metric $d^P_S$ defines the intrinsic geometry of $S$.  One of the basic properties of the path metric $d^P_S$ is that it defines the same topology as $d^P$ does.  In fact, a stronger result holds. Namely, the restriction metric $d^P|_S$ is locally bi-Lipschitz with respect to the path metric $d^P_S$ (see Lemma II 1.5.7 in \cite{em}).  We will also consider the induced path metric $d^E_S$ on $S$ from the Euclidean metric $d^E$.  This new metric $d^E_S$ has not been studied extensively in the literature. 
The conformal structure on $S$ comes from the induced path metric $d^P_S$. 
In general, a surface with a path metric is not known to define a complex structure. The work of Reshtnyak \cite{res}, \cite{troyanov} shows that if a surface with a path metric $(S, d)$ is of bounded curvature, then the path metric $d$ defines a complex structure.  There are several equivalent definitions of surfaces of bounded curvature.  See, for instance, page 12 of \cite{troyanov} or page 6 of \cite{alexandrov2}. Basically, each surface of bounded curvature can be approximated locally uniformly by a sequence of polyhedral surfaces with bounded curvature. Equivalently (see Theorems 2.4 and 2.6 in \cite{troyanov}), a path metric surface $(S, d)$ is of bounded curvature if and only if there exists a sequence of Riemannian surfaces $(S, d_i)$ such that (1) $d_i$ converges to $d$ uniformly on compact subsets and (2) for any compact subsurface $X$ of $S$, the integrals of the absolute values of the Gaussian curvature of $d_i$ over $X$ are uniformly bounded. For instance all convex surfaces in $\R^3$ and $\H^3$ are of bounded curvature
(Theorem 2.7 in \cite{troyanov}). Another fact that we use is that if $(S, d^P_S)$ is a convex surface in $\H^3_P$, then $(S,d^E_S)$ is a surface of bounded curvature. This can be seen as follows. Take a sequence of smooth convex surfaces $(S_i, d^P_{S_i})$ in $\H^3_P$ approximating $(S, d^P_S)$ uniformly on compact subsets. Then $(S_i, d^E_{S_i})$ converges to $(S, d^E_{S})$ uniformly on compact subsets such that the integrals of the absolute values of the Gaussian curvature of $(S_i, d^E_{S_i})$ over compact sets are bounded. It is known that the 2-dimensional Hausdorff measure on a bounded curvature surface $(S, d)$ is equal to the area element which was constructed synthetically using geodesic triangles on $(S, d)$ by Alexandrov (page 262  \cite{alexandrov2} and proposition 1.3 of \cite{SAAJID}). The work of  Reshetynak puts surfaces of bounded curvature in the setting of Riemannian metrics by relaxing the regularity condition on Riemannian metrics. One of the main results of  Reshetynak  (Theorem 7.1.2 in \cite{res}, or Theorem 2.23 in \cite{troyanov}) says that if $(S, d)$ is a surface of bounded curvature, then at each point one can find a local coordinate chart $(U, z)$ and
a function $\lambda(z)$ which is the difference of two subharmonic functions such that
the path metric $d$ restricted to  $U$ coincides with the Riemannian distance associated to the Riemannian metric $e^{\lambda(z)}|dz|$.  Note that the Riemannian metric $e^{\lambda(z)}|dz|$ may not be continuous. These charts $(U, z)$ produce the complex structure on $(S, d)$.  In our case, we need to use the complex structure to compute the extremal lengths of families of curves on bounded curvature surfaces. It requires the notations of the length of curves and area of subsets, which were constructed by Alexandrov.  What  Reshetynak's theorem tells us is that these notions are the same as the ones used in the complex analysis for computing extremal lengths on surfaces of bounded curvature.  The last fact we use is that for a hyperbolic convex surface $S$, the conformal structures associated with the path metrics $d^P_S$ and $d^E_S$ are the same. This is due to (1) $d^P$ and $d^E$ are conformal metrics on $\H^3$ and (2) a theorem of  Reshetynak  (Theorem 7.3.1 in \cite{res}) on isothermal coordinates. The details are in the Appendix of this paper.

\subsection{Outline of the proof of Theorem \ref{1.222} (a)}
Suppose $U=\hat{\C}-X$ is a circle domain in  $\hat{\C}$ such that $X$ contains at least three points. Let $d^U$ be the Poincar\'e metric on $U$. The goal is to  find a circle type closed set $Y \subset \mathbf S^2$ such that $(U, d^U)$ is isometric to $\partial C_K(Y) \subset (\H^3, d^K)$. 

 Produce a sequence of circle domains $U_n=\hat{\C}-X^{(n)}$ with $\partial U_n$ consisting of finitely many circles  such that \{$X^{(n)}$\} converges to $X$ in Hausdorff distance.  More precisely, using a M\"obius transformation, we may assume   $X \subset \{ z \in \C| 2 < |z| < 3\}$.  Let $W_1, ..., W_i, ...$ be a sequence of distinct connected components of $X$ such that $\cup_{i=1}^{\infty} W_i$ is dense in $X$. Since $X$ is a circle type closed set,  each disk connected component of $X$ is in the sequence.
Let $X^{(n)}=\cup_{i=1}^n W_i^{(n)}$,  where  $W^{(n)}_i=W_i$ if $W_i$ is a disk, and $W^{(n)}_i=\{ z \in \C| d^E(z, W_i) \leq l_n\}$ if $W_i$ consists of a single point.  We make $l_n$ small such that $l_n$ decreases to $0$ and $W^{(n)}_i \cap W^{(n)}_j =\emptyset $ for $i \neq j$.  This construction ensures that  $X^{(n)}$ is a circle type closed set having finitely many connected components and
$X^{(n)}$ converges to $X$ in the Hausdorff distance in $\C$.

 For each $X^{(n)}$, by Schlenker's work \cite{sch1} (see also Theorem \ref{sh11}), we construct a circle type closed set $Y^{(n)}\subset \SS^2$ such that there exists an isometry $\phi_n: (U_n, d^{U_n}) \to \partial C_K(Y^{(n)})$.   Using a  M\"obius transformation,  we may assume that $\partial C_K(Y^{(n)})$ contains the origin $(0,0,0)$ and
 $\phi_n(0)=(0,0,0) \in \partial C_K(Y^{(n)})$. By taking a subsequence if necessary, we may assume that $Y^{(n)}$ converges in Hausdorff distance to a closed set $Y \subset \SS^2$.   We will show:

 (1) The sequence $\{\phi_n\}$ contains a subsequence converging uniformly on compact subsets to a continuous map $\phi: U \to \partial C_K(Y)$.   This is achieved by showing  $\{\phi_n: (U_n , d^{\mathbf S}) \to  (\partial C_K(Y^{(n)}), d^E)\}$ is an equicontinuous family. 
 The latter is proved in \S7  using transboundary extremal lengths. In computing the extremal length, we use the fact that the path metric induced by $d^E$ on a possibly non-smooth hyperbolic convex surface in $\H_P^3$ is conformal to its intrinsic path metric induced from $d^P$.

 (2) The limit map   $\phi: (U, d^U) \to \partial C_K(Y)$ is an isometry.
 This follows from  Alexandrov's convergence Theorem \ref{convergence} and  convergence of Poincar\'e metrics (Theorem \ref{poincare}) in \S5;

 (3)  The compact set  $Y$ is of circle type.  Since the Hausdorff limit of a sequence of round disks is a round disk or a point, we will prove in \S7 that each component of $Y$ is the Hausdorff limit of a sequence of components of $Y^{(n)}$'s. This is proved using the equicontinuity property established in step (1).

\subsection{Outline of the proof of Theorem \ref{1.222} (b)}
 Part (b) of Theorem \ref{1.222} states that for any circle type closed set $Y \subset \mathbf S^2$ with $|Y| \geq 3$, there exists a circle domain $U=\hat{\C}-X$ with Poincar\'e metric $d^U$ such that  $\Sigma :=\partial C_P(Y)$ is isometric to $(U, d^U)$.  The
strategy of the proof is the same as that for part (a) of Theorem \ref{1.222}. The only technical complication is due to the estimation of modules of rings in non-smooth convex surfaces (e.g., $\partial C_P(Y)$).

By Theorem \ref{2.1},  we may assume that the set $Y$ is not contained in any circle, i.e., $C_P(Y)$ is  3-dimensional.
By composing with a M\"obius transformation, we may assume that $(0,0,0) \in \Sigma $.  Since $Y$ is a circle type closed set, there exists a sequence $\{Y_n\}$ of components of $Y$ such that $\cup_{i=1}^{\infty}Y_i$ is dense in $Y$ and $(0,0,0) \in C_P(Y_1 \cup ... \cup Y_4)$.  Denote $Y^{(n)}=\cup_{i=1}^n Y_i$ and $\Sigma_n=\partial C_P(Y^{(n)})$.  By construction $(0,0,0) \in  \Sigma_n$ for $n \geq 4$, and the sequence $\{Y^{(n)}\}$ converges in Hausdorff distance to $Y$ in $\mathbf S^2$.

Now each $\Sigma_n$ is a genus zero Riemann surface of finite type. Koebe proved that any genus zero Riemann surface is conformally equivalent to an open domain in $\hat\C$ (see \cite{k1}) and any finitely connected domain is conformally equivalent to a circle domain (see \cite{k2} or page 234 Theorem 1 in \cite{goluzin}).
So there exists a circle domain $U_n =\hat{\C}-X^{(n)}$ and a conformal diffeomorphism $\phi_n:\Sigma_n \to U_n$. Using  M\"obius transformations, we normalize $U_n$ such that
$0\in U_n$, $\phi_n(0,0,0)=0$ and  the closed unit disk $\overline{\D}$ is contained in $U_n$.
 By taking a subsequence if necessary, we may assume that $X^{(n)}$ converges in Hausdorff distance to a compact set $X$ in the spherical metric $d^{\mathbf S}$.  We will show:

 (1) The sequence $\{\phi_n\}$ contains a  subsequence converging uniformly on compact subsets to a continuous map $\phi: \Sigma \to U:=\hat \C-X$.   This is achieved by showing  $\{\phi_n: (\Sigma_n , d^E) \to  (U_n, d^{\mathbf S})\}$ is equicontinuous. The latter is proved in  \S9 using transboundary extremal lengths;

 (2) The limit map   $\phi: \Sigma \to (U, d^U)$ is an isometry.  This is a consequence of
 Alexandrov's convergence theorem and the convergence of Poincar\'e metrics theorem in \S5;

 (3)  The compact set  $X$ is of circle type.  We will prove in \S8 that each component of $X$ is the Hausdorff limit of a sequence of components of $X^{(n)}$. This is proved using the results obtained in step (1).



\section{A proof of a special case of Theorem \ref{1.222}}

Theorem \ref{1.222} for the case that the hyperbolic convex hull $C_P(Y)$ is two-dimensional has to be dealt with separately and will be proved in this section. It is an easy consequence of the Riemann mapping theorem and Carath\'eodory's extension theorem of the Riemann mapping.


\begin{theorem}\label{2.1}
Suppose $X$ is a compact subset of the circle $\mathbf S^1  \subset \hat\C$ such that each connected component of $X$ is a single point and $|X| \geq 3$.  Then there exist two closed sets $Y_1, Y_2  \subset \mathbf S^1$ whose connected components are points such that

(a)  $\hat \C-X$ is conformal to $\partial C_P(Y_1)$, and

 (b) $\partial C_P(X)$ is conformal to $\hat \C-Y_2$.
 \end{theorem}

Recall that since $C_P(Y_i)$ is 2-dimensional, by our convention, 
$\partial C_P(Y_i)$ is the metric double of $C_P(Y_i)$ along its
boundary.

\begin{proof} To see (a), let
 $d^U$ be the Poincar\'e metric on $U=\hat \C -X$. Since $U$ is invariant under the orientation reversing conformal involution $\tau(z)=\frac{1}{\overline{z}}$,  by the uniqueness of the Poincar\'e metric,  we see that $\tau$ is an isometric involution of $(U, d^U)$. Since the fixed point set of an isometric involution of a Riemannian manifold is totally geodesic (page 59, Theorem 5.1  in \cite{kobayashi}), we see that the fixed point set $\mathbf S^1 \cap U$ of $\tau$ is a union of geodesics in the $d^U$ metric.  This implies that
 $U_0=\{|z|\leq1\}-X$ is a simply connected hyperbolic surface with a geodesic boundary in the metric $d^U$.  By the monodromy theorem, there exists an isometric immersion $\psi$ from $( U_0, d^U|_{U_0})$ into  the hyperbolic plane $ \H^2$. Since $( U_0, d^U|_{U_0})$ is convex with boundary consisting of geodesics, $\psi$ is an embedding.  Let $D$ be the image of $\phi$, which is a closed convex domain in $\H^2$ bounded by geodesics.
We claim that  $D$ is the convex hull $C_P(Y_1)$  of a closed set $Y_1 \subset \mathbf S^1$.  To see this, let $Y_1$ be the intersection of the closure of $D$ with the circle $\mathbf S^1$. By convexity, we see that $C_P(Y_1)$ is contained in $D$. To see that $D \subset C_P(Y_1)$, take a point $p$ in $\H^2 -C_P(Y_1)$. We will show that $p \notin D$. Suppose otherwise that $p \in D$.  By the separation theorem for convex sets, there exists a closed half-space $P$ of $\H^2_P$ such that $P$ contains $p$ and is disjoint from $C_P(Y_1)$. Find a geodesic $\gamma$ in $P$ containing $p$. Then $p \in \gamma $ and $\gamma  \cap C_P(Y_1) =\emptyset$.  Since the boundary of $D$ consists of geodesics and $p \in \gamma \cap D $, $\gamma $ must either intersect some boundary geodesic $\beta$ of $D$ or $\gamma  \subset D$. If $\gamma  \cap \beta  \neq \emptyset$, then using $\beta \subset C_P(Y_1)$, we see that $\gamma  \cap C_P(Y_1) \neq \emptyset$. This contradicts the construction of $\gamma$.  If $\gamma  \subset D$, then the endpoints of $\gamma $ are in $Y_1$ by construction. Therefore $\gamma \subset C_P(Y_1).$ This again contradicts that $\gamma  \cap C_P(Y_1) =\emptyset$.  
Therefore $p$ is not in $D$, i.e., $D \subset C_P(Y_1)$.   
  Now both $int(U_0)$ and $int(D)$ are Jordan domains, and $\psi$ is a conformal map between them. Therefore, by Carath\'eodory's extension theorem, $\psi$ extends to be a homeomorphism $\Phi$ between their closures which are $U_0 \cup \mathbf S^1$ and $D \cup Y_1$. This  homeomorphism $\Phi$ sends $X$   to $Y_1$. In particular, each component of $Y_1$ is a point.
  By the Schwarz reflection principle, $\Phi$ can be naturally extended to a conformal homeomorphism between $(U, d^U)$ and the metric double of  $C_P(Y_1)$ along its boundary.  The metric double, as our convention, is $\partial C_P(Y_1)$.

  To see part (b),  since $X \subset \mathbf S^1$, the hyperbolic convex hull $C_P(X) \subset \H^3$ is a topological disk contained in the hyperbolic plane $\H^2 \subset \H^3$. Then by the Riemann mapping theorem,  there exists a conformal diffeomorphism $\phi$ from  $int(C_P(X))$  to the unit disk $\mathbb D=\{ z \in \C \big | |z|<1\}$.  By Carath\'eodory's extension theorem,  $\phi$ extends to a homeomorphism $\Phi$ from the closure $\overline{C_P(X)}$ in $\R^3$ to the closed disk $\overline{\mathbb D}$. Let $Y_2=\Phi(X)$ whose components are all points. By the Schwarz reflection principle, $\Phi$ can be naturally extended to a conformal homeomorphism between $\partial C_P(X)$ and $\hat\C-Y_2$.

\end{proof}
\section{Shortest distance projections and area estimates on convex surfaces}

We prove several estimates on the shortest distance projections, which will enable us to give an area estimate of the transboundary extremal lengths on $\partial C_P(Y)$ in Section 6.4.
The following is the main theorem of this section.

\begin{theorem}\label{area} Suppose $X$ is a convex set of dimensional at least $2$ in the Poincar\'e model $\H_P^3$ of the hyperbolic 3-space. Then the Euclidean area of the convex surface  $\partial X$ is at most $16\pi$.
\end{theorem}

Here the Euclidean area is the 2-dim Hausdorff measure associated to the induced path metric on $\partial X$ from the Euclidean metric $d^E$ in $\mathbf R^3$.
The above theorem holds trivially if $X$ is 2-dimensional. Since, in this case, $X$ lies in a sphere $S$ perpendicular to the unit 2-sphere $\SS^2$ such that $X$ is inside the unit ball. One sees that the Euclidean area of $X$ is at most $2\pi$ since, by our convention, $\partial X$ is the metric double of $X$. For the proof, we will assume that $X$ is 3-dimensional. 
We believe the constant $16\pi$ can be improved to $8\pi$ which can be shown to be optimal.

\subsection{The Poincar\'e and Klein models of hyperbolic 3-space}
Let
$$
\B^3= \{ (x_1, x_2, x_3) \in \R^3| \sqrt{ x_1^2+x_2^2+x_3^2} <1\}
$$
be the open unit ball in 3-space. The Poincar\'e model $(\H^3, d^P)$ or simply $\H^3_P$ of the hyperbolic space is $\B^3$ equipped with the hyperbolic metric
$$ d^Px^2=\frac{ 4 \sum_{i=1}^3 dx_i^2}{ (1-|x|^2)^2}.$$
It is a complete metric of constant sectional curvature $-1$ and is conformal to the Euclidean metric $\sum_{i=1}^3 dx_i^2$. 
By definition,
$ d^Px^2 \geq 4 \sum_{i=1}^3 dx_i^2.$
Let $d^P(x,y)$ and $d^E(x,y)=|x-y|$ be the distances associated with the Poincar\'e and the Euclidean metrics, respectively. Then
\begin{equation}\label{distance-compare} d^P(x,y) \geq 2|x-y|.
\end{equation}

The Klein model of the hyperbolic 3-space $(\H^3, d^K)$ or simply $\H^3_K$ is the unit ball $\B^3$ equipped with the Riemannian metric
$$ d^Kx^2=\frac{  \sum_{i=1}^3 dx_i^2}{ 1-|x|^2}+ \frac{ ( \sum_{i=1}^3 x_idx_i)^2}{ (1-|x|^2)^2}.$$ Since  $d^Kx^2 \geq d^Ex^2$, we have
\begin{equation}\label{dist2} d^K(x,y) \geq |x-y|.
\end{equation}
It is known, from Equations (4.5.2) and (6.1.2) in \cite{rat}, or Formula 19.6.9 in \cite{berger},  that the map 
$$
\Psi(x)=\frac{2x}{1+|x|^2}: \B^3 \to \B^3
$$ 
is an isometry from the Poincar\'e model $\H_P^3$ onto the Klein model $\H_K^3$. 
Furthermore, we have the following estimate.
\begin{proposition}\label{lip}
For all $x, y \in \R^3$, $$|\Psi(x)-\Psi(y)| \leq 2 |x-y|.$$
\end{proposition}

\begin{proof} We have
$$ |\Psi(x)-\Psi(y)| =\frac{2}{(1+|x|^2)(1+|y|^2)} |(1+|y|^2)x-(1+|x|^2)y|$$
$$\leq \frac{2|x-y|}{(1+|x|^2)(1+|y|^2)}+\frac{2}{(1+|x|^2)(1+|y|^2)}|x|y|^2 - y|x|^2|.$$
Now using $|x|y|^2 - y|x|^2|^2 =|x|^2|y|^2|x-y|^2$, we see that
$$ |\Psi(x)-\Psi(y)|  \leq \frac{2|x-y|}{(1+|x|^2)(1+|y|^2)}+\frac{2|x-y||x||y|}{(1+|x|^2)(1+|y|^2)}$$
$$ =\frac{ 2|x-y|( 1+|x||y|)}{(1+|x|^2)(1+|y|^2)}$$
$$\leq  \frac{ 2|x-y|(1+|x||y|)}{ 1+2|x||y|}$$
$$\leq 2 |x-y|.$$
\end{proof}


\subsection{Shortest distance projections and area of hyperbolic convex surfaces}

The main tool we use to prove Theorem \ref{area} is the shortest distance projection in hyperbolic space. 
The shortest distance projection to a closed convex subset in a hyperbolic space is defined in the same way and enjoys similar properties as its counterpart in Euclidean geometry. An excellent reference of the topic is \cite{em}. We will briefly recall the relevant properties and refer to the details of the proofs to \cite{em}, pages 121-127. Recall that  $\overline{A}$ denotes the closure of a set $A$ in the Euclidean space $\R^n$.  Given a non-empty closed convex set $X$ in the n-dimensional hyperbolic space $\H_P^n$, considered as the Poincar\'e ball model, the \it hyperbolic shortest distance projection \rm (or \it the nearest point retraction  \rm    in \cite{em}) is a map $\pi=\pi^H_X: \overline{\H_P^n} \to \overline{X}$ defined as follows.  If $p \in \H^n$,  then $\pi(p) =B \cap X$ where $B$ is the largest ball centered at $p$ such that $int(B) \cap X =\emptyset$, i.e., $\pi(p) \in X$ is the point in $X$ closest to $p$. If $p \in \partial \H^n-\bar X$, then $\pi(p) = X \cap L$ where $L$ is the largest horoball whose interior is disjoint from $X$ and $p\in \bar L$. (Sometimes $p$ is called the center of the horoball $L$. See page 122 of \cite{em}).  The convexity of $X$ implies that $\overline{L} \cap \overline{X}$ consists of a single point. Note that if $p \in X$, then $\pi(p)=p$.


The basic properties of the shortest distance projection $\pi$ are in the following lemma.

\begin{lemma}\label{nearest} Suppose $\pi: \overline{\H^n_P} \to \overline{X}$ is the shortest distance projection.

(a) If $p \notin \overline{X}$, then the hyperbolic codimension-1 plane $W_p$ through $\pi(p)$ perpendicular to the geodesic from $p$ to $\pi(p)$ is a supporting plane for the convex set $X$. Furthermore, $d^P(\pi(x), \pi(y)) \leq d^P(x, y)$ for all $x,y \in \H^n$.

(b) If  $p \notin \overline{X}$ and $x \in X$, then the angle $\angle p\pi(p)x$ at $\pi(p)$ is at least $\pi/2$. 

(c) If $X$ is a geodesic in the Poincar\'e model of the hyperbolic plane $\H_P^2$ and $A$ is a connected component of $\partial \H_P^2-\partial X$, then the shortest distance projection $\pi|_{A}$  coincides with the inversion about the circle $C$ which contains $\partial X$ and bisects the angles formed by $X$ and $A$.
\end{lemma}

\begin{proof} The proof of part (a) is in Lemma II.1.3.2 and Lemma II.1.3.4 in \cite{em}. 

Part (b) follows from part (a) since the set of all points $q$ in $\H^n$ such that the angle $\angle p \pi(p) q =\pi/2$ is the codimension-1 plane $W_p$. The condition that $x$ and $p$ lie in different sides of $W_p$ implies $\angle p\pi(p)x \geq \pi/2$.

To see part (c), we use the upper-half-plane model $\{z \in \C| im(z)>0\}$ of $\H^2$ and take $X =\{ iy | y \in \R_{>0}\}$ and $A=\{ x | x \in \R_{>0}\}$. The shortest distance projection onto $X$ when restricted to $A$ is $\pi(x) = i x$ which is the same as the reflection about the line $C =\{ xe^{\pi i/4}| x \in \R\}$ on $A$. Clearly the line $C$ bisects the angle formed by $A$ and $X$.
\end{proof}



The following lemma gives an estimate of the distortion of the shortest distance projection that appeared in Lemma \ref{nearest} (c).

\begin{lemma}\label{221} Suppose $f(z)$ is the inversion about a circle $C\subset\hat\C$. 

(a) If $C\cap\hat\C$ is a straight line, then $|\overline{f}'(z)| = 1$.

(b) If $C$ is a circle in $\C$ and $z\in\C$ is outside of $C$, then 
$|\overline{f}'(z)|\leq 1$.

(c) If $C$ is a circle in $\C$ and $B$ is a circle in $\C$ intersecting $C$ with an angle of $\frac{\pi}{4}$, then $|\overline{f}'(z)| \leq 2$ for all $z \in B$.
\end{lemma}

\begin{proof} 
Part (a) is trivial. 
Now we assume that
$C$ is a circle of radius $r$ and centered at $p$. Then for any $z\in\C$,
$f(z) = \frac{r^2}{\overline{z-p}} +p$ and
$|\overline{f}'(z)|=\frac{r^2}{|z-p|^2}$. So part (b) holds.

To see part (c), by composing with a rotation and a translation, we may assume that  $B$ is a circle centered at the origin $0$. See Figure \ref{inv1}. Let  $s$ be the point in $B$ closest to $p$.  By the assumption, the angle $\angle 0qp$ is $\frac{3\pi}{4}$ and the angle $\theta =\angle sqp$ is in the interval $[\frac{\pi}{4}, \frac{3\pi}{4}]$. If $z \in B$, then $|\overline{f}'(z)| \leq ( \frac{r}{|s-p|})^2 =(\frac{\sin(\phi)}{\sin(\theta)})^2 \leq \frac{1}{\sin^2(\theta)} \leq 2$ where we have used the Sine law for the triangle $\Delta pqs$, $\phi =\angle qsp$ and $\theta \in [\frac{\pi}{4}, \frac{3\pi}{4}]$.

\begin{figure}[ht!]
\begin{center}
\begin{tabular}{c}
\includegraphics[width=0.38\textwidth]{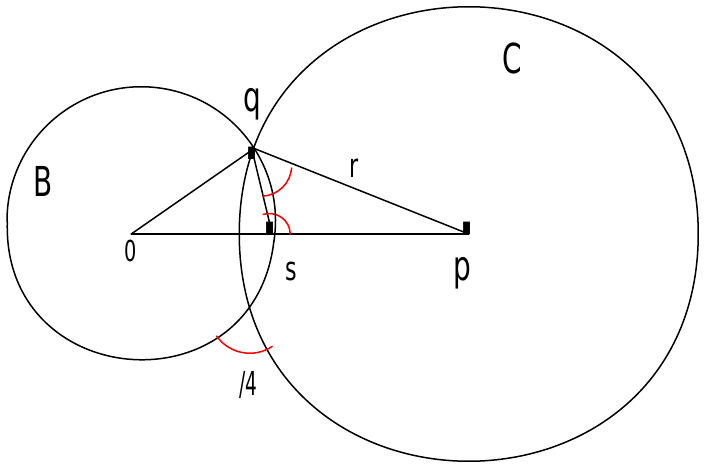}
\end{tabular}
\end{center}
\caption{Lipschitz property of inversion} \label{inv1}
\end{figure}
\end{proof}

To prove Theorem \ref{area}, we may assume the hyperbolic convex set $X$ is closed in $\H_P^3$. The result is obvious if $X$ is a 2-dimensional convex set. Indeed, in this case, $X$ is contained in a (Euclidean) sphere of radius at most one. Therefore its Euclidean area is at most $4\pi$.  Let us assume $X$ contains an interior points. 
Using the fact that the area of the 2-sphere is $4\pi$, we see that Theorem \ref{area} is a consequence of the following two properties of the shortest distance projection $\pi$.

\begin{theorem} 
\label{2Lip}
Let  $X$ be a closed convex set in the Poincar\'e model $\H_P^3$ and $\pi$ be the hyperbolic shortest distance projection from $\SS^2$ to $X$. Then for all $p, q \in \SS^2$,
\begin{equation}\label{lip2}
|\pi(p)-\pi(q)|\leq 2 d^{\SS}(p,q).
\end{equation}
\end{theorem}

We remark that the above theorem also holds for high dimensional hyperbolic spaces $\H^n$.

\begin{proof}
If $\pi(p)=\pi(q)$, then the result holds trivially. So now we assume $\pi(p)\neq \pi(q)$.
If $p\in\bar X$ and $q\in\bar X$, then easily we have that
$$
|\pi(p)-\pi(q)|=|p-q|\leq d^{\SS}(p,q).
$$
In the following argument we will assume $p\notin \bar X$ and $q\notin\bar X$, and the case $p\in\bar X$ or $q\in\bar X$ could be proved in a similar and simpler way.

Let $Y$ be the hyperbolic geodesic joining
$\pi(p)$ to $\pi(q)$ and $W_p$ and $W_q$ be the codimension-1 hyperbolic planes perpendicular to $Y$ at $\pi(p)$ and $\pi(q)$ respectively. Note that $W_p \cap W_q =\emptyset$ since $\pi(p) \neq \pi(q)$ and both are perpendicular to $Y$. Let $H_p$ and $H_q$ be the two disjoint half spaces in $\H^3$ bounded by $W_p$ and $W_q$ respectively. By Lemma \ref{nearest} (b), both angles $\angle p\pi(p) \pi(q)$ and $\angle q \pi(q) \pi(p)$ are at least $\pi/2$, it follows that $p \in U_p:=\overline{H_p}\cap\SS^2$ and $q \in U_q:=\overline{H_q}\cap\SS^2$. 
Since $H_p \cap H_q =\emptyset$,  we have that $U_p\cap U_q=\emptyset$.

\begin{figure}[ht!]
\begin{center}
\begin{tabular}{c}
\includegraphics[width=0.49\textwidth]{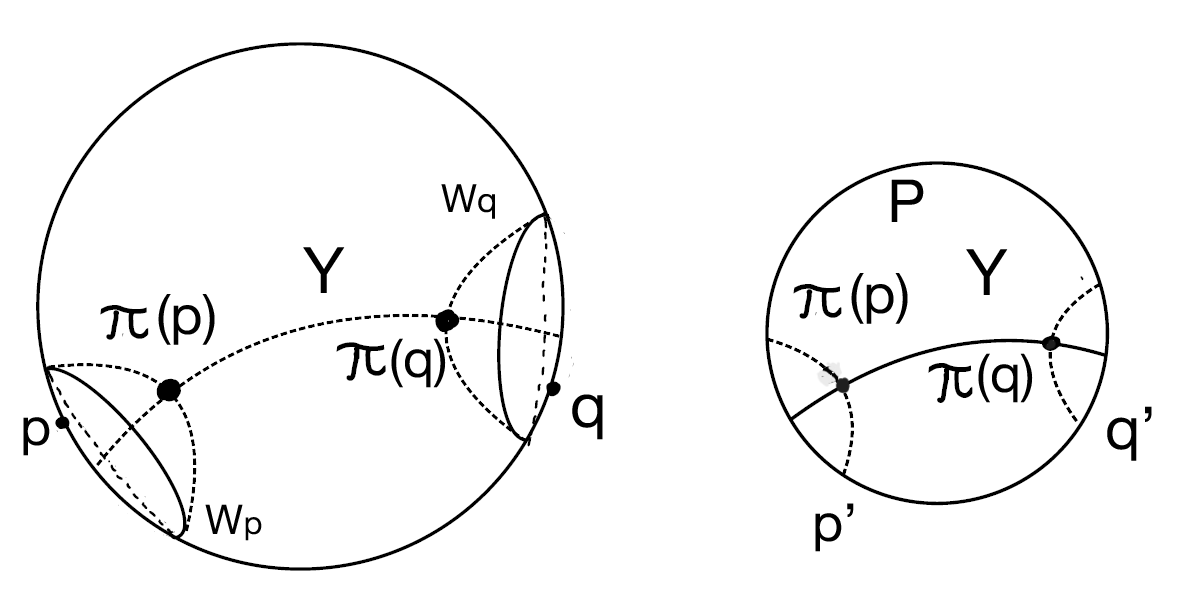}
\end{tabular}
\end{center}
\caption{Lipschitz property of nearest point projection} \label{inv2}
\end{figure}

Let $O_p$ and $O_q$ be the centers of spherical disks $U_p$ and $U_q$ respectively, and $\gamma$ be a shortest geodesic from $O_p$ to $O_q$ on $\SS^2$ such that 
$\gamma$ intersects $\partial U_p$ and $\partial U_q$ at $p'$ and $q'$ respectively. 
Then 
\begin{equation}\label{cap}  d^{\SS}(U_p, U_q) =d^{\SS}(p', q') \leq d^{\SS}(p,q). \end{equation}
Therefore it suffices to show that $|\pi(p)- \pi(q)| \leq 2 d^{\SS}(p', q')$.  To this end, let $P$ be a Euclidean plane  passing through $(0,0,0)$ and containing $Y$ and $\gamma$ and consider the unit disk $D=P\cap\overline{\mathbb B^3}$. The inequality $|\pi(p)- \pi(q) | \leq 2 d^{\SS}(p', q')$ follows from 
$ | \pi(p) -   \pi(q) | \leq  d^E_{\partial D}(p', q')$ since
since $d^{\SS}(p', q') =d_{\partial D}^E(p',  q')$ where $d_{\partial D}^E$ is the induced Euclidean path metric on the unit circle $\partial D$.

Let $M$ be the component of
$\partial D -\overline{Y}$ which contains $\{p',q'\}$; $C$ be the circle or the straight line in the plane $P$ passing through $Y \cap M$ and bisecting the angles between $Y$ and $M$; and $f$ be the inversion about $C$ in $P$. 
Then the angle between $C$ and $\partial D$ could be either $\pi/4$ or $3\pi/4$. In the case of $3\pi/4$, $M$ lies outside of the circle $C$. In either case, by Lemma \ref{221}, $|f'(z)|\leq2$ for any $z\in M$. So we only need to prove $f(p')=\pi(p)$ and $f(q')=\pi(q)$ in order to obtain $|\pi(p)-\pi(q)|\leq2d^{\SS}(p',q')$. 

To see $f(p')=\pi(p)$, note that $\pi(p)$ is the intersection point of $Y$ with $W_p\cap P$. Since the circle $C$ bisects $Y$ and $M$, $f(Y) \subset \partial D$. Furthermore, since $C$ contains the two endpoints of $Y$, $C$ is perpendicular to the circle $Z$ containing $W_P \cap P$. 
 This shows $f(Z) =Z$.   It follows that $f$ sends $Y \cap Z$ to $M \cap Z$, i.e., $f(\pi(p))=p'$ or $f(p') =\pi(p)$.
 The same argument shows $f(q')=\pi(q))$.
\end{proof}

Theorem \ref{2Lip} implies the hyperbolic shortest distance projection $\pi$ is continuous on $\SS^2$. The following lemma shows it is an onto map.

\begin{lemma} \label{homeo}
 Suppose  $X$ is a closed 3-dimensional convex set in $\H^3_P$ with non-empty interior, then the shortest distance projection $\pi$  from $\SS^2$ to $\partial X \cup (\overline{X} \cap \SS^2)$ is surjective. \end{lemma}
\begin{proof} Take a point $p \in \partial X \cup (\overline{X} \cap \SS^2)$. If
$p \in \overline{X} \cap \SS^2$, then $\pi(p)=p$. If $p \in \partial X$, let $W_p$ be a supporting plane for $X$ at $p$ and $\gamma$ be the geodesic ray  perpendicular to $W_p$ at $p$ such that $\gamma$ intersects $X$ only at $p$. The endpoint of the ray $\gamma$ is a point $q \in \SS^2$. By definition, $\pi(q)=p$.
\end{proof}

\subsection{Another estimate for the hyperbolic shortest distance projection}
\begin{figure}[ht!]
\begin{center}
\begin{tabular}{c}
\includegraphics[width=0.35\textwidth]{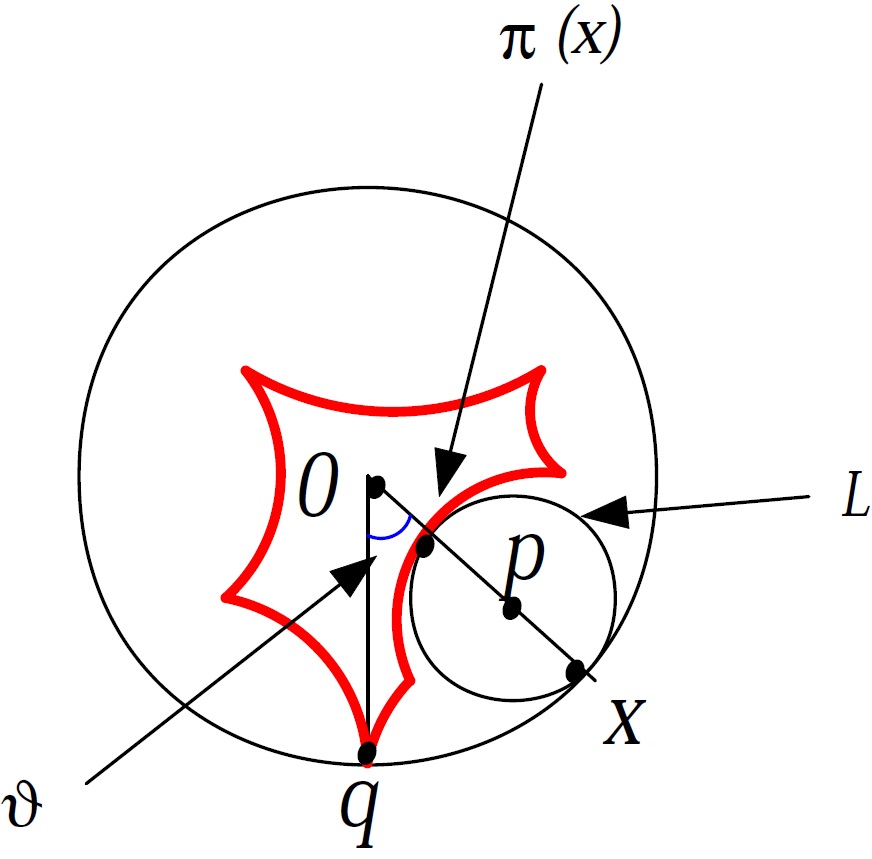}
\end{tabular}
\end{center}
\caption{Eestimate for shortest distance nearest projection} \label{inv3}
\end{figure}

\begin{proposition}\label{3.7}
Assume $X$ is a hyperbolic closed convex set of dimension at least 2 in $\H^3_P$ containing origin $O$ and $q \in\bar X\cap\SS^2$ (See Figure \ref{inv3}). Let $\pi$ be the shortest distance projection from $\SS^2$ to $\partial X \cup (\overline{X}\cap \SS^2)$. 

(a) For any $x \in \SS^2$,
$$
\frac{1}{8}d^{\mathbf S}(x,q)\leq |\pi(x)-q| \leq
2d^{\mathbf S}(\pi(x),q).
$$

(b)
Let $D_R$ be the closed ball of radius $R$ in $(\SS^2,d^{\SS})$ centered at $q$, and $B_R(q)$ be the closed ball of radius $R$ in 
$( \partial X \cup (\overline{X}\cap \SS^2) , d^E)$ centered at $q$. Then for any $R>0$ 
$$
\pi(D_R) \cap \pi(\mathbf S^2-D_{65R})  =\emptyset
$$
and
$$
B_{R/8}(q) \subset \pi(D_R) \subset B_{8R}(q).
$$
\end{proposition}

\begin{proof}
For part (a), 
$|\pi(x)-q|\leq 2d^{\mathbf S}(x,q)$ is a consequence of Theorem \ref{2Lip} using $\pi(q)=q$. It remains to prove the lower bound of $|\pi(x)-q|$ for $x\neq q$.
For simplicity we denote $d^{\mathbf S}(x,q)=\theta \in (0,\pi]$.
Assume $L$ is the largest horoball such that $int(L)\cap X=\emptyset$ and $x\in\bar L$. Let
$r$ and $p$ be the Euclidean radius and Euclidean center of the horoball $L$. Since $X$ contains $(0,0,0)$ and the interior of $L$ is disjoint from $X$, the Euclidean radius $r$ of $L$ is at most $1/2$, i.e., $r\leq1/2$.
If $\theta\geq\pi/2$, then 
$$
|q-\pi(x)|\geq |q-x|-|x-\pi(x)| = |q-x|-2r\geq\sqrt2-2\cdot\frac{1}{2}\geq\frac{\pi}{8}\geq\frac{\theta}{8} 
$$
and we are done.


If $\theta<\pi/2$,  then the Euclidean distance from $p$ to the line through $q,O$ is at least $r$, i.e.,
\begin{equation} \label{radiusest}
\sin \theta\geq\frac{r}{1-r}.
\end{equation}

Since $\pi(x)\in\partial X$, we have, 
\begin{align*}
|q-\pi(x)| \geq |q-p|- |p - \pi(x)| =|p-q|-r.
\end{align*}
To estimate the right-hand side of the above inequality, by the Cosine Law on the triangle $\triangle qOp$, we have, 
$$
|p-q|=\sqrt{1+(1-r)^2-2(1-r)\cos \theta}.
$$
Denote $f(t)=\sqrt{1+(1-t)^2-2(1-t)\cos\theta}$ for $t\in[0,1]$. Then we have
$|f(t)|\geq|\cos\theta-(1-t)|$ and
$$
f'(t)=\frac{\cos \theta-(1-t)}{f(t)}\in[-1,1].
$$
This shows that  $f(t)-t$ is a decreasing function of $t$. Note that $|p-q|=f(r)$.
Now for 
 $\theta<\pi/2$, by (\ref{radiusest}),  we have $r\leq{\sin\theta}/{(1+\sin \theta)}$ and

\begin{align*}
|p-q|-r =&f(r)-r \geq f(\frac{\sin\theta}{(1+\sin \theta)})-\frac{\sin\theta}{(1+\sin \theta)}\\
=& \frac{\sqrt{(1+\sin \theta)^2+1-2(1+\sin \theta)\cos \theta}-\sin \theta}{1+\sin \theta}\\
\geq&\frac{1}{2}(\sqrt{(1+\sin \theta-\cos \theta)^2+\sin^2 \theta}-\sin \theta)\\
\geq&\frac{1}{2}(\sqrt{\sin^2 \theta+\sin^2 \theta}-\sin \theta)\\
=&\frac{\sqrt2-1}{2}\sin \theta\\
\geq&\frac{\theta}{8}.
\end{align*}


Now we prove part (b).
If $z_1 \in D_R$ and $z_2 \in \mathbf S^2- D_{65R}$, then by part (a)
$$
|\pi(z_2)-q|\geq\frac{1}{8}d^{\SS}(z_2,q)
\geq\frac{65R}{8}>2d^{\SS}(z_1,q)\geq|\pi(z_1)-q|.
$$
So $\pi(z_1)\neq\pi(z_2)$.

Then by part (a) and the sujectivity of $\pi$ (Proposition \ref{homeo} (b)), 
$$
B_{R/8}(q)\subset\pi(\SS^2)-\pi(\SS^2-D_{R})
\subset \pi(D_R)\subset B_{8R}(q).
$$
\end{proof}

\begin{remark}  We thank the referee who pointed out that the inequality $|\pi(x)-q| \leq 2 d^{\SS}(\pi(x), q) $ follows from Theorem \ref{2Lip} and improved our original estimate.     
\end{remark}

\section{Hausdorff convergence and the Poincar\'e metrics}

Let us begin by briefly recalling the Hausdorff distance.  Suppose $A$ is a subset of a metric space $(Z,d)$ and $r>0$. The $r$-neighborhood of $A$, denoted by
$N_r(A,d)$,  is the open set $\{z \in Z|d(z,A)<r\}$.  If $A,B$ are two closed subsets of $(Z,d)$, then their \it Hausdorff distance \rm  $d_h(A,B)$  is $\inf\{ r| A \subset N_r(B,d)$ $ \text{and}$ $ B \subset N_r(A,d)\}$.  
A sequence of closed subsets  $\{ X_n\}$ in $(Z,d)$ is said to converge in Hausdorff distance to a closed set $X$ if $\lim_n d_h(X_n, X) =0$.  For a compact metric space $(Z,d)$, the set of all closed subsets in the Hausdorff distance is compact. See \cite{burago}.

Alexandrov convergence is used extensively on the convergence of convex surfaces in the hyperbolic and Euclidean spaces.  A sequence \{$X_n\}$ of closed subsets in a metric space $(Z,d)$ is \it Alexandrov convergent \rm to a closed subset $X$ if (i) for any $p \in X$, there exists a sequence $\{p_n\}$ with $p_n \in X_n$ such that $\lim_n p_n =p$ and (ii) for any convergent sequence $\{p_{n_i}\}$ with $p_{n_i}\in X_{n_i}$ and $\lim_i p_{n_i}=p$, we have $p \in X$. Alexandrov convergence is independent of the choice of the distance $d$. Clearly if $\{X_n\}$ converges to $X$ in the Hausdorff distance, then $\{X_n\}$ Alexandrov converges to $X$. In general, the converse is not true.  But if the space $(Z,d)$ is compact, then Alexandrov convergence is equivalent to the Hausdorff convergence. In particular, for a compact metric space $(Z,d)$,  the Hausdorff convergence of compact subsets is independent of the choice of metrics.   See \cite{burago} and \cite{al} for details.


\subsection{Alexandrov's work on convex surfaces}
By a \it complete convex surface \rm $S$ in $\mathbb R^3$ or $\H^3$ we mean the boundary of a closed convex set of dimension at least $2$. Here, the dimension of a convex set in $\mathbb R^3$ or $\H^3$ is defined to be the dimension of the smallest totally geodesic submanifold that contains the convex set. 
The convergence theorem of Alexandrov, which will be used extensively, is the following.

\begin{theorem}[Alexandrov] \label{convergence} Suppose $\{S_n\}$ is a sequence of complete connected convex surfaces in the Euclidean or hyperbolic 3-space Alexandrov converging to a complete connected convex surface $S$.  If $x, y \in S$ and $x_n, y_n \in S_n$ such that $\lim_n x_n =x$ and $\lim_n y_n =y$, then
$$ \lim_n d_{S_n}(x_n, y_n) =d_{S}(x,y),$$
where $d_{S}$ is the induced path metric on the convex surface $S$.
\end{theorem}

The proof of this theorem for the Euclidean case is on pages 91-95 of \cite{al}. The hyperbolic case was stated in section 3 of Chapter 12 of \cite{al}.

\subsection{The Poincar\'e metrics and their convergence}

If $X$ is a closed set in the Riemann sphere  $\hat{\C}$ which contains at least three points,  then none of the connected components of   $\hat{\C}-X$ is conformal to the complex plane $\C$ or the punctured plane $\C-\{0\}$. Therefore, by the uniformization theorem, each connected component of $\hat{\C}-X$ carries the Poincar\'e metric.  The \it
Poincar\'e metric \rm on $\hat{\C}-X$ is defined to be the Riemannian metric whose restriction to each connected component is the Poincar\'e metric.

The main result in this section is the following.

\begin{theorem} \label{poincare}
Suppose $\{X_n\}$ is a sequence of compact sets in the
Riemann sphere converging to a compact set $X$ in the Hausdorff distance such that $X$ contains at least three points and $X \neq \hat \C$. Let $d_n=a_n(z)|dz|$ and $d^U=a(z)|dz|$ be
the Poincar\'e metrics on $U_n=\hat{\C}-X_n$ and $U
=\hat{\C}-X$ respectively. Then  $a_n(z)$ converges uniformly on compact subsets of $U$ to $a(z)$. Furthermore, if $p,q$ are two points in a connected component of $U$ and $p_n, q_n$ are two points in a connected component of $U_n$ such that $\lim_n p_n =p$ and $\lim_n q_n =q$, then \begin{equation}\label{pequal}  \lim_n d_n(p_n, q_n) =d^U(p,q). \end{equation}
\end{theorem}

Note that by Hausdorff convergence, for any compact set $K \subset U$, $K \subset U_n$ for $n$ large. Therefore $a_n$ is defined on $K$ for large $n$.
Also, by our convention,  the distance $d_n$ in (\ref{pequal}) denotes the Riemannian distance associated with the Poincar\'e metric $d_n=a_n(z)|dz|$. 
\begin{proof}
To begin, let us recall a known theorem from Riemannian geometry about the convergence of Riemannian distance functions when the Riemannian metrics converge.  See for instance Proposition 11.3.2 in \cite{petersen}. We provide a short proof for completeness.

\begin{lemma}\label{r-converge}  Assume $U_n$ and $U$ are the same as in Theorem \ref{poincare}. Suppose $d_n =a_n(z)|dz|$ is a complete Riemannian metric on $U_n$ such that $\{a_n(z)\}$ converges uniformly on compact subsets of $U$ to a positive function $b(z)$ defined on $U$.  Let $d_{\infty}=b(z)|dz|$ be the limiting Riemannian metric and let $p$ and $q$ be two points in a connected component of $U$.  Then,
\begin{equation}\label{ineq12} \limsup_n d_n(p,q) \leq d_{\infty}(p,q), \end{equation} and 
if furthermore $d_{\infty}$ is a complete Riemannian metric, 
\begin{equation}\label{ineq13}\lim_{n} d_n(p, q) =d_{\infty}(p,q).\end{equation}
\end{lemma}
\begin{proof} To prove (\ref{ineq12}), take any $\epsilon >0$ and a smooth path
$\gamma: [0,1] \to (U, d_{\infty})$ from $p=\gamma(0)$ to $q =\gamma(1)$ such that $d_{\infty}(p, q) \geq l_{\infty}(\gamma) -\epsilon$.  Since $\{a_n(z)\}$ converges uniformly on the image of $\gamma$, we have $\lim_n l_n(\gamma) =l_{\infty}(\gamma)$ where $l_n(\gamma)$ and $l_{\infty}(\gamma)$ are the lengths of the curve $\gamma$ in the Riemannian metrics $d_n$ and $d_{\infty}$ respectively. By definition $l_n(\gamma) \geq d_n(p,q)$, we have,  
$$d_{\infty}(p,q) \geq l_{\infty}(\gamma) -\epsilon =\lim_n l_n(\gamma) -\epsilon \geq \limsup_n d_n(p,q) -\epsilon.$$ Therefore 
(\ref{ineq12}) holds. 

To prove (\ref{ineq13}), by (\ref{ineq12}), it suffices to show  $$d_{\infty}(p,q) \leq \liminf_n d_n(p, q).$$  
By (\ref{ineq12}),  we choose $R>0$ large such that $d_n(p, q) \leq R$ for all $n$. Consider the set $W =\overline{B_{4R}(p, d_{\infty}) }=\{ x \in U | d_{\infty}(x,p) \leq 4R\}$. Since $d_{\infty}$ is complete, $W$ is compact. Therefore,  $W \subset U_n$ for $n$ large. Since $\{a_n(z)\}$ converges uniformly to $b(z)$ on $W$, so for sufficiently large $n$, $a_n(z)\geq b(z)/2$ for all $z\in W$. As a consequence 
$$
d_n(p,\partial W)\geq d_{\infty}(p,\partial W)/2=2R.
$$
This shows, using $d_n(p, q) \leq R$ and $d_n$ is  complete,  that the shortest geodesic $\gamma_n$ in $(U_n, d_n)$ joining $p$ to $q$ is contained in the compact set $W$. 
Since $\{a_n(z)\}$ converges to $b(z)$ uniformly on $W$, there exists a sequence of positive numbers $\epsilon_n$ converging to $0$, such that
$$ a_n(z)\geq (1-\epsilon_n)b(z) $$
for all $z\in W$.
Then
$$
\liminf_n d_n(p,q)=
\liminf_n l_{n}(\gamma_n)
\geq\liminf_n(1-\epsilon_n) l_{\infty}(\gamma_n)
$$
$$
\geq\liminf_n(1-\epsilon_n)d_{\infty}(p,q)=d_{\infty}(p,q).
$$
\end{proof}

It suffices to prove the theorem for the case that all $X_n$ contain a fixed set of three points.  Indeed, let $\{u,v,w\} \subset X$ with $u,v,w$ pairwise distinct and consider sequences $u_n, v_n, w_n \in X_n$ such that $\lim_n u_n =u, \lim_n v_n =v$ and $\lim_n w_n =w$. There exists a M\"obius transformation $\psi_n$  sending $u_n, v_n, w_n$ to $u, v, w$ respectively. By the construction, $\psi_n$ converges uniformly to the identity map in the spherical metric $d^{\mathbf S}$ on $\hat{\C}$. This implies that  $\psi_n(X_n)$ converges in Hausdorff distance to $X$ and $\{u,v,w\} \subset \psi_n(X_n)$. Now suppose the theorem has been proved for the sequence $\psi_n(X_n)$. Using the fact that $\psi_n$ induces an isometry between Poincar\'e metrics on the open sets $\hat{\C} -X_n$ and $\hat{\C} -\psi_n(X_n)$, we see Theorem \ref{poincare} holds for the general case.

Now using a M\"obius transformation, we may assume that
 $X_n$ contains $\{0,1, \infty\}$ and converges to $X$ in Hausdorff
 distance. 

The strategy of the proof is as follows. First, we show that for any compact set $K$ in $U$, the family of functions
$\{a_n|_K\}$ contains a uniformly convergent subsequence.  By the Cantor diagonal process, we see that there is a subsequence of $\{a_n\}$ which converges uniformly on compact subsets to a limit function $b(z)$ on $U$. The limiting metric $b(z)|dz|$ 
can be shown easily to have constant curvature $-1$.  Finally, we show that the hyperbolic metric $b(z)|dz|$ is complete. Therefore, by the uniqueness of the Poincar\'e metric, $b(z)|dz|$ is the Poincar\'e metric $d^U=a(z)|dz|$.   Since all limits of convergent subsequences are the same, it follows that the sequence $\{a_n(z)\}$ converges to $a(z)$.

To show that $\{a_n|_K\}$ contains a convergent subsequence in the $L^{\infty}$-norm, write the open set $U$ as a union of open Euclidean round disks $B_j$ such that the Euclidean closure $\overline{B_j}$ is still in $U$.  Then each compact set $K$ in $U$ is contained in a finite union of these closed disks $\overline{B_j}$.  By using the Cantor diagonal process,  it suffices to prove the statement for $K$ to be a compact ball $\{z ||z-p|\leq  r\}$ in $U$.  We will use the following well-known consequence of the Schwarz-Pick lemma (see Theorem 10.5 in \cite{beardon} for proof). 

 \begin{lemma}(Schwarz-Pick) \label{332}   Suppose  $A$ and $B$ are two open sets in $\hat{\C} -\{q_1, q_2, q_3\}$ such that $A \subset B$ and
 $d^A=a_A(z)|dz|$ and $d^B=b_B(z)|dz|$ are the Poincar\'e metrics on $A$ and $B$ respectively. Then
$d^A \geq d^B$, i.e., $a_A(z) \geq b_B(z)$ for all $z\in A$. 

\end{lemma}

 Let $d_{0,1,\infty}$ be the Poincar\'e metric on $\hat\C-\{0,1,\infty\}$. Since $U_n \subset \hat\C-\{0,1,\infty\}$, by Lemma \ref{332},  \begin{equation}\label{hp1}  d_n \geq d_{0,1,\infty}.
 \end{equation}

\begin{lemma}\label{L2} If $K =\{z | |z-p|\leq r\}$ is in $U_n$ for all $n$ large, then the sequence $\{ a_n|_K\}$ contains a convergent subsequence in the $L^{\infty}$-norm. Furthermore, any limiting function $b(z)$ produces a hyperbolic metric $b(z)|dz|$ on $int(K)$.
\end{lemma}
\begin{proof}  
Let $D$ be an open disk centered at $p$ containing $K$ such that $D \subset U_n$ for $n$ large. Consider the incomplete simply connected hyperbolic surface  $(D, d_n|_D)$. Since $D$ is simply connected, by the monodromy theorem, there exists an orientation preserving isometric immersion $f_n: (D, d_n|_D) \to (\D, d_{\D})$ such that $f_n(p)=0$. 
The map $f_n$ is also a branch of an inverse of the universal covering
map.
In particular,
 $f_n: D \to \D$  are holomorphic maps bounded by 1. 
 Therefore
$\{f_n\}$ forms a normal family and contains a subsequence that converges uniformly on compact sets to an analytic function $h$  in $D$. For simplicity, we assume that the subsequence is \{$f_n$\}.  We claim that $|f_n'(p)|$ is bounded away from $0$. Indeed, consider the standard tangent vector $v=\frac{\partial}{\partial x}$ at $p$. By (\ref{hp1}),  the length of $v$ in $d_n$ is at least the length $\delta$ of $v$ in $d_{0,1,\infty}$. It follows that the length of $f_n'(p)$ in the Poincar\'e metric on $\mathbb D$ is at least $\delta$.
This shows that $h'(p) \neq 0$ and, therefore, $h$ is not a constant. Since $h(p)=0$, $h'(p) \neq 0$ and $|h(z)| \leq 1$, we see that $|h(z)|<1$.   Furthermore, the same argument shows $h'(z) \neq 0$ for all $z \in D$. 
    Since the Poincar\'e metric on $\D$ is $\frac{2 |dw|}{1-|w|^2}$, it follows that  $a_n(z) =2\frac{|f_n'(z)|}{1-|f_n(z)|^2}$.  By the uniform convergence of the analytic functions $f_{n}$, we conclude that $a_{n}$ converges uniformly to $b(z)=2\frac{|h'(z)|}{1-|h(z)^2|}$ on $K$.  Since $h$ is analytic and $h'(z) \neq 0$,  hence $b(z)|dz|$ is a hyperbolic metric.
\end{proof}



The next lemma shows the limiting metric $b(z)|dz|$ is complete. 

\begin{lemma}\label{L1} 
If $a_n(z)$ converge to $b(z)$ uniformly on any compact subset of $U$, then $d_\infty=b(z)|dz|$ is a complete Riemannian metric on each component of $U$. In particular, $b(z)|dz|$ is the Poincar\'e metric $d^U=a(z)|dz|$ on $U$.
\end{lemma}
\begin{proof}
Take a connected component $U'$ of $U$ and a
Cauchy sequence $x_n$ in $(U', d_{\infty}|_{U'})$. 
By the Schwarz-Pick lemma applied to $U' \subset \hat{\C}-\{0,1, \infty\}$, 
$
d_{0,1,\infty}(x,y)\leq
d_\infty(x,y)
$
for all $x,y\in U'$ and hence 
$\{x_n\}$ is a Cauchy sequence in $\hat\C-\{0,1,\infty\}$ in the
$d_{0,1,\infty}$ metric. In particular,  there is a point $p \in
\hat\C-\{0,1,\infty\}$ so that $\lim_{n} x_n=p$. We claim that
$p \in U$.

Assuming the claim and using the fact that the topology determined
by $d_{\infty}$ and the Euclidean metric $d^E$ on $U$ are the same, we see that $p$ is in $U'$ and $x_n$ converges to $p$ in the $d_{\infty}$ metric.

To see the claim, suppose otherwise that $p\in X$. Then $U
\subset \hat\C-\{0,1, p\}$. By the definition of Hausdorff convergence,
there exists a sequence $p_n \in X_n$ such that $\lim_n p_n =p$ and $p_n \neq 0, 1,\infty$. Let $d_n'$
be the Poincar\'e metric on $\hat\C-\{0,1,p_n\}$. By Lemma \ref{332} and  $\hat\C-X_n \subset
\hat\C-\{0,1,p_n\}$, we have $ d_n' \leq d_n$.
Let $d_{0,1,p}$ be the Poincar\'e metric on $\hat\C-\{0,1,p\}$ and $\rho_n$ be the  M\"obius transformation sending the triple $(0, 1, p_n)$ to $(0, 1, p)$. Then $\rho_n$ converges to the identity map uniformly on $(\hat{\C}, d^{\\S})$ and $\rho_n$ is an isometry from
$(\hat\C-\{0,1,p_n\},d_n')$ to $(\hat\C-\{0,1,p\},d_{0,1,p})$. In particular, $
d_{0,1,p}(x,y) =\lim_nd_{0,1,p}(\rho_n(x),\rho_n(y))$.
For all $x,y\in U'$,  by (\ref{ineq12}) in Lemma (\ref{r-converge}), we have
$$
d_{0,1,p}(x,y)
=\lim_nd_{0,1,p}(\rho_n(x),\rho_n(y))
=\lim_nd_n'(x,y) \leq \limsup_n d_n(x,y)
$$
$$
\leq d_\infty(x,y).
$$
Hence $\{x_n\}$ is a
Cauchy sequence in the $d_{0,1,p }$ metric on $\hat\C-\{0,1,p\}$. Since,
$d_{0,1,p}$ is a complete metric, it follows that there is $q \neq
p$ in $\hat\C-\{0,1,\infty\}$ such that $\lim_n x_n=q$ in $\hat\C$. This
contradicts the assumption that $\lim_n x_n =p$ in $\hat\C$.  \end{proof}

Finally, let us prove (\ref{pequal}). Note that 
\begin{equation} \label{referee1}
|d_n(p_n, q_n)-d^U(p,q)|\leq
d_n(p_n,p)+d_n(q_n, q) +
|d_n(p,q)-d^U(p,q)|. 
\end{equation}
By Lemmas \ref{r-converge} and \ref{L1}, we have  $|d_n(p,q) -d^U(p,q)| \rightarrow0$. It remains to show that both 
$d_n(p_n, p)$ and $d_n(q_n, q)$ converge to zero.  To see this, 
since in a small neighborhood of $p$, $|a_n(z)|$ is uniformly bounded and the Euclidean distance $|p_n-p|$ goes to 0, we see that $d_n(p_n, p)$ is bounded by $C |p_n-p|$ for some constant $C$ independent of $n$.  Therefore, $\lim_n d_n(p_n, p) =0$ and similary, $\lim_n d_n(q_n, q)=0$. 
\end{proof}
\begin{remark} We thank the referee for suggesting Lemma \ref{r-converge} and the estimate (\ref{referee1}) which drastically simplifies our original proof.  

\end{remark}

\subsection{A generalized form of Arzela-Ascoli theorem}
For our proof, we need a slightly more general form of the Arzela-Ascoli theorem.  Recall that a family of maps $f_n: (X_n, d_n)  \to (Y_n, d_n')$ between metrics spaces is called \it equicontinuous \rm if for any $\epsilon >0$, there exists $\delta>0$ such that for all $n$ and $x_n, y_n \in X_n$ with $d_n(x_n, y_n) < \delta$ we have that $d_n'(f_n(x_n), f_n(y_n)) < \epsilon$.

\begin{theorem}\label{aa}Suppose  $(Z,d ) $ and $(Y, d')$ are compact metric spaces and  \{$f_n:X_n\rightarrow Y$ \} is an equicontinuous family where $X_n \subset Z$ are compact.  Let  $X \subset Z$ be a compact subset such that for any $x\in X$ there exists a sequence $x_n\in X_n$ converging to $x$. Then there exists a subsequence $\{f_{n_i}\}$ converging  uniformly to some continuous function $f:X\rightarrow Y$, i.e., for any $\epsilon>0$ there exist $\delta>0$ and $N>0$ such that for any $i \geq N$, $x_{n_i}\in X_{n_i}$ and $x\in X$ with $d(x_{n_i},x)<\delta$, we have $d'(f_{n_i}(x_{n_i}),f(x))<\epsilon$.
\end{theorem}
\begin{proof} Since $X$ is compact,  we can find a countable subset $A \subset X$ such that its closure $\overline{A}=X$. Then for
 any $\epsilon>0$ there exists a finite subset $A_\epsilon\subset A$ such that
$$
X \subset\cup_{a\in A_\epsilon}B_{\epsilon}(a,d).
$$
For any $a\in A$, by the assumption, there are  $a_{n}\in X_n$  such that $\lim_n a_n =a$.
By the standard diagonal method, we find a subsequence of $\{f_n\}$ which, for simplicity,  we may assume is $\{f_{n}\}$ itself, such that $\{f_{n}(a_{n})\}$ converges for all $a \in A$.  Define
$$
f(a)=\lim_{n\rightarrow\infty}f_{n}({a_{n}}).
$$
We first claim that $f: A \to Y$ is uniformly continuous.  Indeed, for any $\epsilon>0$ there exists $\delta>0$ such that if $x,y\in X_n$ with $d(x,y)<\delta$, then $d'(f_n(x),f_n(y))<\epsilon$.  Now if
$a,a'\in A$ with $d(a,a')<\delta/2$,   then  $d(a_{n},a_{n}')<\delta$ for $n$ sufficiently large, and
$$
d'(f(a),f(a'))=\lim_{n\rightarrow\infty}d'( f_{n}(a_{n}),f_{n}(a_{n}'))\leq\epsilon.
$$
Since $f$ is uniformly continuous on $A$,  we can extend $f$ to a uniformly continuous function, still denoted by $f$,  to $X$.  Now to see the uniform convergence of $f_n$ to $f$,  take any $\epsilon>0$. There  exists $\delta>0$ such that

(1) for any $n$ and $x,y\in X_n$ with $d(x,y)\leq \delta$, $d'(f_n(x),f_n(y))<\epsilon/3$; and

(2) for any $x,y\in X$ with $d(x,y)\leq \delta$, $d'(f(x),f(y))<\epsilon/3$. \\

Find $N=N(\epsilon)$ such that for any $n\geq N$ and  any $a\in A_{\delta/3}$, there exists $a_n \in X_n$ such that $d(a_{n},a)<\delta/3$  and $d'(f_{n}(a_{n}),f(a))<\epsilon/3$. This is possible since $A_{\delta/3}$ is a finite set.
Then for any $x\in X$, find an $a\in A_{\delta/3}$ such that $d(x,a)<\delta/3$. If  $n\geq N$ and $x_n\in X_n$ with $d(x,x_n)<\delta/3$,  then
$$
d(x_n,a_{{n}})\leq d(x_n, x)+d(x,a)+d(a,a_{n})\leq \delta/3+\delta/3+\delta/3=\delta
$$
and
\begin{align*}
d'(f_{n}(x_n),f(x))\leq &d'(f_{n}(x_n),f_{n}(a_{n}))+d'(f_{n}(a_{n}),f(a))+d'(f(a),f(x))\\
\leq&\epsilon/3+\epsilon/3+\epsilon/3=\epsilon.
\end{align*}
\end{proof}


\section{Transboundary extremal lengths and a duality theorem}

The transboundary extremal length introduced by O. Schramm (\cite{schramm})  is a powerful conformal invariant and has been used in many works (see \cite{bonk1}, \cite{bonk2} and others).

Suppose  $\Sigma$ is a Riemann surface homeomorphic to an annulus and $\mathcal F$ and $\mathcal F^*$ are two families of curves in $\Sigma$ such that $\mathcal F$ consists of closed curves separating two ends of $\Sigma$ and $\mathcal F^*$ consisting of paths joining different ends of $\Sigma$. Then a well-known duality theorem states that the extremal lengths satisfy $EL(\mathcal F) EL(\mathcal F^*)=1$. The goal of this section is to show that the duality theorem still holds for transboundary extremal lengths. The latter result (Theorem \ref{4373}) is the key tool for estimating the modules of rings on non-smooth convex surfaces.

\subsection{Transboundary extremal lengths}

Suppose $\Sigma$ is a Riemann surface and $E$ is a set of ends of $\Sigma$  ($E$ may not be the set of all ends). Note that $\Sigma \cup E$ is naturally a topological space with the end topology.  Schramm's transboundary extremal length for any family of curves in $\Sigma \cup E$ is defined as follows.
Take a conformal Riemannian metric $g$ on $\Sigma$.
An \emph{extended metric} $m$ on $\Sigma \cup E$ is a pair
$(\rho g,\mu)$ such that $\rho : \Sigma \to \R_{\geq 0}$ is a Borel measurable function and $\mu:  E \to \R_{\geq 0}$.
The \it area \rm of the extended metric $m$ is defined to be
$$
A(m)=\int_{\Sigma} \rho^2 dA_g + \sum_{a\in  E} \mu(a)^2,
$$ where $dA_g$ is the area form of the Riemannian metric $g$.
By a curve in $\Sigma \cup E $ we mean a continuous map  $\gamma$ from an  interval to $\Sigma \cup E$. The \emph{length} of $\gamma$ in the extended metric $m$ is defined to be
$$ l_m(\gamma) = \int_{\Sigma \cap\gamma} \rho ds +\sum_{ a\in  E \cap \gamma } \mu(a),$$
where $ds$ is the length element in the metric $g$.
  If $\Gamma$ is a family of curves in $\Sigma \cup E$, its length in the extended metric $m$ is defined to be $$l_m(\Gamma) =\inf\{ l_m(\gamma) : \gamma \in \Gamma\}.$$
  Schramm's transboundary extremal length \cite{schramm} of $\Gamma$ is
  \begin{equation} \label{el}  EL(\Gamma)=
  \sup_{m}   \frac{l_m(\Gamma)^2}{A(m)}\end{equation}
   where the supremum is over all finite positive area extended metrics $m$.

   We will drop the adjective ``transboundary" when we refer to extremal lengths below.
  Some of the basic properties of extremal lengths follow from the definition (see \cite{ahlfors} for a proof).

  \begin{lemma} \label{71} (a)  Suppose $\Gamma_1$ and $\Gamma_2$ are two families of curves in $\Sigma \cup E$ such that for any $\gamma_1\in\Gamma_1$, there exists $\gamma_2\in\Gamma_2$ satisfying $\gamma_2\subset\gamma_1$. Then $EL(\Gamma_1)\geq EL(\Gamma_2)$. In particular, if $\Gamma_1 \subset \Gamma_2$, then $EL(\Gamma_1)\geq EL(\Gamma_2)$.

  (b) Suppose $\Gamma_1$ and $\Gamma_2$ are two families of curves in $\Sigma \cup E$ such that they are supported in two disjoint Borel measurable subsets $A_1,A_2$ of $\Sigma \cup E$, i.e., for any $\gamma\in\Gamma_1$, $\gamma\subset A_1$  and for any $\gamma\in\Gamma_2$, $\gamma\subset A_2$.
Then
   $$EL(\Gamma_1 \cup \Gamma_2)^{-1} \geq EL(\Gamma_1)^{-1} + EL(\Gamma_2)^{-1}.$$
\end{lemma}

Another property of extremal lengths is the following conformal invariance.

  \begin{lemma} [Schramm {\cite{schramm}} Lemma 1.1]\label{conf-inv} Suppose $\phi: \Sigma \to \Sigma'$  is a conformal diffeomorphism sending $E$ onto $E'$.  Then for any curve family $\Gamma$ in
$\Sigma \cup E$,
 $$EL(\Gamma) =EL(\phi(\Gamma)),$$
where $\phi(\Gamma)=\{\phi(\gamma):\gamma\in\Gamma\}$.
\end{lemma}


We will apply transboundary extremal lengths in the following special situation in this paper.  Take a topological surface $S$ and a compact set $X \subset S$ such that  $\Sigma =S -X$ is connected and is equipped with a complex structure.  Then  each component $X_i$ of $X$ corresponds to an end, denoted by $[X_i]$ of $\Sigma$. Define  $[X]$ to be the set of ends of the form $[X_i]$ for components $X_i$ of $X$. We will apply transboundary extremal lengths to curves in the space $\Sigma \cup [X]$ which will be denoted by $S^X$.

\subsection{A duality theorem}
 A \it flat cylinder \rm is a  Riemannian surface isometric to $S=\mathbf S^1 \times (0, h)$ equipped with the product metric $g=dx^2+dy^2$ where $(e^{\sqrt{-1}x}, y)$ are points in $S$. 
  A \it square \rm in $S$ is a compact subset of the form $I_1 \times I_2$ where $I_1$ and $I_2$ are two closed intervals of the same length. We consider a point in $S$ as a (degenerated) square.  
The following lemma is known to Schramm \cite{schramm1}.

\begin{lemma} \label{sqcase} Suppose $X$ is a finite disjoint union of squares in a flat cylinder $S$.
Let $\Gamma^*$ be the family of curves in $S^X$ joining the two boundary components of $S$ and $\Gamma$ be the family of all simple loops in $S^X$ separating the two boundary components of $S$.  Then $$EL(\Gamma) EL(\Gamma^*)=1.$$ \end{lemma}

\begin{proof}  We will show that $EL(\Gamma)=\frac{2\pi}{h}$ and $EL(\Gamma^*)=\frac{h}{2\pi}$. Since the computations are similar, we only compute $EL(\Gamma^*)$. Let the components of $X$ be $X_1$, ..., $X_n$ of edge lengths $h_1, ..., h_n $ with $h_i \geq 0$.   Let the coordinate in $S$ be  $(e^{\sqrt{-1}x}, y)$.  Construct an extended metric $m=(\rho (dx^2+dy^2), \mu)$ on $S^X $ such that $\rho=1$ on $S-X$  and $\mu([X_j])=h_j$. Then the area $A(m)$ of $m$ is  $2\pi h$. For any curve $\gamma \in \Gamma^*$, we have $l_m(\gamma) \geq h$ by definition. Therefore, $l_m(\Gamma^*) \geq h$. This shows, $EL(\Gamma^*) \geq l_m(\Gamma^*)^2/A(m) \geq \frac{h}{2\pi}$.  To see that $EL(\Gamma^*) \leq \frac{h}{2\pi}$, take any extended metric $m=(\rho (dx^2+dy^2), \mu)$ and for each $e^{\sqrt{-1} t} \in \mathbf S^1$, let $\gamma_t$ be the  line segment $\{e^{\sqrt{-1} t} \} \times (0, h)$ in $S$.

  Then
$$l_m(\Gamma^*) \leq l_m(\gamma_t) =\int_{(S-X)\cap\gamma_t} \rho(e^{\sqrt{-1} t}, y) dy + \sum_{ j: [X_j]  \in \gamma_t} h_j.$$
This shows,
$$2\pi l_m(\Gamma^*) \leq \int_{0}^{2\pi} l_m(\gamma_t) dt =\int_{S-X} \rho(e^{\sqrt{-1} t},y)dydt + \sum_j h^2_j.$$
By Cauchy inequality we have
$$4 \pi^2 l_m(\Gamma^*)^2 \leq (\int_{S-X} \rho^2(e^{\sqrt{-1} t},y) dydt + \sum_{j} h^2_j)( \int_{S-X} dydt+\sum_j h_j^2).
$$
$$=A(m) (2\pi h).$$ This shows $l_m(\Gamma^*)/A(m) \leq \frac{h}{2\pi}$ and the result follows.
\end{proof}

The main tool that enables us to estimate the module of rings in convex surfaces $\partial C(W)$ is the following theorem.  A version of it for quadrilaterals was proved by Schramm \cite{schramm1} (Theorem 10.1).  Recall that a doubly connected Riemann surface is a topological ring with a complex structure.

\begin{theorem}\label{4373} Suppose $R$ is doubly connected Riemann surface without boundary, and $X$ is a compact subset with finitely many components such that $R\backslash X$ is connected. Let $\Gamma^*$ be the family of curves in $R^X$ connecting the two topological ends of $R$ and $\Gamma$ be the family of all simple loops in $R^X$ separating the two topological ends of $R$.  Then the transboundary extremal lengths satisfy
\begin{equation} \label{dual}  EL(\Gamma) EL(\Gamma^*)=1.\end{equation}
\end{theorem}

\begin{proof}  By Lemmas \ref{conf-inv} and \ref{sqcase},   it suffices to prove that there exists 
a finite collection of disjoint squares $X'$ in a flat cylinder $R'$ such that  (1) $R'-X'$ is conformal to  $R-X$ by a conformal map $\phi$ and (2) $\phi: [X'] \to [X]$ is a bijection. 
The latter result was established by Jenkins (the corollary of Theorem 2 in \cite{jenkins}).
\end{proof}


\subsection{Extremal length estimate on planar regions}

The following is a quantified version of a result of Schramm on transboundary extremal length of curves in cofat domains.

\begin{proposition} \label{8.2} Let  $S$ be  a topological annulus in $\hat{\C} -\{0\}$,
$W\subset S$ be a finite disjoint union of round closed disks and points and $\Gamma$ be the family of all simple loops in $S^W$ separating the two ends of $S$.  If $S$ contains $N$ disjoint rings $R_{i} =\{ r_i <|z|< 2r_i\}$ for $i=1,2,..., N$ such that

(a) each ring $R_i$ separates the two ends of $S$ and

(b) no component in $W$ intersects two $R_i$ and $R_j$,

then
\begin{equation}\label{72eq} EL(\Gamma) \leq \frac{216\pi}{N}. \end{equation}


\end{proposition}

\begin{proof}
We begin with the case of $N=1$.

\begin{lemma} \label{690} Assume that $S$, $W$ and $\Gamma$ are as in Proposition \ref{8.2}. If $S$ contains one annulus $\Omega_r:=\{r < |z| < 2r\}$ separating two ends of $S$, then
\begin{equation}\label{72eq1} EL(\Gamma) \leq 72\pi. \end{equation}
\end{lemma}
\begin{proof}
Let $\Gamma^*$ be the family of all simple paths in $S^W$ joining two boundary components of $S$. Then by Theorem \ref{4373}, we have $EL(\Gamma)^{-1}=EL(\Gamma^*)$.  Thus it suffices to show  $EL(\Gamma^*) \geq \frac{1}{72 \pi}$.

Suppose $\{W_1, ..., W_n\}$ is the set of all components of $W$. Consider the extended metric $m=(\rho |dz|, \lambda)$ on $S^W$ where $\lambda([W_i])=diam(W_i \cap \Omega_r)$ is the diameter in the Euclidean metric and $\rho: S-W \to \R $ is the function which is $1$ on $\Omega_r-W$ and zero otherwise.
Let $B_{2r}$ and $B_{6r}$ be the Euclidean balls of radii $2r$ and $6r$ centered at $0$ and $\mu$ be the Lebesgue measure on $\C$.
We have  \begin{equation}\label{7851} diam(W_i \cap \Omega_r)^2 \leq 2 \mu(W_i \cap  B_{6r}). \end{equation} 
Clearly $diam(W_i \cap \Omega_r) \leq diam(W_i \cap B_{2r}).$   If  $W_i \cap B_{2r} =\emptyset$ or
 $W_i \subset B_{6r}$, then either $diam(W_i \cap B_{2r})^2 =0$  or
 $diam(W_i \cap B_{2r})^2 \leq diam(W_i)^2 \leq 2 \mu(W_i) =2 \mu(W_i \cap B_{6r}).$ Hence the result follows.  If $W_i$ is not inside $B_{6r}$ and $W_i$ intersects $B_{2r}$, then the diameter of $W_i$ is at least $4r$ and $W_i \cap B_{6r}$  contains a disk of radius $2r$.  Then  $diam(W_i\cap B_{2r})^2 \leq diam(B_{2r})^2=16r^2 \leq 2 \mu(W_i \cap B_{6r})$ and (\ref{7851}) holds again.

By  (\ref{7851}), the area satisfies $$A(m)=\int_{\Omega_r-W} dxdy +\sum_{i=1}^n (\lambda([W_i]))^2
\leq \mu(B_{6r}-W) + \sum_{i=1}^n (\lambda([W_i]))^2 $$ $$  \leq 2\mu(B_{6r}-W) + 2 \sum_{i=1}^n \mu(W_i \cap B_{6r}) \leq 2 \mu(B_{6r}).$$
This shows
$$ A(m) \leq 72\pi r^2.$$

For each  path $\gamma$ in $\Gamma^*$ joining two boundary components of $S$, we claim that $l_m(\gamma) =\int_{\gamma-W} ds +\sum_{[W_i] \in \gamma } \lambda([W_i]) \geq r$.  To see this, let $W'$ be the set of all components $W_i$
of $W$ which intersect $\Omega_r$.   Construct a path $\tilde{\gamma}$  in $S^{W-W'}$ by gluing to $\gamma \cap (S-W')$ a line segment of length at most $\lambda([W_i])$ inside $W_i$ for which $[W_i] \in \gamma \cap [W']$.
By the construction, the length of $\gamma$ is at least the length of $\tilde{\gamma}$.
Since $\tilde \gamma$ is a path $S^{W-W'}$ joining two ends of $S$, it contains an arc $\alpha \subset \Omega_r$ joining  $\{|z|=r\}$ to $\{|z|=2r\}$. In particular, the Euclidean length of $\alpha$ is at least $r$. This shows the length of $\tilde{\gamma}$ is at least $r$. 
Therefore $l_m(\gamma) \geq r$ and $l_m(\Gamma) \geq {r}$. This shows
 $$ EL(\Gamma^*) \geq \frac{ l_m(\Gamma^*)^2}{A(m)} \geq \frac{1}{72\pi}.$$
\end{proof}

Now back to the proof of
(\ref{72eq}). By Lemma \ref{690}, we may assume that $N \geq 4$. Without loss of generality, we may assume that $2r_j \leq r_{j-1}$ for $j=2,..., N$. For $j=2,3,..., N-1$, let $A_j$ be the annulus in $S$ defined by $\{ 2r_{j+1} < |z| < r_{j-1}\} -\cup_{k\in I_j} W_k$ where the index set $I_j$ is $\{k| W_k \cap (\{|z|=r_{j-1} \}\cup \{|z|=2r_{j+1}\}) \neq \emptyset\}$. Note that $A_j$ is topologically an annulus since each component of $W$ is a disk or a point and no component of $W$ intersects two rings
$R_i$ and $R_l$. Furthermore, by construction $Z_j:=W \cap A_j \subset A_j$ is the union of all components of $W$ which lie in $A_j$. See Figure \ref{777}.
Let $\Gamma_j$ be the set of simple loops in $A_j^{Z_j}$ separating the two ends of $A_j$.  By construction,  $A_j^{Z_j} \subset S^W$ and $\Gamma_j$ is a subset of $\Gamma$. Furthermore, since no component of $W$ intersect two rings $R_i$ and $R_h$, we see that $A_{j} \cap A_{j+2} =\emptyset$.
Then $\Gamma_2$, $\Gamma_4$, $\Gamma_6$, ..., $ \Gamma_{2[N/2]}$ are in $\Gamma$ and lie in disjoint annuli $A_2^{Z_2}$, $A_4^{Z_4}$,  $A_6^{Z_6}, ..., A_{2[N/2]}^{Z_{2[N/2]}}$ respectively.
By Lemma \ref{690} for $S=A_{2j}$ and $W=Z_{2j}$, $EL(\Gamma_{2j}) \leq 72\pi$, i.e., $EL(\Gamma_{2j})^{-1}\geq \frac{1}{72\pi}$.  By Lemma \ref{71} and $N \geq 4$,
$EL(\Gamma)^{-1} \geq EL(\Gamma_2 \cup \Gamma_4 \cup ... \cup \Gamma_{2[N/2]})^{-1} \geq \sum_{i=1}^{[N/2]} EL(\Gamma_{2i})^{-1} \geq \frac{[N/2]}{72\pi} \geq \frac{N}{216\pi}$.  Combining these, we see inequality (\ref{72eq}) follows.

\begin{figure}[ht!]
\begin{center}
\begin{tabular}{c}
\includegraphics[width=0.7\textwidth]{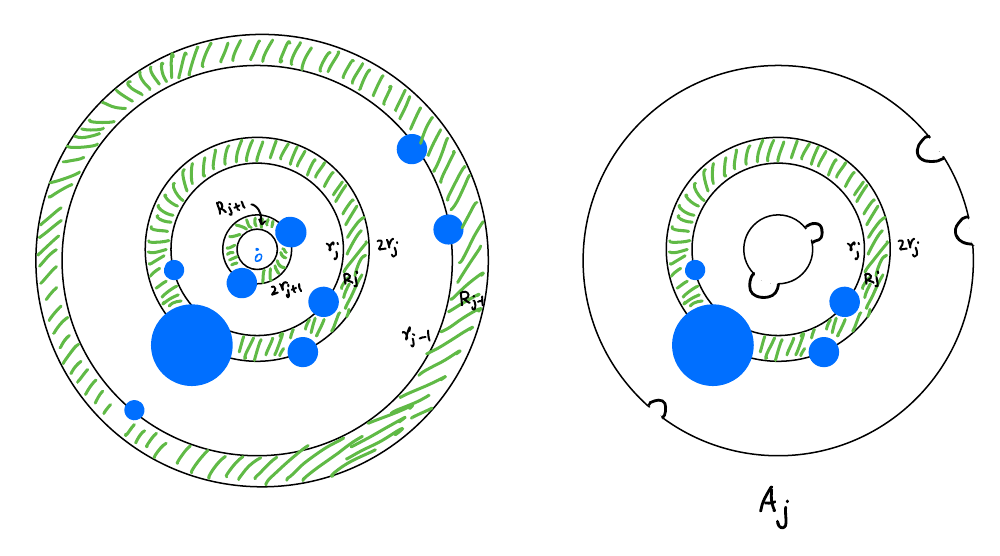}
\end{tabular}
\end{center}
\caption{The rings $A_j$'s.} \label{777}
\end{figure}
\end{proof}

\subsection{Extremal length estimate on convex surfaces $\partial C(Z)$ in the Poincar\'e model $\H^3$}


The counterpart of Proposition \ref{8.2} for non-smooth convex surfaces will be established in this section.  Given a compact set $Z \subset \partial \H^3$,  by the work of Thurston, 
the surface $\partial C_P(Z)$ is hyperbolic and therefore naturally a Riemann surface whose conformal structure is induced by the path metric $d^P_S$. 

Let $W$ be a finite disjoint union of round disks and points in $\mathbf S^2$ such that $(0,0,0) \in \partial C_P(W)$ and $C_P(W)$ has a non-empty interior.  We also fix a point $q \in W$. The goal is to estimate the modules of rings separating $q$ and $(0,0,0)$ in the convex surface $\partial C_P(W)\cup W$.  Let  $D_r=\{z \in \mathbf S^2|  d^{\mathbf S}(z, q) < r\}$ be the ball of radius $r$ at $q$ and $E_r=\{ z\in \mathbf S^2| d^{\mathbf S}(z,q)\geq r\}$ be the complement of $int(D_r)$. The disk $D_r$ is convex if $r < \pi/2$.  
The module of the ring $\{ z \in \mathbf S^2|  a< d^{\mathbf S}(z, q) < b\}$ is $\frac{1}{2\pi} \ln\left(\frac{\tan (b/2)}{\tan(a/2)}\right)$.
To see this, we may assume that $z$ is the south pole of $\SS^2$ and then the stereographic projection conformally maps the ring to a Euclidean ring $\{z\in\C:\tan\frac{a}{2}<|z|<\tan\frac{b}{2}\}$ whose module is well-known to be
$\frac{1}{2\pi} \ln\left(\frac{\tan (b/2)}{\tan(a/2)}\right)$ (see page 53 in \cite{ahlfors0}).
So for small $r>0$, the module of the ring $\Omega_{ r}:=\{ z \in \mathbf S^2|  r< d^{\mathbf S}(z, q) < 65r\} $ is 
$$
\frac{1}{2\pi} \ln\left(\frac{\tan (65r/2)}{\tan(r/2)}\right)\approx\frac{\ln(65)}{2\pi}+O(r^2)
$$ 
which is uniformly bounded away from $0$ and infinity.   The goal is to show that the image of $\Omega_r$ under the shortest distance projection is a ring whose module is uniformly bounded away from zero.

Let $\pi: \mathbf S^2 \to \partial C_P(W)\cup W$ be the shortest distance projection onto the convex hull. Note that $\pi |_W = id$ and $\pi$ is in general not injective.  We claim that for the ring $\Omega_r=D_{65r}-E_r$ in $\mathbf S^2$, there is a well-defined ring in $\partial C_P(W)\cup W$ corresponding to it under the projection $\pi$.
Indeed, by Propositions \ref{homeo} and \ref{3.7}, we see that 
$\pi(\overline{D_r})$ and $\pi(E_{65r})$ are disjoint connected compact sets in the topological 2-sphere $\partial C_P(W)\cup W$. By a well-known fact from surface topology, there exists a unique component of $\partial C_P(W) \cup W -\pi(\overline{D_r}) \cup \pi(E_{65r})$ which is a topological annulus separating 
$\pi(\overline{D_r})$ and $\pi(E_{65r})$. (All other components, if they exist,  are simply connected).  We will denote this ring component by $\hat{\pi}(\Omega_{ r})$.  Since the projection $\pi$ is onto, we see that $\pi(\Omega_r)$ contains $\hat{\pi}(\Omega_r)$. 

\begin{proposition} \label{6571}
Let $W$ be a finite disjoint union of round disks and points in $\mathbf S^2$ such that $(0,0,0) \in \partial C_P(W)$ and $C_P(W)$ has a non-empty interior. Fix a point $q\in W$. Suppose $W'$ is a finite union of connected components of $W$ and $S$ is an annulus in $\partial C_P(W)  \cup W -\{q\}$ such that $S$ contains $W'$. Let $\Gamma$ be the set of all simple loops in $S^{W'}$ separating the two ends of $S$.  If $S$ contains $N$ pairwise disjoint rings $\hat{\pi}(\Omega_{r_i}) =\hat{\pi}(\{ z \in \mathbf S^2 | r_i < d^{\mathbf S}(z, q) < 65 r_i\})$ for $i=1,2,..., N$ such that

(a) each $\hat{\pi}(\Omega_{r_i})$ separates the two ends of $S$, 

(b) no component of $W$ intersects two $\pi(\Omega_{r_i})$ and 
$\pi(\Omega_{r_j})$, and

(c) $65r_i<\pi/2$ for all $i=1,2,...,N$,\\
then $$ EL(\Gamma) \leq 3\frac{ 10^8}{N}.$$
\end{proposition}

 The key step in the proof is to establish the counterpart of Lemma \ref{690}.  The rest will be the same as the argument used in the proof of Proposition \ref{8.2}.





\begin{figure}[ht!]
\begin{center}
\begin{tabular}{c}
\includegraphics[width=0.35\textwidth]{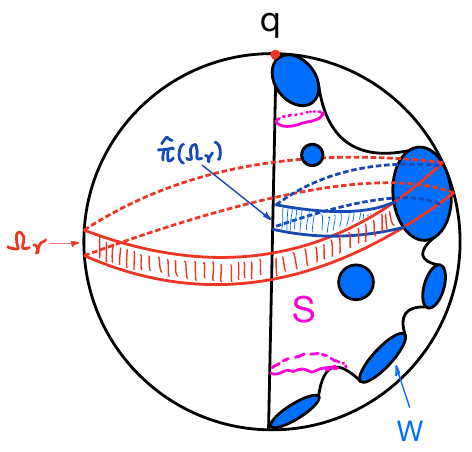}
\end{tabular}
\end{center}
\caption{The annulus surface $S$ in $\partial C_P(W) \cup W$.} \label{92}
\end{figure}

\begin{proof} 
We begin with the case of $N=1$. 

\begin{lemma} \label{67991} Assume that $W$, $W'$, $S$ and $\Gamma$ are as in Proposition \ref{6571}. If $S$ contains a ring $\hat{\pi}( \Omega_{r})$ separating two ends of $S$, then 
$$EL(\Gamma) \leq 10^8.$$
\end{lemma}


\begin{proof}
Let $\Gamma^*$ be the family of paths in $S^{W'}$ joining two boundary components of $S$. By Theorem \ref{4373}, we have $EL(\Gamma)^{-1}=EL(\Gamma^*)$. Hence it suffices to prove
$EL(\Gamma^*) \geq \frac{1}{10^8}$.  To this end,
consider the extended metric $m=(\rho d^E_{S}, \lambda)$ on $S^{W'}$ where $\rho(z)=1$ for $z \in \pi(D_{65r})\cap S -W'$ and is zero otherwise,  and for each component $W_i$ of $W'$,  $\lambda([W_i])=diam(W_i \cap D_{65r})$ is the spherical diameter of $W_i\cap D_{65r}$ in $(\mathbf S^2,d^{\mathbf S})$. 
Recall that $\mu(S)$ denotes the Euclidean area, or the 2-dim Hausdorff measure, of a set $S$ in $\R^3$.
We claim that  \begin{equation}\label{9124} \lambda([W_i])^2 = diam(W_i \cap D_{65r})^2 \leq 4\pi \mu(W_i \cap  D_{195r}).
\end{equation} 
To see this, we will use the well-known fact that for a spherical ball $B_t$ of radius $t < \pi$, $ diam(B_t)^2 \leq 4\pi \mu(B_t)$ and $\mu(B_t) =4\pi\sin^2(t/2) \leq \pi t^2.$
Now if $W_i \cap D_{65r} =\emptyset$ or
 $W_i \subset D_{195r}$, then either $diam(W_i \cap D_{65r})=0$ or $diam(W_i \cap D_{65r})^2   \leq diam(W_i)^2 \leq 4\pi \mu(W_i) = 4\pi \mu(W_i \cap D_{195r}).$  Thus (\ref{9124}) holds.   If $W_i$ is not inside $D_{195r}$ and $W_i$ intersects $D_{65r}$, then $W_i \cap D_{195r}$  contains a spherical ball  of radius $65r$.  Hence  $diam(W_i\cap D_{65r})^2 \leq diam(D_{65r})^2 \leq    4\pi \mu(D_{65r}) \leq  4\pi \mu(W_i \cap D_{195r})$ and (\ref{9124}) holds again.

The area $A(m)=\int_{\pi(D_{65r})\cap (S-W')} d\mu +\sum_{i=1}^n (\lambda([W_i]))^2 \leq \mu(\pi(D_{65r})-W')+\sum_{i=1}^n (\lambda([W_i]))^2$. By (\ref{lip2}) that $\pi$ is 2-Lipschitz,  $\pi(D_{65r})-W'=\pi(D_{65r}-W')$ and (\ref{9124}), we have
$$
A(m)\leq 4\mu(D_{65r}-W') +4\pi \sum_{i=1}^n \mu(W_i\cap D_{195r}) $$
$$
\leq 
4\pi \mu(D_{195r}-W') +4\pi \sum_{i=1}^n \mu(W_i\cap D_{195r}) $$
$$
\leq
4 \pi \mu(D_{195r})\leq   4\pi\cdot\pi195^2 r^2. $$
For each path $\gamma$ in $\Gamma^*$ joining the two boundary components of $S^{W'}$,  let $\tilde \gamma$ be the path on $S$ obtained by gluing to $\gamma \cap (S-W')$ the shortest geodesic path in each component $W_i$ of $W'$ such that $[W_i] \in \gamma \cap [W']$ and $W_i \cap D_{65r} \neq \emptyset$. 
By the separation assumption, the path $\tilde \gamma$ contains an arc joining a point $p_1 =\pi(p_1')$ with $p_1'\in  D_r$ to another point $p_2 =\pi(p_2')$ with $p_2' \in E_{65r}$.   By Proposition \ref{3.7},  
$$ |q-p_1| \leq 8 d^{\mathbf S}(q, p_1') \leq 8r$$ and
$$ |q-p_2| \geq \frac{ d^{\mathbf S}(q, p_2')}{8} \geq \frac{65r}{8}.$$
In particular, $|p_1-p_2| \geq |q-p_2| -|q-p_1| \geq \frac{r}{8}$. 
Now by construction,  $l_m(\gamma)
= \int_{\tilde \gamma-W'} ds +\sum_{W_i\cap  \tilde \gamma \cap D_{65r} \neq \emptyset } \lambda([W_i]) \geq l_m(\tilde \gamma)  $. The later is at least $|p_1-p_2| \geq \frac{r}{8}$ since it contains $p_1$ and $p_2$. 
Therefore, we have
$l_m(\gamma) \geq \frac{r}{8}$  for all $\gamma \in \Gamma^*$ and hence $l_m(\Gamma^*) \geq  \frac{r}{8}.$
This implies $$ EL(\Gamma^*) \geq \frac{ l_m(\Gamma^*)^2}{A(m)} \geq \frac{r^2/8^2}{4\pi\cdot\pi195^2r^2} \geq \frac{1}{10^8}. $$
\end{proof}

The rest of the proof of Proposition \ref{6571} is the same as that of Proposition \ref{8.2}.  Note that the counterpart of the annulus $A_j$ is the annulus component of the $S - \pi(D_{r_j+1}) \cup \pi(E_{r_{j-1}})\cup \cup_{k\in I_j} W_k$ where the index set $I_j$ is $\{ k | W_k \cap (\pi(D_{r_j+1}) \cup \pi(E_{r_{j-1}})) \neq \emptyset\}$. The verification that $A_{j}$ and $A_{j+2}$ are disjoint and there are no components of $W$ intersecting both $A_j$ and $A_{j+2}$ follows from the assumption on $W$. 
We omit the details.
\end{proof}

\section{Proof of part (a) of  Theorem \ref{1.222} }

Suppose $U=\hat{\C}-X$ is a circle domain such that $|X| \geq 3$  and $d^U$ is the Poincar\'e metric on $U$. The goal is to find a circle type closed set $Y \subset\mathbf S^2$ such that $(U, d_{U})$ is isometric to $\partial C_K(Y) \subset (\H^3, d^K)$ in the induced path metric from $d^K$.  Here $(\H^3, d^K)$ is the Klein model of the hyperbolic space, and $C_K(Y)$ is the convex hull in the Klein model.

Using a M\"obius transformation, we may assume that   $X \subset \{ z \in \C| 2 < |z| < 3\}$.  Let $W_1, ..., W_i, ...$ be a sequence of distinct connected components of $X$ such that $\cup_{i=1}^{\infty} W_i$ is dense in $X$. Each disk component of $X$ is in the sequence.
Let $X^{(n)}=\cup_{i=1}^n W_i^{(n)}$ where  $W^{(n)}_i=W_i$ if $W_i$ is a disk and $W^{(n)}_i=\{ z \in \C| d^E(z, W_i) \leq l_n\}$ if $W_i$ consists of a single point.  We make $l_n$ small such that $l_n$ decreases to $0$ and $W^{(n)}_i \cap W^{(n)}_j =\emptyset $ for $i \neq j$ and $W^{(n)}_i \subset \{ z \in \C| 2 < |z| < 3\}$ .  This shows that  $X^{(n)}$ is a circle-type closed set having finitely many connected components and
$X^{(n)}$ converges to $X$ in the Hausdorff distance in $\hat \C$.
 
We will use the following theorem of Schlenker \cite{sch1}. 

\begin{theorem}[Schlenker]\label{sh11} Suppose $g$ is a complete hyperbolic metric on a genus zero surface $\Sigma$ of finite topological type such that each end of $(\Sigma, g)$ is of funnel type. Then there exists a closed set $Y \subset \SS^2$ such that $Y$ is a disjoint union of finitely many round disks and $\partial C_K(Y)$ with the induced path metric from $d^K$ is isometric to $(\Sigma, g)$.    
\end{theorem}

Let $d_n$ be the Poincar\'e metric on $U_n=\hat \C- X^{(n)}$.
By the above theorem, there exists a circle type closed set $Y^{(n)} \subset\mathbf S^2$ with $(0,0,0) \in \partial C_K(Y^{(n)})$ and an isometry $\phi_n: (U_n , d_n) \to \partial C_K(Y^{(n)})$ such that $\phi_n(0)=(0,0,0)$.
By taking a subsequence if necessary, we may assume that $\{Y^{(n)}\}$ converges in Hausdorff distance to a compact set $Y$. 

A key step in the proof Theorem \ref{1.222} is to show the following equicontinuity property.

 \begin{theorem}\label{equicont}  The sequence $\{\phi_n: (U_n, d^{\SS}) \to (\partial C_K(Y^{(n)}), d^E)\}$ is equicontinuous, 
 i.e., for any $\epsilon>0$, there exists a $\delta>0$ such that for all $n$ and all $x,y \in U_n$ with $d^{\SS}(x,y) < \delta$,  $|\phi_n(x)-\phi_n(y)| < \epsilon$.
 \end{theorem}

This section is organized as follows.  In \S7.1, we prove that  $Y$ contains at least three points and is not equal to $\mathbf S^2$. In \S7.2, we prove Theorem \ref{equicont}.  In \S7.3, we show, using Theorem \ref{equicont},  that $\phi_n$'s can be extended to continuous functions, still denoted by $\phi_n$ from  $(\hat{\C}, d^{\mathbf S})$ to $(\partial C_K(Y^{(n)}) \cup Y^{(n)}, d^E)$ and the extended family $\{\phi_n\}$ remains equicontinuous. Finally, in \S7.4,
 we prove that each component of $Y$ is the Hausdorff limit of components of $Y^{(n)}$. Therefore $Y$ is a circle type closed set. On the other hand, using Theorems \ref{convergence}, \ref{poincare} and \ref{aa}, we know that $\phi_n$ converges uniformly on compact subsets to an isometry from $(U, d^U)$ and $\partial C_P(Y)$.  This ends the proof Theorem \ref{1.222}(a).

\subsection{The Hausdorff limit $Y$ of  the sequence \{$Y^{(n)}$\} }

Recall that the injectivity radius at a point $p$ of a hyperbolic surface is the largest positive number $r$ such that the open ball $B_r(p)$  of radius $r$ centered at $p$ is isometric to the standard hyperbolic ball of radius $r$.  In particular, $B_r(p)$ simply connected.

\begin{lemma}\label{3pts} (a) There exists a positive lower bound on the injectivity radii of $(U_n, d_n)$ at $0$.

(b) The  closed set $Y$ contains at least three points and $Y \neq\mathbf S^2$.

(c)  There exists a positive lower bound on the Euclidean diameters of $\phi_n(\D)$.
\end{lemma}

\begin{proof}
To see part (a), take three points $\{p_1, p_2, p_3\}$ in $X^{(n)}$ for $n$ large and let $d^W$ be the Poincar\'e metric on $W=\hat \C-\{p_1, p_2, p_3\}$. Since $U_n \subset W$, the Schwarz-Pick lemma shows that $d_n \geq d^W$ and $B_r(0, d_n) \subset B_r(0, d^W)$ for any $r>0$. Let $r_0>0$ be a small positive number such that $r_0$ is less than the injectivity radius of $d^W$ at $0$ and $B_{r_0}(0, d^W)$ is contained in  $U_n$ for all $n$.  Then $B_{r_0}(0, d_n)$ is contained in the simply connected subset $B_{r_0}(0, d^W)$ in $U_n$. Therefore,  $B_{r_0}(0, d_n)$ can be isometrically lifted to the universal cover of $(U_n, d_n)$. This implies $B_{r_0}(0, d_n)$ is isometric to the standard radius $r_0$ ball in the hyperbolic plane.  Therefore, $r_0$ is a lower bound for the injectivity radius of $(U_n, d_n)$ at $0$ for all $n$.

To see part (b), since $Y^{(n)}$ converges to $Y$ and $(0,0,0) \in \partial C_K(Y^{(n)})$, we see that $(0,0,0) \in \partial C_K(Y)$ and hence $Y$ contains at least two points and $Y \neq\mathbf S^2$ . Now if $Y$ contains only two points, say $Y=\{p_1, p_2\}$, then there exists a sequence of positive numbers $r_n \to 0$ such that
$Y^{(n)} \subset B_{r_n}(p_1) \cup B_{r_n}(p_2)$ in $\mathbf S^2$ where $B_r(p)$ is the ball of radius $r$ centered at $p$ in the spherical metric. Therefore, there exists a sequence of homotopically non-trivial loops $\gamma_n \subset \partial C_K(Y^{(n)})$ through $(0,0,0)$ whose hyperbolic lengths tend to $0$.  The loops $\gamma_n$ can be constructed as follows. Let $V_n$ be the hyperbolic surface $\partial C_K(B_{r_n}(p_1) \cup B_{r_n}(p_2))$, $q_n \in  V_n$ be a point converging to $(0,0,0)$ such that the shortest distance projection from $V_n$ to $\partial C_K(Y^{(n)}))$ sends $q_n$ to $(0,0,0)$. Note that $V_n$ is topologically a ring and there is only one simple closed geodesic separating the two ends of $V_n$.  Let $\delta_n$ be the simple geodesic loop in $V_n$ based at $q_n$ such that (1) $\delta_n$ separates the two ends of $V_n$ and (2) away from the based point $q_n$,  $\delta - \{q_n\}$ is a geodesic. Let
$\gamma_n$ be the image of $\delta_n$ under the shortest distance projection from $V_n$ to $\partial C_K(Y^{(n)})$. 
Since $r_n$ tends to zero, the length of the curve $\delta_n$ tends to zero. On the other hand, the shortest distance projection decreases the distances. Therefore the length of $\gamma_n$ tends to zero.  Finally, since $\delta_n$ separates the two ends of $V_n$, we claim that $\gamma_n$ is homotopically non-trivial, i.e., essential,  in $\partial C_K(Y^{(n)})$.  To this end, let us assume without loss of generality that $C_K(Y^{(n)})$ is 3-dimensional. Let $M=V_n \cup  B_{r_n}(p_1) \cup B_{r_n}(p_2)$ be the topological 2-sphere and $p_i^{(n)}$ be a point in $B_{r_n}(p_i)\cap Y^{(n)}$,  for $i=1,2$. Take a point in the interior of $C_K(Y^{(n)})$ and let $\Theta : \partial C_K(Y^{(n)}) \cup Y^{(n)} \to M$ be the radial projection map induced from the point.  Note that $\Theta$ is a homeomorphism map from a topological 2-sphere to the topological 2-sphere $M$, is the identity map on $Y^{(n)}$, and $\Theta^{-1}(Y^{(n)})=Y^{(n)}$. We will show that  $\Theta(\gamma_n)$ is essential on $M-\{p_1^{(n)}, p_2^{(n)}\}$ which implies that $\gamma_n$ is essential. Let $\pi: M \to  \partial C_K(Y^{(n)}) \cup Y^{(n)}$ be the shortest distance projection.  
Since $\delta_n$ separates $p_1^{(n)}$ from $p_2^{(n)}$ in the topological 2-sphere $M$ and $\Theta \circ \pi: M \to M$ is a degree one continuous map (by Lemma \ref{radialprj}) such that  $\Theta \circ \pi$ sends the annulus $M-\{p_1^{(n)},p_2^{(n)}\}$ to itself, we see that $\Theta (\gamma_n) = \Theta(\pi( \delta_n))$ separates $p_1^{(n)}$ from $p_2^{(n)}$. Thus the claim follows.

Using the isometry $\phi_n$, we see that the homotopically non-trivial loops $\gamma'_n=\phi_n^{-1}(\gamma_n)$ in $(U_n, d_n)$ pass through $0$ such that their lengths in the Poincar\'e metrics $d_n$ tend to zero, i.e.,  $l_{d_n}(\gamma_n') \to 0$.  But this contradicts part (a) that the injectivity radii of $d_n$ at $0$ are bounded away from $0$.

To see part (c), choose the radius $r_0$ in the proof of part (a) to be small such that $B_{r_0}(0, d^W) \subset \D$. Then $B_{r_0}(0, d_n) \subset B_{r_0}(0, d^W) \subset \D$.    Let $Z_n$ be the surface $\phi_n(B_{r_0}(0, d_n))$ which is contained in $\phi_n(\D)$.
Part (c) follows by showing that
$$ \liminf_n diam_{d^E}(Z_n) >0.$$
We will prove a stronger result that  $4\pi  diam_{d^E}(Z_n)^2 \geq Area_{d^E}(Z_n)$ and the Euclidean area $Area_{d^E}(Z_n)$ of $Z_n$ is bounded away from $0$. Note that $Z_n$ lies in a (Euclidean) convex surface $\partial C_K(Y^{(n)})$ and $Z_n$ is contained a Euclidean ball $B_1$ of radius $diam_{d^E}(Z_n)$. 
It is known that if a convex closed surface $A$  (e.g., $A=\partial B_1$)  contains a convex closed surface $A'= \partial (C_K(Y^{(n)}) \cap B_1$),  then $Area_{d^E}(A) \geq Area_{d^E}(A')$. Indeed, by Proposition 2.1.3 \cite{bertsekas} and Lemma 1.2.3 in \cite{schneider}, the shortest distance projection map $P: A \to A'$ is distance decreasing and is surjective. Therefore, $P$ decreases the area. 
This shows $4\pi diam_{d^E}(Z_n)^2 \geq Area_{d^E}(A')$.  But $Area_{d^E}(A') \geq Area_{d^E}(Z_n)$ due to $Z_n \subset A'$.  Thus 
 $4\pi  diam_{d^E}(Z_n)^2 \geq Area_{d^E}(Z_n)$ holds. To obtain the lower bound on $Area_{d^E}(Z_n)$, by construction,  $Z_n \subset \partial C_K(Y^{(n)})$ with the induced metric from $d^K$ is isometric to the standard hyperbolic disk of radius $r_0$ in $\H^2$. In particular,  the hyperbolic area of $Z_n$ is $4\pi \sinh^2(r_0/2)$. On the other hand, since $\phi_n$ is an isometry, we see that $Z_n \subset B_{r_0}(0, d^K)$ in $\H^3_K$. Hence, there is a compact set $Q \subset \H^3$, which contains all $Z_n$ for $n$ large. On the compact set $Q$, there exists a constant $C_1$ such that $d^E \leq d^K \leq C_1 d^E$. 
Therefore, $Area_{d^E}(Z_n) \geq C_1^{-2} Area_{d^K}(Z_n) = 4\pi C_1^{-2} \sinh^2(r_0/2)$ and the result follows.
\end{proof}

\subsection{Proof of Theorem \ref{equicont}}
Let $\Sigma_n= \partial C_K(Y^{(n)})$ and $\tilde\Sigma_n$ be $\partial C_K(Y^{(n)}) \cup Y^{(n)}$ which is a topological 2-sphere.
By the work of Thurston,  the surface  $\Sigma_n$ is 
naturally a  Riemann surface. 
 The conformal map $\phi_n: U_n \to \Sigma_n$ implies that   $\hat \C ^{ X^{(n)}}$
and $\tilde \Sigma_n^{ Y^{(n)}}$ are conformally equivalent. 

We  prove Theorem \ref{equicont} by contradiction. 
Suppose otherwise that there exists  $\epsilon_0>0$ and sequences $x_{k_n}, x_{k_n}' \in U_{k_n}$  such that
$d^{\SS}(x_{k_n}, x_{k_n}') \to 0$ and $|\phi_{k_n}(x_{k_n})-\phi_{k_n}(x_{k_n}')| \geq \epsilon_0$.  If the sequence $\{k_n\}$ is not bounded,  we may assume without loss of generality (after taking a subsequence), that $k_n=n$.  If the sequence $\{k_n\}$ is bounded, we may assume after taking a subsequence that $k_n$ is a constant.  Below we will focus on the main case that $k_n=n$. The same proof also works for the simpler case that $k_n$ is a constant. We omit the details. 

\begin{remark} \label{remark1}
The case $k_n$ being a constant is equivalent to the statement  that each  $\phi_n: (U_n, d^{\SS})  \to (\Sigma_n, d^E)$ is uniformly continuous.  We can see uniform continuity by using  Caratheodory's extension theorem.
By compositing $\phi_n$ with the inverse of the homeomorphism $\Psi(X) =\frac{2x}{1+|x|^2}: \overline{\H^3_P} \to \overline{H^3_K}$, it suffices to show that any conformal map $\rho:(U_n, d_n) \to \partial C_P(Y^{(n)})$ can be extended continuously to their compact closures in $\R^3$. Let us recall  
Caratheodory's extension theorem.
Suppose $A_1$ and $A_2$ are two Riemann surfaces and $B_i$ is a subsurface of $A_i$ bounded by finitely many disjoint Jordan curves $c_i$ in $A_i$ for $i=1,2$. Then Caratheodory's extension theorem says that any biholomorphism from $B_1$ to $B_2$ extends continuously to $B_1\cup c_1$ to $B_2 \cup c_2$. The standard form of Carathéodory's theorem applies to Jordan domains $B_1$ and $B_2$. 
However, the proof of this theorem is inherently local, relying on the standard length-area estimate, which allows it to hold in the more general context described above. 
Now in our case, we take $A_1 =\hat{\C}$, $B_1=U_n$, $B_2=(\partial C_P(Y^{(n)}), d^P_{ \partial C_P(Y^{(n)})})   $ and $A_2$ to be the metric double of the bounded curvature surface $(\partial C_P(Y^{(n)}), d_{\partial C_P(Y^{(n)})}^E)$ across its boundary. By the gluing theorem of Alexandrov-Zalgaller (Theorem 8.3.1 in \cite{res}), the metric double $A_2$ is again a surface of bounded curvature. Thus $A_2$ is a Riemann surface containing $B_2$. Since the conformal structures on $\partial C_P(Y^{(n)})$ induced by $d^P$ and $d^E$ are the same, we see that $B_2$ is conformally embedded in $A_2$ whose boundary consists of Jordan curves. By the Caratheodory extension theorem, we see that $\rho$ extends  to a continuous map from the compact closure $\overline{U_n}$ to $\overline{\partial C_P(Y^{(n)})}$. Therefore, $\rho$ is uniformly continuous. 
\end{remark}

Going forward, we'll assume 
that  $d^{\SS}(x_n, x_n') \to 0$ and $|\phi_n(x_n) -\phi_n(x_n')| \geq \epsilon_0$. 
By taking a subsequence if necessary, we may further assume that $x_n, x_n' \to p \in \hat{\C}$.



\begin{figure}[ht!]
\begin{center}
\begin{tabular}{c}
\includegraphics[width=0.6\textwidth]{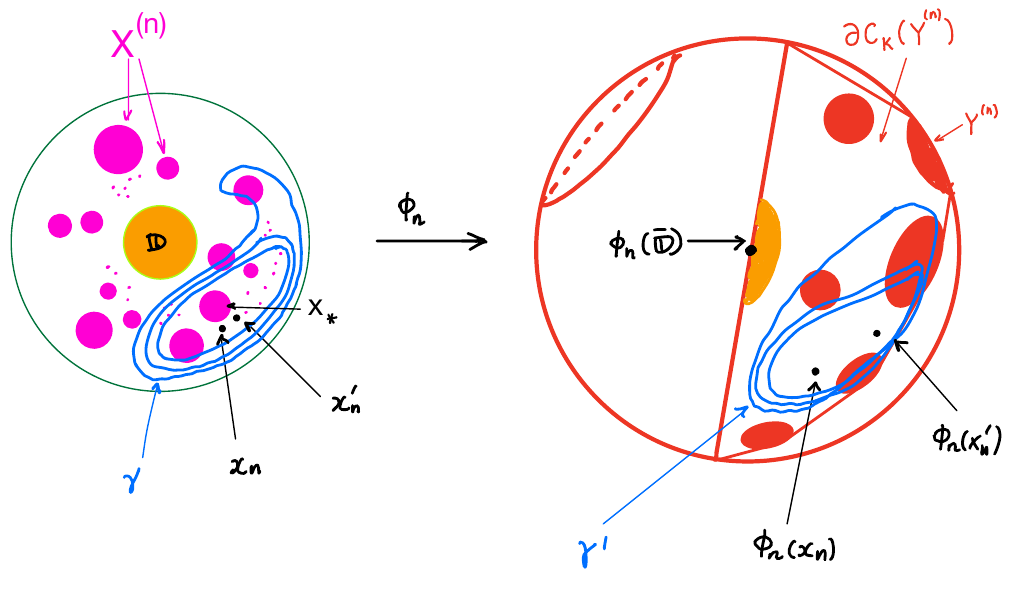}
\end{tabular}
\end{center}
\caption{Equicontinuity and $X^{n}$, $X_*$ and $Y^{(n)}$
.} \label{71111}
\end{figure}

\begin{lemma}\label{53} The limit point $p$ is in the set $X$.
\end{lemma}

\begin{proof}
Suppose otherwise that $p \notin X$, i.e., $p \in U$. Then by the Hausdorff convergence, there exists an open connected neighborhood $W$ of $p$ such that $W \subset U_n$ for $n$ large. 
Let $d^W$ be the Poincare\'e metric on $W$. 
Then the Schwarz-Pick lemma shows $d_n(x,y) \leq d^W(x,y)$ for all $x, y \in W$.  But we also have $d^W(x_n, x_n') \to 0$. Therefore $d_n(x_n, x_n') \to 0$.
By (\ref{dist2}) that $d^K(x,y) \geq d^E(x,y)$ for all $x,y \in \H^3_K$, we have
\begin{align*}
&d_n(x_n, x_n') =d^K_{\Sigma_n}(\phi_n(x_n), \phi_n(x_n'))\\
\geq  &d^E_{\Sigma_n}(\phi_n(x_n), \phi_n(x_n'))) \geq  |\phi_n(x_n)- \phi_n(x_n')|  \geq \epsilon_0.
\end{align*}
This is contradictory to $d_n(x_n, x_n') \to 0$.
\end{proof}


By Lemma \ref{53} and $x_n \in U_n$, we see that $p \in X \cap \partial U$. Let $X_*$ be the connected component of $X$ which contains $p$. Due to the normalization condition on $U_n$, the closed unit disk $\overline{\D}$ is contained in $U$ and $U_n$ for all $n$. 

Recall that if $Z$ is a compact subset of a surface $S$, we define $S^Z$ to be $(S-Z)\cup [Z]$ which is the surface $S-Z$ by adding the ends of $S-Z$ corresponding to connected components of $Z$ in the end topology. 
Construct two families of paths $\Gamma_n$ and $\Gamma_n'$ as follows.
If $X_*$ is a single point, the family
$ \Gamma_n$ is defined to be the set of simple loops in $\hat \C^{X^{(n)}}$ separating $\{x_n,x_n'\}$ and $\D$ and
$\Gamma_n' $ is defined to be the set of simple loops in $\tilde\Sigma_n^{Y^{(n)}}$ separating $\{\phi_n(x_n), \phi_n(x_n')\}$ and $\phi_n(\D)$.  If $X_*$ is a round disk, then  by the construction of $X^{(n)}$, $X_*$ is a connected component of $X^{(n)}$ for $n$ sufficiently large. Therefore $p \in \partial X^{(n)}$ for $n$ large. We define
$ \Gamma_n$  as the union of two families: simple loops $\gamma_1$ and simple arcs $\gamma_2$. Here $\gamma_1$ are simple loops in $\hat \C^{X^{(n)}}-\{[X_*]\}$ such that $\gamma_1$ separate $\{x_n, x_n'\}$ and $\D$  in $\C^{X^{(n)}}$ and $\gamma_2$ are simple arcs in $\hat \C^{X^{(n)}}-\{[X_*]\}$ with end points in $X_*$ such that the simple loops $\gamma_2 \cup [X_*]$ separate $\{x_n, x_n'\}$ and $\D$  in $\C^{X^{(n)}}$. We define 
$\Gamma_n'$ to be the image of $\Gamma_n$ under $\phi_n$ in  $\tilde \Sigma_n^{Y^{(n)}}$. More precisely, $\Gamma_n'$ consists of two families of simple loops $\gamma_1'$ and simple arcs $\gamma_2'$ such that $\gamma_1'$ are in  $\tilde\Sigma_n^{Y^{(n)}}-\{[\phi_n(X_*)]\}$ and separate $\{\phi_n(x_n), \phi_n(x_n')\}$ and $\phi_n(\D)$ and $\gamma_2'$ are simple arcs in  $\tilde\Sigma_n^{Y^{(n)}}-\{[\phi_n(X_*)]\}$ with end points in $Y_*^{(n)}$ and separate $\{\phi_n(x_n), \phi_n(x_n')\}$ and $\phi_n(\D)$  in $\tilde\Sigma_n^{Y^{(n)}}$.  Here $Y_*^{(n)}$ is the component of $Y^{(n)}$ corresponding to $X_*$ under $\phi_n$.

 The conformal invariance of extremal length implies that $EL(\Gamma_n) =EL(\Gamma_n')$.
   We will derive a contradiction by showing that $\liminf_n EL(\Gamma_n') >0$ and $\lim EL(\Gamma_n)=0$.

\subsubsection{Extremal length estimate I:  $\lim_n EL(\Gamma_n)=0$     }

Recall that by the normalization condition $X^{(n)} \subset \{z \in \C| 2 < |z| < 3\}$ and $X_*$ is the component of $X$ containing $p =\lim_n x_n$.

\begin{lemma}\label{6517} 
For any $r>0$, there exist $r'<r/2$ and $N$ such that for $n >N$, no component of  $X^{(n)}-X_*$ intersects both $\{ z\in \C|   |z-x_n|=r\}$ and $\{ z\in \C|   |z-x_n|=r'\}$.
\end{lemma}

\begin{proof} 
Suppose otherwise, there exists a sequence of components $Z_n$ of $X^{(k_n)}-X_*$ such that $Z_n$ intersects both $\{ z\in \C|   |z-x_{k_n}|=r\}$   and $\{ z\in \C|   |z-x_{k_n}|=\frac{1}{n}\}$.  By taking a subsequence if necessary, we may assume that $Z_n$ converges in Hausdorff distance to a disk $Z'$ of positive diameter which intersects  $\{ z\in \C|   |z-x_{k_n}|=r\}$  and contains $p$.  Since the sequence \{$X^{(n)}$\} converges in Hausdorff distance to $X$, there exists a component $X_j$ of $X$ such that $Z' \subset X_j$. Since $p \in Z'$, we have $p\in X_j$. On the other hand, $X_*$ is the component of $X$ containing $p$. Therefore $X_j=X_*$. This shows that $Z' \subset X_*$. This is impossible if $X_*$ is a one-point component of $X$. Hence  $X_*$ is a disk component of $X$. Since $Z_n$ and $X_*$ are different components of $X^{(k_n)}$,  we see the distance from the center of  $Z_n$ to $X_*$ is bounded away from zero. This shows that the center of $Z'$ is outside of $X_*$ and contradicts $Z'
 \subset X_*$.
\end{proof}

%
By Lemma \ref{6517} and $x_n, x_{n}' \to p$, we construct a sequence of positive numbers $\{r_i\}$  and a sequence of integers $\{N_i\}$ increasing to infinity
such that

(1) $r_{i+1} <r_i/2$ and $r_1< \frac{1}{2}$,

(2) if $n \geq N_i$, each connected component of $X^{(n)}-X_*$  intersects at most one of annuli $\{ z\in \C| r_j<  |z-x_{n}|< 2r_j\}$
for $j=1,2, ...,i $,

(3) if $n \geq N_i$, $|x_n -x_n'| \leq r_i/2$, and

(4) if $n\geq N_i$ and  $X_*$ is a disk, then $X_*\cap   \{ z\in \C| r_j<  |z-x_{n}|< 2r_j\} \neq\emptyset$, for all $j=1,2,..., i.$

 For each $n$, construct an annulus $S_n$ as follows.  Since for any round disk $B=\{ z \in\C| |z-a| <r \}$, the intersection $B \cap (\C -X^{(n)})$ is path connected, we can join $x_n$ to $x_{n}'$ by a path $\alpha_n$ in the ball $\{z \in \C|  |z-x_n| <2 |x_n-x_{n}'|\}$ such that $\alpha_n \cap X^{(n)} =\emptyset$.  Let the annulus $S_n$ be $\hat{\C} -(\{|z| \leq 1\} \cup \alpha_n)$.

For any large $i$, take $n \geq N_i$. If $X_*$ consists of one point,  the disjoint annuli $\{ r_j<|z -x_n| < 2r_j\}$ for $j=1,2,...,i$ in $S_n$ satisfy conditions in Proposition \ref{8.2} where $W = X^{(n)}\cap S_n$. Therefore, by (\ref{72eq}), $EL(\Gamma_n^*) \leq \frac{216\pi}{i}$ where $\Gamma_n^*$ is the set of all simple loops in $S_n^{X^{(n)}}$ separating the two ends of $S_n$. On the other hand by construction, $\Gamma_n$ consists of all simple loops in $\hat{\C}^{X^{(n)}}$   separating $ \{|z| \leq 1\}$ from $\{x_n, x_{n}'\}$. Therefore $\Gamma_n^* \subset \Gamma_n$.  By monotonicity of extremal lengths, we have $EL(\Gamma_n) \leq EL(\Gamma_n^*)$. It follows that $EL(\Gamma_n) \leq \frac{216\pi}{i}$ as long as $n \geq N_i$. Thus $\lim_n EL(\Gamma_n)=0$.

If $X_*$ is a disk, by construction, $X_*$ is a component of $X^{(n)}$ for $n$ large.  The disjoint annuli $\{ r_j<|z -x_n| <2r_j\}$ for $j=1,2,...,i$ in $S_n$ satisfy conditions in Proposition \ref{8.2} where $W = X^{(n)}\cap S_n-X_*$. Therefore, by (\ref{72eq}), $EL(\Gamma_n^*) \leq \frac{216\pi}{i}$ where $\Gamma_n^*$ is the set of all simple loops $\alpha$ in $S_n^{W}$ separating the two ends of $S_n$.
By definition, $\Gamma_n$ consists of all simple arcs $\beta$ in $\hat{\C}^{X^{(n)}}$ such that $\beta$ or $\beta \cup [X_*]$ are simple loops   separating $ \{|z| \leq 1\}$ from $\{x_n, x_{n}'\}$. By condition (4) of the choices of $r_j$'s, we see that each simple loop $\alpha$ in $\Gamma^*_n$ contains an arc $\beta$ which is in $\Gamma_n$.  By Lemma \ref{71}(a) on the monotonicity of the extremal lengths, we have $EL(\Gamma_n) \leq EL(\Gamma_n^*) \leq \frac{216\pi}{i}$ when $n > N_i$. Therefore, $\lim_n EL(\Gamma_n)=0$.

\subsubsection{ Extremal length estimate II: $\liminf_n EL(\Gamma_n') >0$}
We use the work of  Reshetnyak \cite{res} on conformal geometry of surfaces of bounded curvature to justify some of computations in this subsection. See Appendix \S11 for details.
Suppose a convex surface $S$ is $\partial C_P(Z)$ for a compact set $Z \subset \partial \H^3$. Then Thurston's theorem says $(S, d^P_S)$ is hyperbolic and therefore naturally a Riemann surface.  However, the surface $S$ may not be smooth in $\H^3$. Now consider the two conformally equivalent Riemannian metrics $d^E$ and $d^P$ on $\H^3$.  If $S$ is smooth, then clearly the induced Riemannian metrics $d^E_S$ and $d^P_S$ are conformally equivalent on $S$.  In our case, the surface $S$ may well be non-smooth and the induced path metric $d^E_S$ may not be Riemannian. But $(S, d^E_S)$ is a surface of bounded curvature.  Reshetnyak's work implies that $d^E_S$ and $d^P_S$ are conformally equivalent in the sense of Ahlfors-Beurling definition in extremal lengths. More precisely, the path metric $d^E_S$ can be written as $\lambda d^P_S$ for some non-negative Borel measureable function $\lambda$ on $S$.
In particular,
for a convex surface $\Sigma =\partial C_K(Z)$ in the Klein model $(\H^3, d^K)$, the induced path metrics $d^K_{\Sigma}$ and $d'_{\Sigma}$ on $\Sigma$ are conformal where $d'(x, y) =|\Psi^{-1}(x)-\Psi^{-1}(y)|$ with $\Psi(x)=\frac{2x}{1+|x|^2}$. 
By Proposition \ref{lip} and Theorem \ref{area},  $d'(x,y) \geq \frac{1}{2}|x-y|$ and the area of a convex surface $\Sigma$ in $d'$ metric is at most $16\pi$.
 Define an extended metric $m_n$ on $\tilde \Sigma_n^{ Y^{(n)}}$ to be the pair $(d'_{\Sigma_n}, \nu_n)$ where $d'_{\Sigma_n}$ is the induced path metric from $d'$ on $\Sigma_n$ and $\nu_n$ on a connected component of $Y^{(n)}$ is the spherical diameter of the component.
   Hence the area of $m_n$ is uniformly bounded from above by $16\pi+\frac{\pi}{2}\cdot4\pi^2<100$ since the square of the diameter of a spherical disk is at most $\pi/2$ times its area.   Since
$EL(\Gamma_n') \geq l_{m_n}(\Gamma_n')^2/A(m_n) \geq \frac{1}{100} l_{m_n}(\Gamma_n')^2$,
 it suffices to show that $\liminf_n l_{m_n}(\Gamma_n') >0$.  We prove $\liminf_n l_{m_n}(\Gamma_n') >0$ by replacing the
 metric $d'$ in $m_n$ by the Euclidean metric $d^E$. Let $m'_n=(d^E|_{\Sigma_n}, \nu_n)$ be the extended metric on $\tilde \Sigma_n^{Y^{(n)}}$. Then due to $d'(x,y) \geq \frac{1}{2}|x-y|$,
 $\liminf_n l_{m_n}(\Gamma_n') \geq \frac{1}{2} \liminf_n l_{m_n'}(\Gamma_n')$  and the result follows by showing the following

 \begin{equation} \label{78.1}
 \liminf_n l_{m_n'}(\Gamma_n') >0. \end{equation}

Note that by assumption $(0,0,0)\in \partial C_K( Y^{(n)})$. It follows that each component $Y'$ of $Y^{(n)}$ is a spherical ball of radius at most $\pi/2$ and hence is convex on $\mathbf S^2$. We now prove (\ref{78.1}) by contradiction.
Suppose otherwise that $\liminf_n l_{m_n'}(\Gamma_n') =0$. After taking a subsequence, we may assume that there exists a sequence of simple loops $\gamma_n \in \Gamma_n'$ such that $l_{m_n'}(\gamma_n) \to 0$. For each $\gamma_n$, construct a new path
$\tilde \gamma_n$ obtained by gluing to $\gamma_n -[Y^{(n)}]$ the shortest spherical geodesic segment $\delta$ in each component $Y^{(n)}_i$ of $Y^{(n)}$ for which $[Y_i^{(n)}] \in \gamma_n \cap [Y^{(n)}]$ such that the end points of $\delta$ are the end points of $\gamma_n -Y^{(n)}$.  By the construction of $m_n'$, we have $l_{m_n'}(\gamma_n) \geq l_E(\tilde \gamma_n)$ where $l_E(\beta)$ is the Euclidean length of a path $\beta$.   It follows that $\lim_n l_E(\tilde \gamma_n) =0$.

If $X_*$ is a single point, by construction, $\tilde \gamma_n$ is a simple loop in $\tilde \Sigma_n$ separating two compact sets
$A_n:=\phi_n(\overline{\D})$ and $B_n=\{\phi_n(x_n), \phi_n(x_n')\}$. By Lemma \ref{3pts} and the assumption on $B_n$, both Euclidean diameters of $A_n$ and $B_n$ are bounded away from $0$.  After taking a subsequence, we may assume that $A_n$ and $B_n$ converge in Hausdorff distances to two compact sets $A$ and $B$ of positive Euclidean diameter and $\tilde \gamma_n$ converges uniformly to $\tilde \gamma$ in
$\partial C_K(Y) \cup Y$ as Lipschitz maps. 
Then $\gamma$ seprates $A,B$ in $\partial C_K(Y)\cup Y$ and
by the well known fact on path metrics, $l_E(\tilde \gamma) \leq \liminf_n l_E(\tilde \gamma_n)$. 
Therefore $l_E(\tilde \gamma)=0$, i.e. $\tilde \gamma$ is a single point. However, a single point cannot separate $A,B$ in $\partial C_K(Y)\cup Y$, which is a topological sphere since $Y$ contains at least 3 points by Lemma \ref{3pts}.

If $X_*$ is a disk component of $X$, then by construction $X_*$ is a component of $X^{(n)}$ for $n$ large.  Let $Y^{(n)}_*$ be the corresponding component of $Y^{(n)}$. 
If $\tilde\gamma_n$ is a simple arc ending at $[Y_*^{(n)}]$, 
construct a new loop $\gamma_n^*$ by gluing to $\tilde\gamma_n$ the shortest geodesic segment $\beta_n$ in $\mathbf S^2$ such that $\partial \beta_n =\partial  \tilde{\gamma_n}$. Since $Y_*^{(n)}$ is convex in $\SS^2$  $\beta_n \subset Y_*^{(n)}$. If $\tilde\gamma_n$ is already a closed loop, just let $\gamma_n^*$ be $\tilde\gamma_n$.
By the construction $\gamma^*_n$ separates $A_n$ from $B_n$ in $\tilde \Sigma_n$.  
We claim that $\lim l_E(\gamma_n^*)=0$ and therefore reduce this case to the case just proved above.  To this end, let $\beta_n'$ be the Euclidean line segment having the same endpoints as $\beta_n$.  Since $\beta_n$ has length at most $diam_{d^{\mathbf S}}(Y_*^{(n)}) \leq \pi$, we have $l_E(\beta_n') \geq \frac{2}{\pi} l_{E}(\beta_n).$ On the other hand, $l_E(\beta_n') \leq l_E(\tilde \gamma_n)$ since they have the same end points. It follows that
$$ l_E(\gamma_n^*) =l_E(\tilde \gamma_n) +l_E(\beta_n) \leq (1+\frac{\pi}{2}) l_E(\tilde \gamma_n).$$ Therefore $\lim_n l_E(\gamma_n^*)=0$.



\subsection{Extension of $\phi_n$ to the Riemann sphere}
Let $\Sigma_n=\partial C_K(Y^{(n)})$. By Theorem \ref{equicont},  each map $\phi_n$ is uniformly continuous and hence can be extended continuously to a continuous map, still denoted by $\phi_n$, from $(\overline{U_n}, d^{\mathbf S}) $ to $(\overline{\Sigma_n}, d^{E})$. Here $\bar Z$ is the closure of a set $Z$ in $\R^3$ or $\hat \C$. Furthermore, the
 family of the extended maps $\{\phi_n: (\overline{U_n}, d^{\mathbf S}) \to (\overline{\Sigma_n}, d^{E}) \}$ is again equicontinuous.

Our next goal is to extend $\phi_n$ continuously to a map from $\hat{\C} =U_n \cup X^{(n)}$ to $ \Sigma_n \cup Y^{(n)}$ such that the extended family is still equicontinuous with respect to the spherical metric on $\hat \C$ and the Euclidean metric. 
In the spherical metric $d^{\SS} =\frac{2|dz|}{1+|z|^2}$ on the Riemann sphere $\hat{\C}$, all Euclidean disks and half spaces are spherical closed balls. 
We extend each homeomorphism $\phi_n$ to a continuous map from $(\hat{\C}, d^{\SS})$ to $\Sigma_n \cup Y^{(n)}$ by coning from the centers of disks. More precisely, let $D_r =\{ z \in \C| |z| \leq r\}$ and $S_r^1 =\partial D_r$ be the disk of radius $r$ and its boundary in $\C$. Given any homeomorphism
$f: S^1_r \to S^1_R$, its \it Euclidean central extension \rm $F: D_r \to D_R$ is the homeomorphism defined by the formula
\begin{equation}\label{cent} F(z) =\frac{|z|}{r} f(\frac{r z}{|z|}). \end{equation}
For a round disk $W=B_{r}(p, d^{\mathbf S})$ in the 2-sphere $\mathbf S^2 \subset \R^3$ of radius $r \leq \pi/2$, let $\hat{W}$ be the 2-dimensional Euclidean disk $\hat{W} \subset \R^3$ such that $\partial \hat W=\partial W$. The projection $\rho: \hat W \to W$ from $-p$ sends each point $x \in \hat{W}$ to the intersection of the ray from $-p$ to $x$ with $\mathbf S^2$. It is a bi-Lipschitz homeomorphism whose bi-Lipschitz constant is at most $\pi$.
We extend a homeomorphism
$f$ from the boundary of a spherical disk to the boundary of a spherical disk by the formula $\hat{f}=\rho_1 \circ F \circ \rho_2^{-1}$ where $F$ is the central extension to the Euclidean disk and $\rho_1$ and $\rho_2$ are bi-Lipschitz homeomorphisms produced above. For simplicity, we still call $\hat{f}$ the \it central extension \rm of $f$ with respect to the spherical metrics.
Take a disk component $Z$ of $X^{(n)}$. Then $\phi_n(\partial Z)$ is the boundary of a disk component $Z'$ of $Y^{(n)}$. Both $Z$ and $Z'$ have spherical radii at most $\pi/2$ by the normalization condition that $X^{(n)} \subset \{ z \in \C| 2 < |z| < 3\} $ and $(0,0,0) \in \partial C_K(Y^{(n)})$. 
Extending $\phi_n$ to $Z$ by the spherical central extension produces a homeomorphism, 
still denote it by $\phi_n$ which is now defined on $\hat{\C}$ with image in $\Sigma_n \cup Y^{(n)}$.
Proposition \ref{512} below shows that the family of extended continuous maps $\{ \phi_n: (\hat{\C}, d^{\SS}) \to (\Sigma_n \cup Y^{(n)}, d^E)\}$ are equicontinuous. It is proved in two steps. In the first step, we show that spherical central extensions of functions in an equicontinuous family of maps between circles form an equicontinuous family.
Due to the bi-Lipschitz property of projections $\rho$'s, it suffices to show that
the Euclidean central extensions of members of an equicontinuous family of maps between circles are still equicontinuous. This is in Proposition \ref{equi2}.
In the second step, we show that the extended maps $\{\phi_n\}$ on $(\hat{\C}, d^{\SS})$ are equicontinuous.






Suppose $g:(X,d) \to (Y,d') $ is a map between two non-empty metric spaces. Its modulus of continuity function is
$\omega(g, \delta) =\sup\{ d'(g(x_1),g(x_2))| d(x_1, x_2) \leq \delta\}$. By definition, $d'(g(x),g(y)) \leq \omega(g, d(x,y))$. If $\delta_1 > \delta_2$, then $\omega(g, \delta_1) \geq \omega(g, \delta_2)$. Also, if $g: X \to Y$ is a surjective map from $(X,d)$ to $(Y,d')$, then
\begin{equation}\label{equi1} \omega(g, diam_d(X)) =diam_{d'}(Y)
\end{equation}
where $diam_d(A)$ is the diameter of a set $A$ in a metric space $(X,d)$. We consider the standard 2-dim Euclidean metric in the following proposition.

\begin{proposition} \label{equi2}Suppose $F: (D_r,d^E) \to (D_R,d^E)$ is the central extension of a continuous map $f: (\partial D_r,d^E_{\partial D_r}) \to (\partial D_R,d^E)$ given by (\ref{cent}). Here $d^E$ denotes the standard 2-dim Euclidean metric and $d^E_{\partial D_r}$ denotes the natural length metric on $\partial D_r$ induced by $d^E$.
Then $$\omega(F, \delta) \leq 4 \omega(f,  \pi\delta).$$

\end{proposition}

\begin{proof}  The proof is based on several lemmas. 


\begin{lemma} \label{equi5} Given $r_1, r_2 \geq 0$, $| r_1 e^{\sqrt{-1} \theta_1} -r_2 e^{\sqrt{-1} \theta_2}| \geq \frac{ r_1}{2}
|e^{\sqrt{-1} \theta_1} - e^{\sqrt{-1} \theta_2}|.$
\end{lemma}
\begin{proof}
We divide the proof into two cases. In the first case $\cos(\theta_2-\theta_1) \leq 0$. Then
$|r_1 e^{\sqrt{-1}\theta_1}-r_2e^{\sqrt{-1}\theta_2}| =|r_1 -r_2e^{\sqrt{-1}(\theta_2-\theta_1)}|
\geq Re(r_1 -r_2e^{\sqrt{-1}(\theta_2-\theta_1)})$
$=r_1 -r_2\cos(\theta_2-\theta_1) \geq r_1 \geq \frac{ r_1}{2}|e^{\sqrt{-1} \theta_1} - e^{\sqrt{-1} \theta_2}| .$ Here $Re(z)$ is the real part of a complex number $z$.
In the second case $\cos(\theta_2-\theta_1)\geq 0$. Then $|sin(\theta_2-\theta_1)| \geq \frac{ 1}{2} |e^{\sqrt{-1} \theta_1} - e^{\sqrt{-1} \theta_2}|$. Therefore,
$|r_1 e^{\sqrt{-1}\theta_1}-r_2e^{\sqrt{-1}\theta_2}| =|r_1e^{\sqrt{-1}(\theta_1-\theta_2)}-r_2|
\geq Im( r_1e^{\sqrt{-1}(\theta_1-\theta_2)}-r_2)$ $=r_1|\sin(\theta_1-\theta_2)| \geq \frac{ r_1}{2} |e^{\sqrt{-1} \theta_1} - e^{\sqrt{-1} \theta_2}|.$
Here $Im(z)$ is the imaginary part of a complex number $z$.
\end{proof}

\begin{lemma}\label{equi3} If $x \in (0,1]$ is a real number, then $x \omega(f, \delta) \leq 2 \omega(f, x \delta).$
\end{lemma}
\begin{proof}
It follows from the definition and triangle inequality that if $k \in \mathbb Z_{>0}$ is a natural number, then
\begin{equation}\label{equi4}  \omega(f, \frac{\delta}{k})  \geq  \frac{1}{k} \omega(f, \delta).
\end{equation}
Therefore, for $x \in (0,1]$,
$$\omega(f, x \delta) \geq \omega(f, \frac{\delta}{[\frac{1}{x}]+1}) \geq \frac{1}{[\frac{1}{x}]+1} \omega(f, \delta) \geq \frac{1}{2/x} \omega(f, \delta)
=\frac{x}{2}\omega(f, \delta).$$
\end{proof}

Now we prove Proposition \ref{equi2}. Assume that $r_1e^{\sqrt{-1}\theta_1},r_2e^{\sqrt{-1}\theta_2}$ are two points in $D_r$ such that 
$|r_1e^{\sqrt{-1}\theta_1}-r_2e^{\sqrt{-1}\theta_2}| \leq  \delta$.  This implies $|r_1-r_2| \leq \delta$ since $|r_1-r_2| \leq |r_1e^{\sqrt{-1}\theta_1}-r_2e^{\sqrt{-1}\theta_2}|.$ Also, by Lemma \ref{equi5}, 
$$
r_1|e^{\sqrt{-1}\theta_1}- e^{\sqrt{-1}\theta_2}| \leq 2 \delta.
$$
Then by Lemmas \ref{equi3}, \ref{equi5},  (\ref{equi4}) and (\ref{equi1}), we have
$$|F(r_1e^{\sqrt{-1}\theta_1})-F(r_2e^{\sqrt{-1}\theta_2})|$$
$$\leq |F(r_1e^{\sqrt{-1}\theta_1})-F(r_1e^{\sqrt{-1}\theta_2})| + |F(r_1e^{\sqrt{-1}\theta_2})-F(r_2e^{\sqrt{-1}\theta_2})|$$
$$ =\frac{r_1}{r} |f(re^{\sqrt{-1}\theta_1})-f(re^{\sqrt{-1}\theta_2})| +\frac{|r_1-r_2|}{r}| f(re^{\sqrt{-1}\theta_2})|$$
$$ 
\leq \frac{r_1}{r} 
\omega\left(f,  d^E_{\partial D_r}( re^{\sqrt{-1}\theta_1}, r e^{\sqrt{-1}\theta_2})      \right )
+\frac{|r_1-r_2|}{r} R$$
$$      \leq  \frac{r_1}{r} 
\omega\left(f,  \frac{\pi}{2}| re^{\sqrt{-1}\theta_1}- r e^{\sqrt{-1}\theta_2}|      \right)
+\frac{|r_1-r_2|}{r} \omega(f, \pi r)                       $$
$$     \leq 2 
\omega\left(f,  \frac{\pi}{2}| r_1e^{\sqrt{-1}\theta_1}- r_1 e^{\sqrt{-1}\theta_2}|      \right)
+ 2 \omega(f, \pi|r_1-r_2|)    $$
$$      \leq 2 \omega(f,  \pi \delta) +2 \omega(f, \pi\delta)    $$
$$\leq 4 \omega(f, \pi\delta).$$
\end{proof}

As a consequence, we have

\begin{corollary} \label{5.6}  (a)  If $\{ f_n: S_{r_n}^1 \to S_{R_n}^1: n \in \mathbb N \}$ is an equicontinuous family of maps between circles, then the central extensions $\{ F_n: D_{r_n} \to D_{R_n}| n \in \mathbb Z_{>0}\}$ are equicontinuous.

(b) If $\{f_n: \partial W_n \to \partial  W^*_n:n \in \mathbb N  \}$ is an equicontinuous family of maps between boundaries of spherical balls $W_n$ and $W_N^*$ of radii at most $\pi/2$, then the central extensions $\{\hat{f_n}:  W_n \to W^*_n\}$  are equicontinuous.
\end{corollary}


\begin{proposition}\label{512} (a) The family of extended homeomorphisms $\{ \phi_n:( \hat{\C}, d^{\mathbf S}) \to (\Sigma_n \cup Y^{(n)}, d^E)\}$ is equicontinuous. In particular, there exists a subsequence of $\{\phi_n\}$ converging uniformly to a continuous function $\phi: (\hat \C, d^{\mathbf S}) \to \partial C_K(Y) \cup Y$.

(b) The limit function $\phi$ is a surjective map from $\hat{\C}$ onto  $\partial C_K(Y) \cup Y$ such that $\phi(X)=Y$ and  $\phi|_U$ is an isometry from
$(U, d^U)$ to $\partial C_K(Y)$.  In particular, $\partial C_K(Y)$ is connected.
\end{proposition}

\begin{proof} To see (a), by the normalization conditions that $(0,0,0) \in \partial C_K(Y^{(n)})$ and $X^{(n)} \subset \{z \in \C: 1<|z|<2\}$, each component of $Y^{(n)}$ and $X^{(n)}$ has radius at most $\pi/2$ in $d^{\mathbf S}$.  Thus Corollary \ref{5.6}(b) applies.  Take any $\epsilon >0$, by Theorem \ref{equicont} and Corollary \ref{5.6}, there exists $\delta>0$  such that if $d^{\SS}(x,y) \leq \delta$  and either (i)  $x, y \in \overline{\hat{\C}-X^{(n)}} $ or (ii) $x, y$ are both in a connected component of $ X^{(n)}$, then $|f_n(x)-f_n(y)|\leq \epsilon$.  It remains to prove the cases where  the pair $x, y$ with $d^{\SS}(x,y) \leq \delta$ satisfy that (1)
one of them is  in $X^{(n)}$ and the other is in $\hat{\C}-X^{(n)} $, or  (2) $x, y$ are in different connected components of $X^{(n)}$. Consider a shortest geodesic $\gamma$ joining $x$ to $y$ in $(\hat\C,d^{\mathbf S})$.
In the first case (1), we may assume that $x \in  X^{(n)}$ and $y\in \hat{\C}-X^{(n)} $. Let $z$ be an intersection point of $\gamma \cap \partial X^{(n)}$ such that $x$ and $z$ are in the same component of $X^{(n)}$. Then $d^{\SS}(x,z) \leq \delta$ and $d^{\SS}(z,y) \leq \delta$. Therefore,  $|f_n(x) -f_n(y)| \leq |f_n(x)-f_n(z)|+|f_n(z)-f_n(y)| \leq 2 \epsilon$.  In the second case (2)  that $x, y$ are in different components of
$X^{(n)}$, then the geodesic segment $\gamma$ contains a point $z \notin X^{(n)}$ with $d^{\SS}(x,z) \leq \delta$ and $d^{\SS}(z,y) \leq \delta$. Then by the  case just proved, $|f_n(x) -f_n(y)| \leq |f_n(x)-f_n(z)|+|f_n(z)-f_n(y)| \leq 4 \epsilon$.
In all cases, we have established the equicontinuity.  Now by the Azela-Ascoli theorem applied to \{$\phi_n: (\hat \C, d^{\mathbf S}) \to (\R^3, d^E)$\}, we see that it has a subsequence converging uniformly to a continuous function $\phi: (\hat \C, d^{\mathbf S}) \to (\R^3, d^E)$.  By the construction of $\phi_n$, we see that $\phi(\hat \C) \subset \partial C_K(Y) \cup Y$. This proves part (a).

To see part (b),  we  begin with,

\begin{lemma} \label{456} Let  $X_n$ (resp. $X_n'$) be compact subsets of a metric space $W$ (resp. $W'$) such that $X_n$ (resp. $X_n'$) converges in Hausdorff distance to a compact subspace $X$ (resp. $X'$). Suppose $f_n: X_n \to X_n'$ is a sequence of continuous functions converging uniformly to $f: X \to X'$.
If $A_n \subset X_n$ is a sequence of compact sets converging in Hausdorff metric to a compact set  $A$, then $f_n(A_n)$ converges in Hausdorff metric to $f(A)$.  In particular, if $f_n$ is onto for all $n$, $f$ is onto.
\end{lemma}
\begin{proof} 
Take any point $f(a)$ in $f(A)$ with $a\in A$. By definition that $A_n$ converges in Hausdorff distance to $A$, there exists $p_n \in A_n$ such that $p_n \to a$. By uniform convergence, $f(a) =\lim_n f_n(p_n)$ where $f_n(p_n) \in f_n(A_n)$.  Next, suppose $f_{n_i}(p_{n_i})$ is a converging sequence whose limit is $q$.  By taking a subsequence if necessary, we may assume that $p_{n_i} \to a \in A$. Therefore, by uniform convergence, $q =\lim_i f_{n_i}(p_{n_i}) = f(a) \in f(A)$.  This shows that $\{f_n(A_n)\}$ Alexandrov converges to $f(A)$.   Since $f_n(A_n)$ and $f(A)$ are compact,  we see $f_n(A_n)$ converges in Hausdorff distance to $f(A)$.
\end{proof}

 Lemma \ref{456} implies that $\phi(\hat{\C})=\partial C_K(Y) \cup Y$  and $\phi(X)=Y$ since $\phi_n$ and $\phi_n|_{X^{(n)}}$ are onto maps.
By the work of Alexandrov and convergence of Poincar\'e metrics (Theorems \ref{convergence} and Theorem \ref{poincare}), we see that the restriction $\phi|_{U}$ is an isometry from $(U, d^U)$ into the component of  $\partial C_K(Y) $ which contains $(0,0,0)$.
Using $\phi(U \cup X)=\partial C_K(Y) \cup Y$,  $\phi(X)=Y$ and $\phi(U) \subset \partial C_K(Y)$, we see that $\phi(U) =\partial C_K(Y)$. In particular, the map $\phi|_U$ is an isometry from $(U, d^U)$ onto $\partial C_K(Y)$. Since $U$ is connected, we see $\partial C_K(Y)$ is connected.




\end{proof}

\subsection{Finishing the proof of part (a) of Theorem \ref{1.222}}

Now take a connected component $Y_k$ of $Y$. To show it is a round disk or a point, we use 
Proposition \ref{512} to find  $Y_k$ =$\phi(X_k)$ for some connected component $X_k$ of $X$.  Indeed since $\phi$ is onto, there exists a connected component, say $X_k$ of $X$, which is mapped by $\phi$ into $Y_k$.  Since $\phi|_U$ is a homeomorphism from $U$ to $\partial C_K(Y)$ and  $\phi|_U$ induces bijection on spaces of ends (see \cite{richards}),  we have $\phi(X_k) =Y_k$.   Since $X$ is a circle type closed set, there exists a sequence $X^{(n)}_{k_n}$ of components of $X^{(n)}$ converging in Hausdorff distance to $X_k$.  By the uniform convergence of $\phi_n$ to $\phi$ and Lemma \ref{456},
$\phi_n(X^{(n)}_{k_n})$ converges in Hausdorff distance to $\phi(X_k)=Y_k$.  But  each $\phi_n(X^{(n)}_{k_n})$ is a component of $Y^{(n)}$ which is a round disk or a point. Therefore $Y_k$ is a round disk or a point.

\section{Proof of  part (b) of Theorem \ref{1.222} assuming equicontinuity}
Recall Theorem \ref{1.222} (b) states that for any circle type close set $Y \subset \SS ^2$ with $|Y| \geq 3$, there exists a circle domain $U=\hat{\C}-X$ with Poincar\'e metric $d^U$ such that the boundary of the hyperbolic convex hull $\partial C_P(Y)$ is isometric to $(U, d^U)$.  The basic
strategy of the proof is the same as that in Theorem \ref{1.222} (a). 
 In this section, we prove Theorem \ref{1.222} (b) using an equicontinuity property which will be established in \S9.

We will use the Poincar\'e ball model $(\H^3,d^P)$ of the hyperbolic 3-space in the rest of the section unless mentioned otherwise.  By Theorem \ref{2.1}, we may assume that the set $Y$ is not contained in any circle, i.e., $C_P(Y)$ is  3-dimensional.
Composing with a M\"obius transformation of $\partial \H^3$, we may assume that $(0,0,0) \in \partial C_P(Y)$.  By Carath\'eodory's theorem on convex hull (Proposition B.6 in \cite{bertsekas}), there exist four components $Y_1, .., Y_4$ of $Y$ such that $(0,0,0) \in C_P(\cup_{i=1}^4 Y_i)$. Since $Y$ is a circle type closed set, there exists a sequence $\{Y_1, ..., Y_n, ...\}$ of components of $Y$ such that $\cup_{i=1}^{\infty}Y_i$ is dense in $Y$.  Note that the density implies each disk component of $Y$ is in the sequence. Define $Y^{(n)}=\cup_{i=1}^n Y_i$ and $\Sigma_n =\partial C_P(Y^{(n)})$. By construction $(0,0,0) \in \partial C_P(Y^{(n)})$ for $n \geq 4$ and the sequence $\{Y^{(n)}\}$ converges in Hausdorff distance to $Y$ in $\overline{\H^3}$.

Since  $\Sigma_n$ is a genus zero Riemann surface of a finite topological type, by Koebe's circle domain theorem, there exists a circle domain $  U_n =\hat{\C}-X^{(n)}$ and a conformal diffeomorphism $\phi_n:\Sigma_n \to   U_n$. Using  M\"obius transformations, we normalize $  U_n$ such that
$0\in   U_n$, $\phi_n(0,0,0)=0$ and the open unit disk $\D$ is a maximum disk contained in $U_n$, i.e., $X^{(n)} \subset \hat \C-\D$ and $X^{(n)} \cap \partial \D \neq \emptyset$.
 By taking a subsequence if necessary, we may assume that $X^{(n)}$ converges in Hausdorff distance to a compact set $X$ in $\hat \C$ such that  $X \subset \hat \C-\D$ and $X \cap \partial \D \neq \emptyset$.     Our goal is to prove that $U=\hat \C -X$ is a circle domain and $(U, d^U)$ is isometric to $\partial C_P(Y)$           


\begin{lemma} \label{8.111}
The hyperbolic injectivity radii of the surfaces $(\Sigma_n, d_{\Sigma_n}^P) $ at (0,0,0) are bounded away from zero.
\end{lemma}
\begin{proof} Let $S_n=\partial C_K(Y^{(n)})$ be the corresponding surface in the Klein model and $d^K_{S_n}$ be the induced path metric on $S_n$.  
The isometry $\Psi(x)=\frac{2x}{1+|x|^2}$ from $d^P$ to $d^K$ induces an isometry from $(\Sigma_n, d_{\Sigma_n}^P) $ to $(S_n, d^K_{S_n})$. Since $\Psi(0)=0$, we will prove the result for $(S_n, d^K_{S_n})$ at $(0,0,0)$.
In the Klein model, both $C_K(Y^{(n)})$ and $S_n$ are a Euclidean convex body and an Euclidean convex surface, respectively.  
By the assumption that $Y$ is not in a circle, the convex set $C_K(Y)$ is 3-dimensional and contains an Euclidean ball. It follows that there exists an Euclidean ball $B$ which is contained in $C_K(Y^{(n)})$ for all large $n$.  This implies that the Euclidean injectivity radii of $S_n=\partial C_K(Y^{(n)})$ in the path metric $d^E_{S_n}$ at $(0,0,0)$ are bounded away from zero. Indeed, if otherwise, we find a sequence of homotopically non-trivial loops $\delta_{n_k}$ based at $(0,0,0)$ in $\partial C_K(Y^{(n_k)})$ such that the lengths of $\delta_{n_k}$  tend to zero. Then the 3-dimensional convex bodies $C_K(Y^{(n_k)})$ will converge in Hausdorff distance to a 1-dimensional convex set. This contradicts the fact that $C_K(Y^{(n_k)})$ contains $B$ for all $k$. Since   $d^K \geq d^E$ and $B_r(p, d^K) \subset B_r(p, d^E)$, 
the injectivity radii of $(S_n, d^K_{S_n})$ are bounded away from $0$.
\end{proof}

\begin{lemma} \label{822}  The closed set $X$ contains at least three points, i.e., $|X| \geq 3$.
Furthermore,  for any $r>0$, there exists $r'>0$ such that the spherical ball $B_{r'}(0, d^{\SS })$ is contained in $B_r(0, d_n)$
in $U_n$ for all $n$ where $d_n =a_n(z)|dz|$ is the Poincar\'e metric on $U_n$
 \end{lemma}
\begin{proof}  
If $X$ contains at most 2 points, say $X\subset \{a,b\}\subset \hat\C -\D$. Then for any $\epsilon>0$, $X^{(n)}\subset B_{\epsilon}(a, d^{\SS }) \cup B_{\epsilon}(b, d^{\SS })$ for sufficiently large $n$. We  claim that
\begin{equation} \label{limi} \lim_na_n(0)=0. \end{equation}
Assume $f$ is the M\"obius transformation on $\hat\C$ such that $f(a)=0,f(b)=\infty, f(0)=1$, and $b_n(z)|dz|$ is the Poincar\'e metric on $f(U_n)=\hat\C-f(X^{(n)})$. Then $a_n(0)=b_n(1)|f'(0)|$.  We will show that  $\lim_n b_n(1) =0$. Consider  a sequence of rings $A_{R_n}= \{z| R_n^{-1} < |z| < R_n\}$ with $\lim_n R_n =\infty$ such that $ A_{R_n} \cap  f(X^{(n)}) =\emptyset$.  Let the Poincar\'e metric on $A_R =\{ z \in \C| |R^{-1} <|z| < R\}$ be $c_{R}(z)|dz|$.  Then by the Schwarz-Pick lemma $c_{R_n}(z)\geq b_n(z)$.  It is well known  (\cite{beardon}, page 49) that
$$
c_R(1)=\frac{\pi}{2\log R}.
$$
 It follows $|b_n(1)| \leq \frac{\pi  }{2 \log R_n}$ and hence $\lim_n a_n(0)=\lim_n b_n(1) |f'(0)|=0$.

On the other hand,
since $(\Sigma_n,d_{\Sigma_n}^P)$ and $( U_n,d_n)$ are isometric,  by Lemma \ref{8.111},  there exists a hyperbolic ball $B_r(0, d_n)$  of radius $r>0$ centered at $0$ in $U_n$ such that $B_r(0, d_n)$ is isometric to the standard ball $B_r(0, d^P)$ for all $n$. Let $f_n$ be an orientation preserving isometry from the ball $B_r(0, d^P)$  in the Poincar\'e disk  to $B_r(0,d_n)$  in $(U_n, d_n)$.  Since $f_n$ is an isometry, we have $\frac{2|dz|}{1-|z|^2} =a_n(f_n(z))|f_n'(z)|$. This shows $|f_n'(0)|=2/a_n(0)$ and by (\ref{limi}),  $|f_n'(0) |\to \infty$. Recall  Koebe's quarter theorem says that if $g:  B_r(0, d^E) \to \C$ is an injective analytic map, then its image $g(B_r(0, d^E))$ contains the Euclidean ball of radius $\frac{|g'(0)|r}{4}$ centered at $g(0)$.  Applying it to the injective analytic maps $f_n$ defined on $B_r(0, d^P)$, we see that $f_n(B_r(0, d^P))$ contains the Euclidean disk of radius 2 centered at $0$ for $n$ large. This contradicts the assumption  that $\partial \D \cap \partial U_n \neq \emptyset$ and shows
$|X| \geq 3$.

Finally to see the second part of the Lemma, by the Schwarz-Pick lemma applied to $\D \subset U_n$,  we have $\frac{2|dz|}{1-|z|^2} \geq a_n(z)|dz|$ and in particular $a_n(0) \leq 2$.   Therefore, $|f_n(0)| \geq 1$. Then Koebe's quarter Theorem implies that $f_n(B_r(0, d^P))$ contains $B_{r'}(0, d^E)$ for some $r'$ independent of $n$. Since $d^E$ and $d^{\SS }$ are bi-Lipschitz equivalent when restricted to $\D$, the result follows.


\end{proof}

Now we prove Theorem \ref{1.222}(b). The key result used in the proof is the following equicontinuity theorem
to be proved in \S9.
\begin{theorem}\label{equicont1} The sequence $\{\phi_n: (\partial C_P(Y^{(n)}), d^{E}) \to (U_n, d^{\SS })\}$  is equicontinuous.
 \end{theorem}

Assuming the theorem, by equicontinuity, each map $\phi_n: \Sigma_n \to   U_n$ extends to a continuous map, still denoted by $\phi_n: \overline{\Sigma_n} \to \overline{  U_n}$ between their closures in $\R^3$ and $\hat{\C}$. Furthermore, the extended family $\{\phi_n: ( \overline{\Sigma_n}, d^E) \to (\overline{  U_n}, d^{\SS })\}$ is still equicontinuous.  Now each boundary component of $\Sigma_n$ and $  U_n$ is a round circle or a point. Use central extension to extend $\phi_n$ to be a continuous map,  still denoted by $\phi_n$, from $\Sigma_n \cup Y^{(n)}$ to $  U_n \cup X^{(n)} =\hat{\C}$.  By the normalization condition that $(0,0,0) \in \partial C_P(Y^{(n)})$ and 
$X^{(n)} \subset \hat \C -\D$, each component of $Y^{(n)}$ and $X^{(n)}$ has spherical radius at most $\pi/2$ for $n>4$.   Using  Corollary \ref{5.6} and a similar argument for Proposition \ref{512}, the extended family $\{ \phi_n: (\Sigma_n \cup Y^{(n)}, d^E) \to (\hat{\C}, d^{\SS })$\} is equicontinuous. 
Since $Y^{(n)}$ and  $\partial C_P(Y^{(n)})$ and converge to $Y$ and $\partial C_P(Y)$ in Hausdorff metrics respectively, it follows that their union
$\partial C_P(Y^{(n)}) \cup Y^{(n)}$ converges in Hausdorff metric to $\partial C_P(Y) \cup Y$. 
By the generalized Arzela-Ascoli Theorem \ref{aa},   we may assume, after taking a subsequence,  that  $\phi_n$ converges uniformly to a continuous map
 $\phi: \partial C_P(Y) \cup Y \to \hat{\C}$. Since $\phi_n(Y^{(n)})=X^{(n)}$, by Lemma \ref{456},
 $\phi(Y)=X$ and $\phi$ is onto.

 By Lemma \ref{822},  Alexandrov's convergence theorem \ref{convergence} and convergence theorem of Poincar\'e metrics (Theorem \ref{poincare}), we see that $\phi|_{\partial  C_P(Y)}$ is an isometric embedding of $\partial C_P(Y)$ into a component $\Omega'$ of $\hat{\C}-X$.  In particular,  $\phi(\partial C_P(Y)) \subset \hat{\C}-X$. Together with $\phi(Y)=X$ and that $\phi$ is onto, we see that $\phi(\partial C_P(Y)) =\hat{\C}-X$. Therefore, $\phi$ is an isometry from $\partial C_P(Y)$ to $\hat{\C}-X$.  Now we claim that $X$ is a circle type closed set. Indeed, by the same argument as in \S7.4,  each component $X_k$ of $X$ is of the form $\phi(Y_k)$ for some component $Y_k$ of $Y$.  By Lemma \ref{456} and that $Y$ is of circle type, each $\phi(Y_k)$ is the Hausdorff limit of a sequence of components $\phi_n(Y^{(n)}_{k_n}) =X^{(n)}_{k_n}$ of $X^{(n)}$. Therefore the result follows.



\section{Proof of the equicontinuity Theorem \ref{equicont1}}

Recall  Theorem \ref{equicont1} states,

  \begin{theorem*} The family $\{\phi_n: (\partial C_P(Y^{(n)}), d^{E}) \to (U_n, d^{\mathbf  S})\}$  is equicontinuous.
 \end{theorem*}

We prove the above theorem by deriving a contradiction.
Suppose otherwise,  there exist $\epsilon>0$ and two sequences $y_{k_n}, y_{k_n}' \in\Sigma_{k_n}:=\partial C_P(Y^{(k_n)})$ such that $|y_{k_n}-y_{k_n}'| \to 0$ and
$d^{\mathbf  S}(\phi_{k_n}(y_{k_n}),\phi_{k_n}(y_{k_n}'))>\epsilon. $ 
If the sequence $\{k_n\}$ is not bounded,  we may assume,  after taking a subsequence, that $k_n=n$.  If the sequence $\{k_n\}$ is bounded, we may assume after taking a subsequence that $k_n$ is a constant.  Below we will focus on the main case that $k_n=n$. The same proof also works for the simpler case that $k_n$ is a constant. We omit the details. Since the case $k_n$ is equivalent to the uniform continuity of $\phi_n$, we can also prove the uniform continuity of $\phi_n$ by using Caratheodory's extension theorem as discussed in \S7.2. 

Moving forward, we assume that $|y_n-y_n'|\to 0$, 
\begin{equation}\label{956} d^{\mathbf S}(\phi_{n}(y_{n}),\phi_{n}(y_{n}'))>\epsilon,  \end{equation}
and $\{y_n\}$ converges to some point $p \in \partial C_P(Y)\cup Y$. 
We claim that $q \in Y$. To see this, we need the following lemma. 



 \begin{lemma} \label{9.11} Let $d_n$ be the Poincar\'e metric on $U_n$. There exists a constant $C_0>0$  independent of $n$ such that
$$d_n(\phi_n(y_n),\phi_n(y_n'))\geq C_0 d^{\mathbf  S}(\phi_n(y_n),\phi_n(y_n')).$$
\end{lemma}
\begin{proof} 
Since the sequence \{$X^{(n)}$\}  Hausdorff converges to $X$ and $|X| \geq 3$ (Lemma \ref{822}), we can choose a 3-point set $\{w_1, w_2, w_3\} \subset X$. 
 Let $W=\hat{\C}-\{w_1, w_2, w_2\}$ and $d^W =a(z)|dz|$ be the Poincar\'e metric on $W$.  Note that $a(z)$ tends to infinity as $z$ approaches $\partial W$ since each $w_i$ corresponds to a cusp end of $W$. Therefore, there exists a constant $C_W>0$ such that $d^W \geq C_W d^{\mathbf  S}$ on $\hat{\C}$.   Now take three points $u_i^{(n)} \in X^{(n)}$, $i=1,2,3$, such that $\lim_n u_i^{(n)} = w_i$.
 We claim that there exists a constant $C_0>0$ independent of $n$ such that
  \begin{equation}\label{9123} d^{V_n} \geq C_0 d^{\mathbf  S}, \end{equation}
   where $d^{V_n}$ is the Poincar\'e metric on $V_n=\hat{\C} -\{u_1^{(n)}, u_2^{(n)}, u_3^{(n)}\}$.
To see this, let $M_n$ be the Moebius transformation sending $u^{(n)}_i$ to $w_i$ for $i=1,2,3$. Then $M_n$ converges uniformly to the identity map in $(\hat{\C}, d^{\\S}).$ Then there exists a positive $C_0$ such that
$$d^{V_n}(z,w) =d^W(M_n(z), M_n(w)) \geq C_W d^{\SS}(M_n(z), M_n(w)) \geq C_0 d^{\SS}(z, w).$$
 Using the Schwarz-Pick lemma that  $d_n \geq d^{V_n}$, we see  $$d_n(\phi_n(y_n), \phi_n(y_n')) \geq d^{V_n}(\phi_n(y_n), \phi_n(y_n')) \geq C_0 d^{\mathbf  S}(\phi_n(y_n), \phi_n(y_n')).$$
\end{proof}

\begin{remark} We thank the referee for pointing out a gap in our original proof and for providing a new argument in proving (\ref{9123}). \end{remark}

Now we prove that $q \in Y$ by contradiction. If otherwise $q\in  \partial C_P(Y)$,    then $$d_{\Sigma_n}^P(y_n,y_n')=d_n(\phi_n(y_n),\phi_n(y_n'))\geq C_0 d^{\mathbf  S}(\phi_n(y_n),\phi_n(y_n'))>C_0\epsilon.$$ But  $\lim_n y_n=\lim_n y_n' =q \in \Sigma$, by Alexandrov's convergence Theorem \ref{convergence},
$
\lim_{n}d_{\Sigma_n}^P(y_n,y_n')= d_{\Sigma}^P(q,q)=0.
$
This contradicts that $d_{\Sigma_n}^P(y_n, y_n') \geq C_0\epsilon.$

Let $\delta_0\in(0,1)$ be a lower bound of the injectivity radii produced in Lemma \ref{8.111} such that
 $B_n=\{x\in\Sigma_n:d_{\Sigma_n}^P(x,0)\leq \delta_0\}$ is an embedded disk in $\Sigma_n$ for all large $n$.  By Lemma \ref{822} the image $\phi_n(B_n)$ contains a spherical ball $B_{\delta}(0, d^{\mathbf  S})$ for some radius $\delta>0$ independent of $n$. 
 Let $Y_*$ be the component of $Y$ which contains $q$ and let $\tilde \Sigma_n =\Sigma_n \cup Y^{(n)}$ which is a topological 2-sphere.
By the same argument as in \S7, we see that $\phi_n$ induces a conformal equivalence between  $\tilde \Sigma_n ^{Y^{(n)}}$ and $\hat \C^{ X^{(n)}}$ and a bijection between components of $X^{(n)}$ and $Y^{(n)}$. 
Construct two families of paths $\Gamma_n$ and $\Gamma_n'$ as follows.

If $Y_*$ is a single point, then
$ \Gamma_n$ is defined to be the set of simple loops in ${\tilde\Sigma_n}^{Y^{(n)}}$ separating $\{y_n,y_n'\}$ and $B_{n}$,
and
$\Gamma_n' $ is defined to be the set of simple loops in ${\hat{\C}}^{X^{(n)}}$ separating $\{\phi_n(y_n), \phi_n(y_n')\}$ and $\phi_n(B_{n})$.

If $Y_*$ is a round disk, then  by the construction of $Y^{(n)}$, $Y_*$ is a connected component of $Y^{(n)}$ for $n$ sufficiently large. Therefore $q \in \partial Y^{(n)}$ for $n$ large. We define
$ \Gamma_n$  to be the union of two family curves: simple loops $\gamma_1$ and simple arcs $\gamma_2$. Here $\gamma_1$ are simple loops in $\tilde{\Sigma}_n^{Y^{(n)}}-\{[Y_*]\}$ such that
$\gamma_1$ separate $\{y_n, y_n'\}$ and $B_n$ in $\tilde{\Sigma}_n^{Y^{(n)}}$, and $\gamma_2$ are simple arcs in $\tilde{\Sigma}_n^{Y^{(n)}}-\{[Y_*]\}$ with end points in $Y_*$ such that 
$\gamma_1 \cup \{[Y_*]\}$ is a simple loop separating $\{y_n, y_n'\}$ and $B_n$  in $\tilde{\Sigma}_n^{Y^{(n)}}$. We define 
$\Gamma_n'$ to be the image of $\Gamma_n$ under the conformal map $\phi_n$. 
The conformal invariance of extremal length implies that $EL(\Gamma_n) =EL(\Gamma_n')$.  We will derive a contradiction by showing that $\liminf_n EL(\Gamma_n') >0$ and  $\lim EL(\Gamma_n)=0$.

\subsection{Extremal length estimate I: $\liminf_n EL(\Gamma_n') >0$}
Consider the extended metric  $m=(d^{\mathbf  S}|_{U_n},\mu)$ on $\hat{\C}^{X^{(n)}}$ such that $\mu([X^{(n)}_i])=diam_{d^{\mathbf  S}}(X^{(n)}_i)$. Here $X^{(n)}_i$ is a component of $X^{(n)}$.

By definition,
$$
EL(\Gamma_n')\geq\frac{l_m(\Gamma_n')^2}{A(m)}.
$$
Now
$$
A(m)\leq Area_{d^{\mathbf  S}}(U_n)+\sum_{i}diam_{d^{\mathbf  S}}(X^{(n)}_i)^2
$$ $$\leq Area_{d^{\mathbf  S}}(U_n)+\sum_{i}4\pi Area_{d^{\mathbf  S}}(X_i^{(n)})=4\pi Area_{d^{\mathbf  S}}(\hat\C^2)=16\pi^2.
$$
It remains to show $l_m(\Gamma_n') >0$.

For any
$\gamma\in\Gamma_n'$, construct the new loop  $\gamma^* \subset \hat \C$ obtained by gluing to $\gamma \cap (\hat \C-X^{(n)})$ a  spherical geodesic segment of length at most $\mu([X_i^{(n)}])$ inside the $X_i^{(n)}$ for each intersection point  $[X_i^{(n)}]$ of
$[X^{(n)}]\cap \gamma$.
%
%
%
 Then by the construction of extended metric $m$,  $l_m(\gamma) \geq l_{d^{\mathbf  S}}(\gamma^*)$.  Note that $\phi_n(B_n)$ contains a spherical ball $B_{\delta}$ of radius $\delta$ and hence its spherical diameter is bounded below by $2\delta$. Also,   
 $d^{\mathbf  S}(\phi_n(y_n), \phi_n(y_n'))>\epsilon$ by (\ref{956}).

  If $Y_*$ is a single point, then
   $\gamma^*$ separates $\{\phi_n(y_n), \phi_n(y_n')\}$ from the  ball $\phi_n(B_n)$ in $\hat \C$. Hence, there exists $\delta'>0$ independent of $\gamma$ and $n$  such that \begin{equation}\label{delta}
   l_{d^{\mathbf  S}}(\gamma^*) \geq \delta'.\end{equation}  Indeed, consider the two Jordan domains $Q_1$ and $Q_2$ bounded by $\gamma^*$ in $\hat \C$. One of them contains a spherical ball of radius $\delta$, and the other contains two points of spherical distance $\epsilon$ apart. Therefore, the spherical diameters of $Q_1$ and $Q_2$ are at least $\delta''>0$ for some $\delta''$ independent of $\gamma$ and $n$. Now the length of $\gamma^*$ is at least $\min( diam_{d^{\mathbf  S}}(Q_1), diam_{d^{\mathbf  S}}(Q_2), \pi/2)$. Therefore, inequality (\ref{delta}) follows.


If $Y_*$ is a round disk, then for $n$ large,  by the construction, $Y_*$ is a component of $Y^{(n)}$. Let the corresponding component in $X^{(n)}$ be $X_*$. Then
the set $\gamma^* \cup X_*$ separates $\phi_n(B_n)$ and $\{\phi_n(y_n), \phi_n(y_n')\}$ in $\hat{\C}$.  
If $\gamma^*$ is a simple loop, $l_m(\gamma^*)\geq\delta'$ as above.
If $\gamma^*$ is a simple arc with endpoints $q', q''$ in $X_*$, 
let $\beta$ be the shortest spherical geodesic from $q'$ to $q''$.  Since $X_* \cap \D =\emptyset$, $X_*$ is convex in $(\hat{\C}, d^S)$. We see that $\beta$ is contained in $X_*$.  The loop
$\gamma^* \beta$ formed by $\gamma^*$ followed by $\beta$ separates $ \{\phi_n(y_n),\phi_n(y_n') \}$ and $\phi_n(B_n)$ in $\hat{\C}$. Therefore, $l_m(\gamma^*\beta) \geq \delta'$ as above. Since $\beta$ is a shortest spherical geodesic, $l_m(\gamma^*) \geq l_m(\beta)$. It follows that  $$l_m(\gamma^*) \geq  \frac{1}{2}(l_m(\gamma^*) +l_m(\beta)) =\frac{1}{2}
 l_m(\gamma^* \beta) \geq \frac{\delta'}{2}.$$

\subsection{Extremal length estimate II:  $\lim EL(\Gamma_n)=0$  }
The methods of proof are similar to those in \S 7.2.1. 
Recall that $q =\lim_n y_n \in Y_*$, a component of $Y$.
Choose a point $q_n$ in $Y^{(n)}$ such that $\lim_{n} q_n=q$. Clearly, $\lim_n |y_n-q_n| =0$. 
Denote 
$D_r=B_r(q, d^{\mathbf S})$,  $E_r={\mathbf  S^2-D_r}$,  
$D_r^n=B_r(q_n, d^{\mathbf S})$,  and $E_r^n={\mathbf  S^2-D_r^n}$ as subsets of $\mathbf S^2$.
Since $\lim_nq_n=q$, $\bar D^n_r$ and $E^n_r$ converge in Hausdorff distance to $\bar D_r$ and $E_r$ respectively.
The loops and arcs in $\Gamma_n$ separate $B_n=\{x\in\Sigma_n:d^P_{\Sigma_n}(x,0)\leq\delta_0\}$ from $\{y_n, y_n'\}$
where $\delta_0\in(0,1)$. 
The hyperbolic disk $B_n$ is contained in the Euclidean ball $\{x \in \mathbf R^3| |x| < \delta_0/2\}$ since $d^P_{\Sigma} \geq d^P \geq 2d^E$. 
Recall that $\pi$ denotes the shortest distance projection from $\partial \H^3$ to $\tilde\Sigma_n.$ 
By taking $r_0<1/260$ small and using Proposition \ref{3.7}, 
we see that $B_n \subset \pi(E^n_{65r_0})$ for all $n$. 


\begin{lemma}\label{9.3} 
For any $r>0$, there exists $r'<r/65$ such that no component of $Y-Y_*$ intersects both $\partial D_r$ and $\partial D_{r'}$.
\end{lemma}

\begin{proof} By a simple spherical area estimate,  we see that there exist only finitely many components $Z_1,...,Z_m$ of $Y-Y_*$ such that $Z_i$ intersects both $\partial D_r$ and $\partial D_{r/65}$. Let $r'$ be a positive number with $r'<r/65$ and $r'<\min_{1\leq i\leq m}d^{\mathbf  S}(Y_*,Z_i)$. The result follows.
\end{proof}

By Lemma \ref{9.3} and $q =\lim_n y_n=\lim_n q_n$, we construct a sequence of positive numbers $\{r_i\}$ and a sequence of integers $\{N_i\}$ increasing to infinity such that

(1) $r_{j+1} < r_j/65$ for all $j$, 

(2) each component of $Y-Y_*$ intersects at most one of  $\overline{D_{65r_j}}-D_{r_j}$ for all $j$, 


(3) if $n \geq N_i$, then $\{y_n, y_n'\} \subset \pi(D^n_{r_i})$.

Note that since $Y^{(n)} \subset Y$, we see that each component of $Y^{(n)}-Y_*$ intersects at most one of $\overline{D_{65r_j}}-D_{r_j}$ for all $j$. 

\medskip
We now prove $\lim_n EL(\Gamma_n)=0$ by showing that
$EL(\Gamma_n)\leq3\cdot 10^8/i$ for all $n\geq N_i$. We will use Proposition \ref{6571}
by constructing a family of annuli $S_n$ as follows. 
For simplicity, we will use the following notation. If $Z$ is a closed subset of $\mathbf S^2$, then $Z^{*}$ denotes the union of $Z$ and all components $Y_k$ of $Y-Y_*$ which intersect $Z$, i.e.,
$$ Z^* = Z \cup(\cup_{Y_k \cap Z \neq \emptyset, Y_k \neq Y_*} Y_k).$$
Note that if $Z$ is connected, then so is $Z^{*}$. 
Now fix $i\geq1$ and $n\geq N_i$. Then by Proposition \ref{3.7} and condition (2) above, $\pi(\bar D_{r_i}^{n*}) \cap \pi( E_{65r_0}^{n*}) =\emptyset$. It follows that $\pi(\bar D_{r_i}^{n*})$ and $ \pi( E_{65r_0}^{n*})$ are disjoint connected compact sets in $\tilde{\Sigma}_n$. 
Hence there is a unique component of $\tilde{\Sigma}_n -(\pi(\bar D_{r_i}^{n*}) \cup \pi( E_{65r_0}^{n*})) $ which is an annulus. We denote this annulus by $S_n$. By conditino (3),  the annulus $S_n$ separates $\{y_n, y_n'\}$ and $B_n$ in $\tilde{\Sigma}_n$. 
See Figure \ref{922}.
Let $W'=Y^{(n)}\cap S_n-Y_*$, 
$\Omega_{r_j}=\{z\in\mathbf S^2|r_j<d^{\mathbf S}(z,q_n)<65r_j\}$ for $j=0,1,..., i-1$,
and
$\Gamma_n^*$ be the set of all simple loops in $S_n^{W'}$ separating the two ends of $S_n$. 
Applying Proposition \ref{6571} to $S=S_n$, $W=Y^{(n)}$, and $\Omega_{r_j}$ for $j=0,1,..., i$,
we obtain, 
$$
EL(\Gamma_n^*) \leq \frac{3}{i} 10^8.
$$
\begin{figure}[ht!]
\begin{center}
\begin{tabular}{c}
\includegraphics[width=0.6\textwidth]{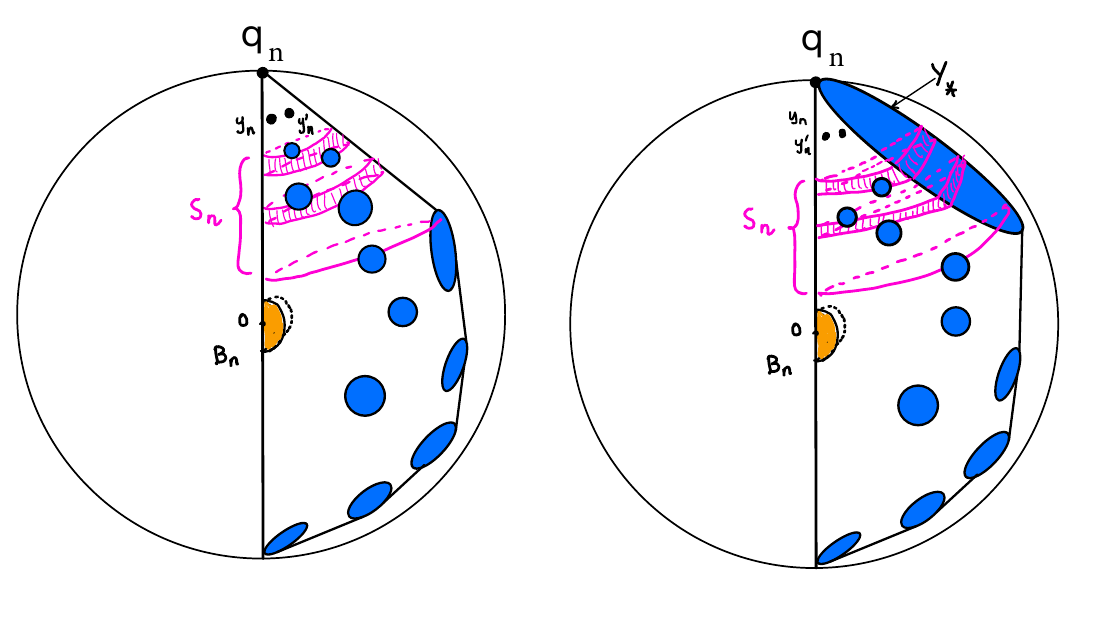}
\end{tabular}
\end{center}
\caption{$S_n$ in the case of $Y_*$ is a point or a disk.} \label{922}
\end{figure}

The rest of the proof is similar to that of \S7.2.1.  If $Y_*$ is a one-point set, then by the construction $\Gamma_n^* \subset \Gamma_n$. Therefore, the monotonicity of the extremal length implies
$$ EL(\Gamma_n) \leq EL(\Gamma_n^*) \leq \frac{3}{i} 10^8.$$

The same proof used in \S7.2.1 shows that $\lim_n EL(\Gamma_n)=0$ if $Y_*$ is a disk. 
\section{Application to discrete conformal geometry of polyhedral surfaces}


A polyhedral surface is a triple $(S, V, d)$ where $S$ is connected surface, $V \subset S$ is a discrete subset, and $d$ is a flat cone metric on $S$ with cone points contained in $V$.    We call $d$ a PL or polyhedral metric on $(S,V)$.  The discrete curvature $K$ of the polyhedral surface is a function defined on the vertices, and $k(v)$ is $2\pi$ less the cone angle at $v$.  Usually, these metrics are obtained by isometric gluing of Euclidean triangles along pairs of edges by isometries. Thus, a polyhedral surface $(S, V, d)$ can be represented by a triangulated PL surface $(S, \T, l)$ where $\T$ is a triangulation with vertex set $V$ and $l: E(\T) \to \R_{>0}$ is the edge length function, i.e., $l(e)$ is the length of the edge $e$. Here $E(\T)$ is the set of all edges in $\T$.     One of the goals of the discrete conformal geometry is to define discrete conformal equivalence among PL metrics on $(S, V)$ and establish the corresponding discrete uniformization theorem. In our recent work \cite{glsw}, we are able to introduce a discrete conformality for polyhedral metrics and establish a discrete uniformization theorem for compact surfaces. We will briefly recall the related results and their relationship to the Weyl problem and discuss one application of the main result in discrete conformal geometry.

A basic tool in computational geometry is the Delaunay triangulation.  In the 2-dimensional case, a triangulated PL surface $(S, \T, l)$ is called \it Delaunay  \rm if the circumdisk of each triangle contains no vertices in its interior.
This is equivalent to the condition that if $e$ is an edge adjacent to two triangles $t$ and $t'$, then the sum of the  two angles in $t$ and $t'$ which are opposite to $e$ is at most
$\pi$. A fundamental theorem in computational geometry says that for any closed PL surface $(S, V, d)$, there is always a Delaunay triangulation $\T$ of $(S, V, d)$ with vertex set equal to $V$. In general,  Delaunay triangulations of $(S,V,d)$ may not be unique. 
Discrete conformal geometry tries to define discrete conformal equivalence between to polyhedral surfaces $(S, V, d)$ and $(S, V, d')$ such that (i) discrete conformal equivalence is computable, (2) discrete conformal maps converge to smooth conformal maps as meshes tend to zero, and (3) there exists a discrete uniformization theorem within each discrete conformal class. 




\begin{figure}[ht!]
\begin{center}
\begin{tabular}{c}
\includegraphics[width=0.7\textwidth]{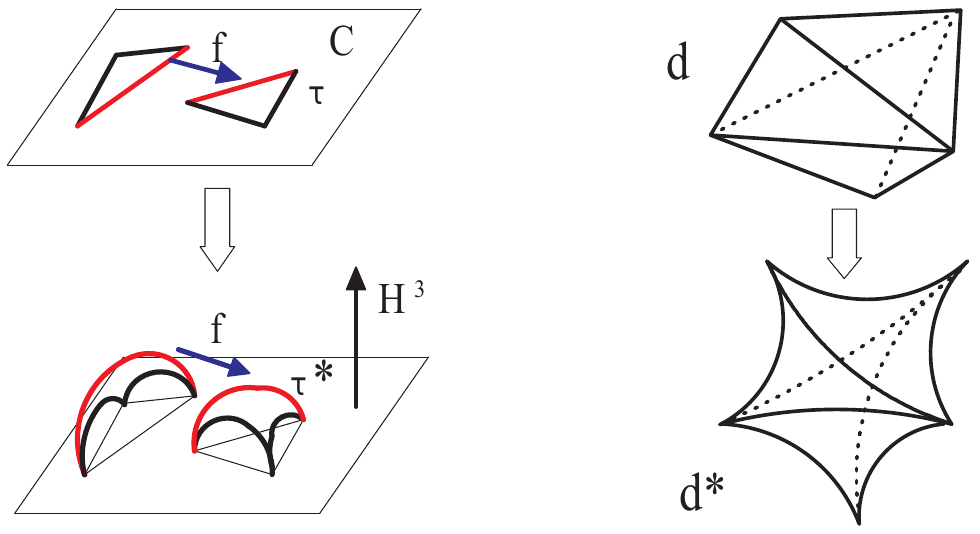}
\end{tabular}
\end{center}
\caption{Discrete conformality in terms of hyperbolic geometry}
\label{figure_1}
\end{figure}





 Given
a PL metric $d$ on $(S,V)$, construct a Delaunay triangulation $\T$ 
of $(S,V,d)$.  For each Euclidean triangle $\tau =\Delta ABC$ in
$\T$ (considered as a triangle in $\C$), replace $\tau$ by the ideal hyperbolic triangle  $\tau^*$ in
the upper-half-space model $\C \times \R_{>0}$ of the hyperbolic 3-space such that $\tau^*$ and $\tau$ have
the same set of vertices $\{A,B,C\}$ in $\C$. 
If $\tau$ and $\sigma$ are two
Euclidean triangles in $\T$ glued along a pair of edges by a
Euclidean isometry $f$, 
then one glues $\tau^*$ and $\sigma^*$ along their corresponding
edges by the \it same \rm isometry $f$, considered as a rigid
motion of $\H^3$.  In this way, one produces a complete finite
area hyperbolic metric $d^*$ on $S-V$. From the definition of
Delaunay triangulation, one sees that $d^*$ is independent of the
choices of the Delaunay triangulation $\T$.

\begin{definition}(Discrete conformalitly of PL metrics) \cite{glsw} Two PL metrics $d_1$ and $d_2$ on a marked surface $(S, V)$ are \it discrete conformal \rm if there exists an isometry $\phi: (S-V, d_1^*) \to (S-V, d_2^*)$ such that $\phi$ is homotopic to the identity map relative to $V$. 
\end{definition}

The main theorem proved in \cite{glsw} is,

\begin{theorem} \label{main1} Suppose $(S, V)$ is a closed connected marked surface
and  $d$ is a PL metric on $(S, V)$.  Then for any $K^*:V \to
(-\infty, 2\pi)$ with $\sum_{v \in V} K^*(v) =2\pi \chi(S)$, there
exists a PL metric $d'$, unique up to scaling and isometry
homotopic to the identity on $(S, V)$, such that $d'$ is discrete
conformal to $d$ and the discrete curvature of $d'$ is $K^*$.

\end{theorem}

For the constant function $K^*=2\pi\chi(S)/|V|$ in theorem
\ref{main1}, we obtain a constant curvature PL metric $d'$, unique
up to scaling, discrete conformal to $d$. This is a discrete
version of the uniformization theorem. 

Theorem \ref{main1} takes care of compact polyhedral surfaces.  For non-compact simply connected polyhedral surface $(S, V, d)$, the discrete uniformization problem asks if it is discrete conformal to the following two types of surfaces:  $(\C, V', d_{st})$ or $(\mathbb D, V', d_{st})$. Here $d_{st}$ is the standard Euclidean metric on $\C$ and $V'$ is a discrete set in $\C$ or $\mathbb D$.  It is easy to see that if $V'$ is the set of vertices of a Delaunay triangulation in $\C$ or $\D$, then the associated hyperbolic metric to $(\C, V', d_{st})$ and $(\mathbb D, V', d_{st})$ are exactly the boundary of the convex hulls $\partial C_P(V')$ and $\partial C_P(V' \cup \mathbb D^c)$. It is easy to see that the hyperbolic metric associated with $(S, V, d)$ is a complete hyperbolic metric with cusp ends at points in $V'$.  Thus the discrete uniformization problem for non-compact surfaces proposed in \cite{lsw} is the following,

\medskip
\noindent
{\bf Conjecture 4.}  Suppose $(\Sigma, d)$ is a complete hyperbolic surface with countably many ends, and all but at most one are the cusp ends. Then there exists, unique up to M\"obius transformations, a circle type closed set $X$ such that $(\Sigma, d)$ is isometric to $\partial C_P(X)$. 

\medskip
This conjecture is our original motivation for Theorem \ref{1.222}.  Furthermore, Theorem \ref{1.22} implies the existence part of Conjecture 4 holds.   The uniqueness part remains open.

\section{Appendix: Conformal structure on non-smooth surfaces of bounded curvature}



In this appendix, using the work of Reshetnyak \cite{res} and \cite{troyanov},  we will justify the computations of extremal lengths of curve families in conformal structures on non-smooth  surfaces like $\partial C_H(Y)$.
We begin with a brief  recall of Alexandrov's theory of surfaces of bounded curvature and then discuss the associated conformal structure.


\subsection{Surfaces of bounded curvature}

Suppose $(S,d)$ is a surface with a path metric $d$. This means that the distance $d(x,y)$ between two points is equal to the infimum of lengths of all paths from $x$ to $y$ measured in $d$.  Geodesics in $(S,d)$ are locally distance minimizing curves. Suppose $\beta$ and $\gamma$ are two geodesics from a point $O=\beta(0)=\gamma(0)$ parameterized by arc lengths. Then the (upper) \it angle \rm $\alpha$ between them at $O$ is defined to be
$$ \alpha =\limsup_{s \to 0, t\to 0} \arccos(\frac{ d(0, \beta(s))^2+d(0, \gamma(t))^2 -d(\beta(s), \gamma(t))^2}{2 d(0, \beta(s)) d(0, \gamma(t))}).$$
For a geodesic triangle $T$, the (upper) \it excess \rm of the $T$ is $\delta(T) =\theta_1+\theta_2+\theta_3 -\pi$ where $\theta_i$ are the angles of the triangle
$T$ at three vertices.  The surface $(S,d)$ is called of  \it bounded curvature \rm if for every point $p \in S$, there exists a neighborhood $U$ of $p$ and a constant $M(U) < \infty$ such that for any collection of pairwise disjoint geodesic triangles $T_1, T_2, ..., T_n$ in $U$, we have
$$ \sum_{i=1}^n \delta(T_i) \leq M(U).$$

All smooth Riemannian surfaces,  polyhedral surfaces, convex surfaces in the Euclidean and hyperbolic spaces are surfaces of bounded curvature. For a smooth Riemannian surface $(S,g)$ with Gaussian curvature $K$, the area form $dA_g$, the curvature form $KdA_{g}$ and the total geodesic curvature of a smooth curve are well defined. In a surface of bounded curvature,  Alexandrov redefined these notations using the distance $d$ and they become the area measure, curvature measure $w$ and total geodesic curvature.
We say a sequence of metrics $d_n$ on a space $X$ converges uniformly to a metric $d$ on $X$ if $d_n(x,y) \to d(x,y)$ uniformly on $X \times X$.  One of the key theorem on surfaces of bounded curvature is the following approximation theorem. See Theorem 6.2.1 in \cite{res}

\begin{theorem}[Alexandrov-Zalgaller] Suppose $(S,d)$ is a surface of bounded curvature and $p \in S$. Then there exist a disk neighborhood $U$ of $p$, a constant $C(U)>0$ and a sequence of polyhedral metrics $d_n$ on $U$ such that $d_n$ converges uniformly to $(U, d|_U)$ and
\begin{equation}\label{app5}  |w_{d_n}|(U) + |\kappa_{d_n}|(\partial U) \leq C, \end{equation}
where $w_{d_n}$ is the curvature measure and $\kappa_{d_n}$ is the total geodesic curvature of a path. Conversely, if $(S, d_n)$ is a sequence of polyhedral surfaces converging uniformly to a path metric $(S, d)$ such that (\ref{app5}) holds on $S$ for some constant $C$, then $(S, d)$ is a surface of bounded curvature.
\end{theorem}




\subsection{Conformal structures on surfaces of bounded curvature}

The construction of conformal charts for surfaces of bounded curvature by Reshetnyak goes as follows. Suppose $U$ is an open disk in the plane and $w$ is a signed  Borel measure on $U$. 
Then the function $\ln \lambda(z) =\frac{1}{\pi} \int_U \frac{1}{|z-\zeta|} w(d\zeta) +h(z)$ is the difference of two subharmonic functions on $U$ where $h$ is a harmonic function.  Since the Hausdorff dimension of the set of points where $\ln \lambda(z)=-\infty$ is zero,  for an arbitrary Euclidean rectifiable path $L$ in $U$, the integral $\int_L \sqrt{ \lambda(z(s))} ds$ is well defined (could be $\infty$). 
  One defines the distance $d_U$ on $U$ between two points to be the infimum of the  lengths of paths between them in $\lambda(z) |dz|^2$.  There may be some points whose $d_U$-distance  to any other point is infinite. These are called points at infinity and they form a discrete subset in $U$. Let $\tilde U$ be the complement of the set of points at infinity.
 It is proved by Reshetnyak that  $(\tilde U, d_U)$ is a surface of bounded curvature whose curvature measure is $w$.  The main theorem of Reshetnyak's conformal geometry of surfaces of bounded curvature is Theorem 7.1.2 in \cite{res}.

\begin{theorem}[Reshetynyak] \label{res2} Let $(S, d)$ be a surface of bounded curvature. Then for any point $p \in S$, there exists a neighborhood $U'$ of $p$ and an open disk $U$ in the plane together with a Borel measure $w$ such that $(U', d|_{U'})$ is isometric to $(\tilde U, d_U)$.
\end{theorem}

Let $\phi:  (U', d|_{U'}) \to  (\tilde U, d_U)$ be  an orientation preserving isometry produced in the above theorem. Then $\{(U', \phi)\}$ forms the analytic charts on the surface $(S, d)$.

In conclusion, the metrics in surfaces with bounded curvature can be treated as Riemannian distance derived from Riemannian metrics by relaxing the smoothness condition.

Therefore, in a surface of bounded curvature $(S,d)$ whose area measure is $m$ and  the underlying conformal structure is $\mathcal C$,  we can use the path metric $d$ and area measure $m$ to compute the extremal length of a curve family $\Gamma$. In particular, we have
the estimate
$$ EL(\Gamma, S, \mathcal C) \geq \frac{ l_{d}^2(\Gamma)}{ m(S)}.$$
This estimate has been used extensively in  previous sections on $(\partial C_H(Y), d^P_{\partial C_H(Y)})$.  

Finally, for a compact set $Y \subset \mathbf S^2$ and $\Sigma =\partial C_H(Y)$,  we claim that the two induced path metrics $d^P_{\Sigma}$ and $d^E_{\Sigma}$ on
$\Sigma$ produce the same complex structure.  In particular, this shows for any curve family $\Gamma$ in $\Sigma$, $EL(\Gamma, \Sigma, d^P_{\Sigma}) =EL(\Gamma, \Sigma, d^E_{\Sigma})$.
To see the claim, following Alexandrov \cite{al}, one constructs a sequence of convex hyperbolic polyhedral surfaces converging uniformly on compact sets to $(\partial C_H(Y), d^P_{\partial C_H(Y)})$. The convexity implies that (\ref{app5}) holds. Now on polyhedral surfaces, the conformal structures induced from $d^P$ and $d^E$ are the same since these two metrics are conformal in $\H^3$.  The work of Reshetnyak (\cite{res}, Theorems 7.3.1,  p112)  shows that if a sequence of bounded curvature surfaces converge uniformly to a bounded curvature surface such that (\ref{app5}) holds, then the  isothermal coordinates (with appropriate normalization) converge to the isothermal coordinate of  the limit surface.  Therefore the conformal structures on $\partial C_H(Y)$ are the same.

\end{document}